\documentclass[jair,twoside,11pt]{article}
\pdfoutput=1
\usepackage{jair,theapa,rawfonts}
\usepackage{amsmath,amssymb,amsthm}
\usepackage{enumerate}
\usepackage{graphicx}
\usepackage[table]{xcolor} 
\usepackage{multirow}
\usepackage{rotating} 
\usepackage{url}
\usepackage{algorithm,algpseudocode}
\algrenewcommand\algorithmicdo{}
\algrenewcommand\algorithmicindent{.7em}%
\usepackage{caption}
\usepackage{subcaption}
\usepackage{etoolbox} 
\usepackage{pinlabel}

\usepackage{tabularx}

\jairheading{56}{2016}{153-195}{03/15}{06/16}
\ShortHeadings{Global Continuous Optimization with Error Bound and
Fast Convergence}{Kawaguchi, Maruyama, \& Zheng}

\firstpageno{153}

\DeclareRobustCommand{\us}[2]{#1}

\newcommand*{\gra}{\cellcolor{gray!50}}

\newcommand*{\OO}{\tilde{O}}
\newcommand*{\E}{\mathrm{E}}
\newcommand*{\W}{\Omega}
\newcommand*{\R}{\mathbb{R}}
\newcommand*{\Z}{\mathbb{Z}}
\newcommand*{\sem}{\ell}
\newcommand*{\ops}{\epsilon}
\newcommand*{\val}{\mathit{val}}
\newcommand*{\valmax}{\mathit{val}_\mathit{max}}
\newcommand*{\hmax}{h_\mathit{max}}
\newcommand*{\Rmax}{R_\mathit{max}}
\newcommand*{\gb}{\gamma}

\newcommand{\Select}{\underline{Select}}
\newcommand{\Evaluate}{\underline{Evaluate}}
\newcommand{\Group}{\underline{\smash{Group}}}
\newcommand{\Divide}{\underline{Divide}}

\newcommand*{\abs}[1]{\lvert #1 \rvert}
\newcommand*{\norm}[1]{\lVert #1 \rVert}
\newcommand*{\floor}[1]{\lfloor #1 \rfloor}
\newcommand*{\bbfloor}[1]{\biggl\lfloor #1 \biggr\rfloor}

\newtheorem{theorem}{Theorem}
\newtheorem{assumption}{Assumption}
\newtheorem{corollary}{Corollary}
\newtheorem{lemma}{Lemma}
\newtheorem{assA}{Assumption}

\newtheorem{assB}{Assumption}

\theoremstyle{definition}
\newtheorem{definition}{Definition}



\begin{document}

\title{Global Continuous Optimization with Error Bound \\ and Fast
Convergence}

\author{\name Kenji Kawaguchi \email kawaguch@mit.edu\\
\addr Massachusetts Institute of Technology\\
Cambridge, MA, USA
\AND
\name Yu Maruyama \email maruyama.yu@jaea.go.jp\\
\addr Nuclear Safety Research Center\\ Japan Atomic Energy Agency\\
Tokai, Japan
\AND
\name Xiaoyu Zheng \email zheng.xiaoyu@jaea.go.jp\\
\addr Nuclear Safety Research Center\\  Japan Atomic Energy Agency\\
Tokai, Japan}


\maketitle

\begin{abstract}
This paper considers global optimization with a black-box unknown
objective function that can be non-convex and non-differentiable. Such a
difficult optimization problem arises in many real-world applications,
such as parameter tuning in machine learning, engineering design problem,
and planning with a complex physics simulator. This paper proposes a
new global optimization algorithm, called \emph{Locally Oriented Global
Optimization} (LOGO), to aim for both fast convergence in practice and
finite-time error bound in theory. The advantage and usage of the new
algorithm are illustrated via theoretical analysis and an experiment
conducted with $11$ benchmark test functions. Further, we modify the
LOGO algorithm to specifically solve a planning problem  via policy search with continuous
state/action space and long time horizon while maintaining its finite-time
error bound. We apply the proposed planning method to accident management
of a nuclear power plant. The result of the application study demonstrates
the practical utility of our method.
\end{abstract}

\section{Introduction}\label{sec1}

Optimization problems are prevalent and have held great importance
throughout history in engineering applications and scientific
endeavors. For instance, many problems in the field of artificial
intelligence (AI) can be viewed as optimization problems. Accordingly,
generic \emph{local} optimization methods, such as hill climbing and the
gradient method, have been successfully adopted to solve AI problems since
early research on the topic \cite{kir70,gul94,dei11}. On the other hand,
the application of \emph{global} optimization to AI problems has been
studied much less despite its  practical importance. This
is mainly due to the lack of necessary computational power in the past
and the absence of a practical global optimization method with a strong
theoretical basis. Of these two obstacles, the former is becoming less serious today, as evidenced by a number of studies on global
optimization in the past two decades \cite{hor90,ryo96,he04,rio13}. The
aim of this paper is to partially address the latter obstacle.

The inherent difficulty of the global optimization problem has led
to two distinct research directions: development of heuristics
without theoretically guaranteed performance and advancement of
theoretically supported methods regardless of its difficulty. A degree
of practical success has resulted from heuristic approaches such as simulated annealing, genetic algorithms
(for a brief introduction on the context of AI, see \citeR{rus09}), and
swarm-based optimization (for an interesting example
of a recent study, see \citeR{dal09}). Although these methods are heuristics without
strong theoretical supports, they became very popular partly because
their optimization mechanisms aesthetically mimic nature's physical or
biological optimization mechanism.

On the other hand, the Lipschitzian approach to global optimization
aims to accomplish the global optimization task in a theoretically
supported manner. Despite its early successes in theoretical viewpoints
\cite{shu72,mla86,pin86,han91}, the early studies were based on an
assumption that is impractical in most applications: the Lipschitz
constant, which is the bound of the slope on the objective function,
is known. The relaxation of this crucial assumption resulted in the
well-known DIRECT algorithm \cite{jon93} that has worked well in practice,
yet guarantees only consistency property. Recently, the Simultaneous
Optimistic Optimization (SOO) algorithm \cite{mun11}  achieved
the guarantee of a finite-time error bound without
knowledge of the Lipschitz constant.
However, the practical performance of the algorithm is unclear.

In this paper, we propose a generic global optimization algorithm that
is aimed to achieve both a satisfactory performance in practice and a
finite-loss bound as the theoretical basis without strong additional
assumption%
\footnote{In this paper, we use the term ``strong additional assumption''
to indicate the assumption that the tight Lipschitz constant is known
and/or the main assumption of many Bayesian optimization methods that
the objective function is a sample from Gaussian process with some known
kernel and hyperparameters.}
(Section~\ref{sec2}), and apply it to an AI planning problem
(Section~\ref{sec3}). For AI planning problem, we aim at solving
 real-world engineering problem with a long planning
horizon and with continuous state/action space. As an illustration of
the advantage of our method, we present the preliminary results of an
application study conducted on accident management of a nuclear power
plant as well.
Note that the optimization problems discussed in this paper are practically relevant  yet inherently  difficult to scale up for higher dimensions, i.e., NP-complete \cite{murty1987some}. Accordingly, we discuss  possible extensions of our algorithm for higher dimensional problems, with an experimental illustration with a 1000-dimensional problem. 

\section{Global Optimization on Black-Box Function}\label{sec2}

The goal of global optimization is to solve the following very general
problem:
\begin{gather*}
{\max}_x f(x)\\
\text{subject to }x\in \W
\end{gather*}
where $f$ is the objective function defined on the domain $\W\subseteq
\R^D$. Since the performance of our proposed algorithm is independent of
the scale of $\W$, we consider the problem with the rescaled domain $\W' =
[0, 1]^D$. Further, in this paper, we focus on a deterministic function~%
$f$. For global optimization, the performance of an algorithm can be
assessed by the loss $r_n$, which is given by
\[
r_n=\max_{x\in\W'} f(x) - f(x^+(n)).
\]
Here, $x^+(n)$ is the best input vector found by the algorithm after
$n$ trials (more precisely, we define $n$ to denote the total number of
divisions in the next section).

A minimal assumption that allows us to solve this problem is that
the objective function~$f$ can be evaluated at all points of $\W'$
in an arbitrary order. In most applications, this assumption is easily
satisfied, for example, by having a simulator of the world dynamics or an
experimental procedure that defines $f$ itself. In the former case, $x$
corresponds to the input vector for a simulator $f$, and only the ability
to arbitrarily change the input and run the simulator satisfies the
assumption. A possible additional assumption is that the gradient of function
$f$ can be evaluated. Although this assumption may produce some
effective methods, it  limits the applicability in terms of
real-world applications. Therefore, we assume the existence of a simulator
or a method to evaluate $f$, but not an access to the gradient of $f$. The
methods in this scope are often said to be \emph{derivative-free} and
the objective function is said to be a \emph{black-box} function.

However, if no further assumption is made, this very general problem
is proven to be intractable. More specifically, any number of function
evaluations cannot guarantee getting close to the optimal (maximum)
value of $f$ \cite{dix78}. This is because the solution may exist in an
arbitrary high and narrow peak, which makes it impossible to relate the
optimal solution to the evaluations of $f$ at any other points.

One of the simplest additional assumptions to restore the tractability would be
that the slope of $f$ is bounded. The form of this assumption studied
the most is Lipschitz continuity for~$f$:
\begin{equation}\label{eq1}
\abs{f(x_1)-f(x_2)}\le b\norm{x_1-x_2},\quad
\forall x_1,x_2 \in \W',
\end{equation}
where $b > 0$ is a constant, called the \emph{Lipschitz constant}, and
$\norm{\,\cdot\,}$ denotes the Euclidean norm. The global optimization with
this assumption is referred to as \emph{Lipschitz optimization}, and has
been studied for a long time. The best-known algorithm in the early days
of its history was the Shubert algorithm \cite{shu72}, or equivalently the
Piyavskii algorithm \cite{piy67} as the same algorithm was independently
developed. Based on the assumption that the Lipschitz constant is known,
it creates an upper bound function over the objective function and
then chooses a point of $\W'$ that has the highest upper bound at each
iteration. For problems with higher dimension $D\ge 2$, finding the point
with the highest upper bound becomes difficult and many algorithms have
been proposed to tackle the problem \cite{may84,mla86}. These algorithms
successfully provided finite-loss bounds.

However appealing from a theoretical point of view, a practical
concern was soon raised regarding the assumption that the Lipschitz
constant is known. In many applications, such as with a complex physics
simulator as an objective function $f$, the Lipschitz constant is indeed
unknown. Some researchers aimed to relax this somewhat impractical
assumption by proposing procedures to estimate the Lipschitz constant
during the optimization process \cite{str73,kva03}. Similarly, the Bayesian optimization method with upper
confidence bounds \cite{bro09} \textit{estimates} the objective function and its upper confidence bounds with a certain model assumption, avoiding the prior knowledge on the Lipschitz constant. Unfortunately,
this approach, including the Bayesian optimization method, results in
mere heuristics unless the several additional assumptions hold. The
most notable of these assumptions are that an algorithm can maintain the
overestimate of the upper bound and that finding the point with
the highest upper bound can be done  in a timely manner. As noted
by \citeA{han95}, it is unclear if this approach provides any advantage,
considering that other successful heuristics are already available. This
argument still applies to this day to relatively recent algorithms such
as those by \citeA{kva03} and \citeA{bub11}.

Instead of trying to estimate the unknown Lipschitz constant, the
well-known DIRECT algorithm \cite{jon93} deals with the unknowns by
simultaneously considering all the possible Lipschitz constants, $b$:
$0 < b < \infty$. Over the past decade, there have been many successful
applications of the DIRECT algorithm, even in large-scale engineering
problems (\citeR{car01,he04}; \citeR{zwo05}). Although it works well in many practical
problems, the DIRECT algorithm only guarantees consistency property,
$\lim_{n\to\infty}r_n=0$ \cite{jon93,mun13}.

The SOO algorithm \cite{mun11} expands the DIRECT algorithm and solves
its major issues, including its weak theoretical basis. That is, the SOO
algorithm not only guarantees the finite-time loss bound without knowledge
of the slope's bound, but also employs a weaker assumption. In
contrast to the Lipschitz continuity assumption used by the DIRECT
algorithm (Equation~\eqref{eq1}), the SOO algorithm only requires the
local smoothness assumption described below.

\begin{assumption}[Local smoothness]\label{ass1}
The decreasing rate of the objective function $f$ around at least
one global optimal solution $\{x^*\in\W' : f(x^*)=\sup_{x\in\Omega'} f(x)\}$
is bounded by a~semi-metric $\sem$, for any $x\in\W'$ as
\[
f(x^*)-f(x)\le\sem(x,x^*).
\]
\end{assumption}

Here, semi-metric is a generalization of metric in that it does not
have to satisfy the triangle inequality. For instance, $\sem(x,x^*) =
b\norm{x^*-x}$ is a metric and a semi-metric. On the other hand, whenever
$\alpha > 1$ or $p < 1$, $\sem(x,x^*) = b\norm{x^*-x}^\alpha_p$ is not
a metric but only a semi-metric since it does not satisfy the triangle
inequality. This assumption is much weaker than the assumption described
by Equation~\eqref{eq1} for two reasons. First, Assumption~\ref{ass1}
requires smoothness (or continuity) only at the global optima,
while Equation~\eqref{eq1} does so for any points in
the whole input domain, $\W'$.  Second, while Lipschitz continuity assumption in
Equation~\eqref{eq1} requires the smoothness to be defined by a metric,
Assumption~\ref{ass1} allows a semi-metric to be used. To the best of
our knowledge, the SOO algorithm is the only algorithm that provides a
finite-loss bound with this very weak assumption.

Summarizing the above, while the DIRECT algorithm has been successful
in practice, concern about its weak theoretical basis led to the recent
development of its generalized version, the SOO algorithm. We further
generalize the SOO algorithm to increase the practicality and strengthen
the theoretical basis at the same time. This paper adopts a very weak
assumption, Assumption~\ref{ass1}, to maintain the generality and the
wide applicability.

\section{Locally Oriented Global Optimization (LOGO)
Algorithm}
In this section, we modify the SOO algorithm \cite{mun11} to accelerate
the convergence while guaranteeing theoretical loss bounds. The new
algorithm with this modification, the LOGO (Locally Oriented Global
Optimization) algorithm, requires no additional assumption. To use the
LOGO algorithm, one needs no prior knowledge of the objective function~%
$f$; it may leverage prior knowledge if it is available. The algorithm
uses two parameters, $\hmax (n)$ and $w$, as inputs where $\hmax
(n)\in[1,\infty)$ and $w\in\Z^+$. $\hmax (n)$ and~$w$ act in part to
balance the local and global search. $\hmax (n)$ biases the search towards
a global search whereas $w$ orients the search toward the local area.

The case with $w = 1$ or $4$ (top or bottom diagrams) in Figure~\ref{fig1}
illustrates the functionality of the LOGO algorithm in a simple
$2$-dimensional objective function. In this view, the LOGO algorithm is a
generalization of the SOO algorithm with the local orientation parameter
$w$ in that SOO is a special case of LOGO with a fixed parameter
$w = 1$.

\subsection{Predecessor: SOO Algorithm}\label{sec2.1}

\begin{figure}[t]
\small
\includegraphics[width=0.92\textwidth]{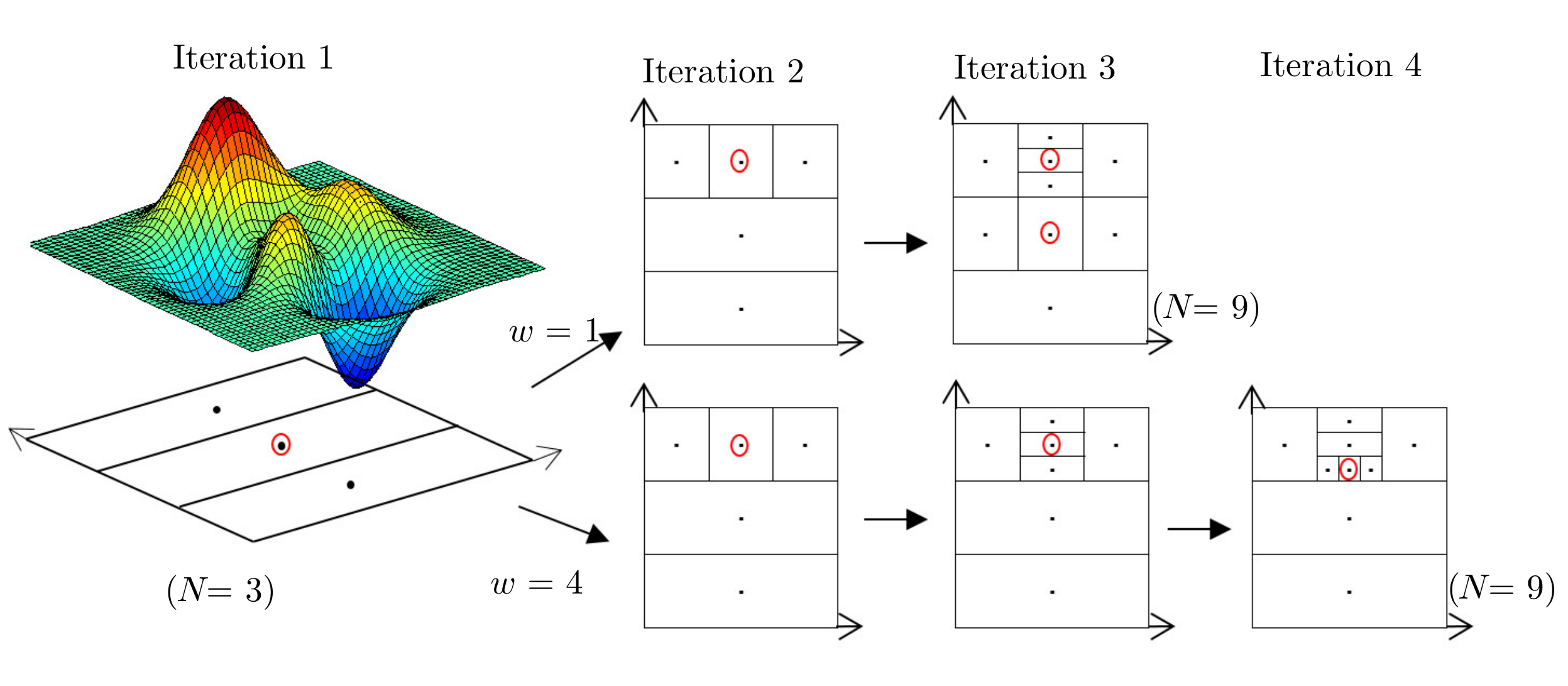}
\caption{Illustration of SOO ($w = 1$) and LOGO ($w = 1$ or $4$) at the
end of each iteration}\label{fig1}
\end{figure}

Before we discuss our algorithm in detail, we briefly describe its direct
predecessor, the SOO algorithm%
\footnote{We describe SOO with the simple division procedure that LOGO
uses. SOO itself does not specify a division procedure. }
\cite{mun11}. The top diagrams in Figure~\ref{fig1}
(the scenario with \(w=1\)) illustrates the functionality of the SOO algorithm in a simple
$2$-dimensional objective function. As illustrated in Figure~\ref{fig1},
the SOO algorithm employs hierarchical partitioning to maintain
hyperintervals, each center of which is the evaluation point of the
objective function $f$. That is, in Figure~\ref{fig1}, each rectangle
represents the hyperintervals at the end of each iteration of the
algorithm. Let $\psi_h$  be a set of rectangles of the same size that
are divided $h$ times. The algorithm uses a parameter, $\hmax
(n)$, 
to limit the size of rectangle so as to be not overly small
(and hence restrict the greediness of the search). In order to select
and refine intervals that are likely to contain a global optimizer,
the algorithm executes the following procedure:
\begin{enumerate}[(i)] 
\item\label{alg1-1} Initialize $\psi_0 = \{\Omega'\}$ and for all \(i>0\), $\psi_i = \{\varnothing\}$
\vspace{-7pt}
\item\label{alg1i} Set $h = 0$
\vspace{-7pt}
\item\label{alg1ii} Select the interval with the maximum center value among the
intervals in the set $\psi_h$
\vspace{-7pt}
\item\label{alg1iii}\label{alg1iv} If the interval selected by \eqref{alg1ii}  has a center value
greater than that of any \textit{larger} interval (i.e., intervals in  \(\psi_l\) for all \(l<h\)), divide  it and adds the new intervals to \(\psi_{h+1}\). Otherwise, reject the interval and skip this step. 
\vspace{-0pt}
\item\label{alg1v}  Set $h = h + 1$
\vspace{-7pt}
\item\label{alg1vi} Repeat \eqref{alg1ii}--\eqref{alg1v} until no smaller interval exists (i.e., until $\psi_l=\{\varnothing\}$ for all \(l\ge h\)) or $h > \hmax(n)$
\vspace{-7pt}
\item\label{alg10} Delete  all the intervals   already divided in \eqref{alg1iii} from \(\psi \) and repeat \eqref{alg1i}--\eqref{alg1vi}
\end{enumerate}

We now explain this procedure using the example in Figure~\ref{fig1}. For brevity, we use the term, ``iteration'', to  refer to the iteration  of step \eqref{alg1i}--\eqref{alg10}. In
Figure~\ref{fig1}, the center point is shown as a (black) dot in each
rectangle and each rectangle with a (red) circle around a (black) dot
is the one that was divided (into three smaller rectangles) during an
iteration. At the beginning of the first iteration, there is only one
rectangle in $\psi_0$, which is the entire search domain $\W'$. Thus,
step \eqref{alg1ii} selects this rectangle and step \eqref{alg1iv} divides it,
resulting in the leftmost diagram with $N = 3$ (the rectangle with
the center point with a red circle is the one divided during the first
iteration and the other two are created as a result). At the beginning
of the second iteration, there are three rectangles in $\psi_1$ (i.e.,
the three rectangles in the leftmost diagram with $N = 3$) but none in
$\psi_0$ (because step \eqref{alg10} in the previous iteration deleted the interval in $\psi_0$). Hence, steps \eqref{alg1ii}--\eqref{alg1iv} are not executed for $\psi_0$ and we 
begin with $\psi_1$. Step \eqref{alg1ii} selects the top rectangle from the three
rectangles because it has the maximum center point among these. Step \eqref{alg1iii}
divides it because there is no larger interval, resulting in the second
diagram on the top (labeled with $w = 1$). Iteration 2 continues by conducting steps
\eqref{alg1ii}--\eqref{alg1iv} for $\psi_2$ because there are three smaller rectangles in
$\psi_2$. Step \eqref{alg1ii} selects the center rectangle
on the top (in the second diagram on the top labeled with $w = 1$). However, step
\eqref{alg1iii} rejects it because its center value is not greater than that of
the larger rectangle in \(\psi_l\) with \(l<h=2\). There is no
smaller rectangle in \(\psi\) and iteration 2 ends. At the beginning of iteration 3,
there are two rectangles in $\psi_1$ and three rectangles in $\psi_2$
(as shown in the second diagram on the top labeled with $w = 1$). Iteration 3
begins by conducting steps \eqref{alg1ii}--\eqref{alg1iv} for $\psi_1$. Steps \eqref{alg1ii}--\eqref{alg1iii}
select and divide the top rectangle. For
rectangles in $\psi_2$, steps \eqref{alg1ii}--\eqref{alg1iii} select and divides the middle
rectangle. Here, the middle rectangle was rejected in iteration 2 because of a larger rectangle with a larger center value that existed in iteration 2. However, that larger rectangle no longer exists in iteration 3 due  to  step \eqref{alg10} in the end of iteration 2, and hence it is not rejected. The result is the third diagram (on the
top labeled with $w = 1)$. Iteration 3 continues for the newly created rectangles
in $\psi_3$. It halts, however, for the same reason as iteration 2.

\subsection{
Description of LOGO }\label{sec2.2}

 Let $\Psi_k$ be the superset that is the union of the $w$
sets as $\Psi_k=\psi_{kw}\cup\psi_{kw+1}\cup\dots\cup\psi_{kw+w-1}$ for
$k=0,1,2,\dots$. Then, similar to the SOO algorithm, the LOGO algorithm
conducts the following procedure to select and refine
the intervals that are likely to contain a global optimizer:
\begin{enumerate}[(i)]
\item\label{alg1-1} Initialize $\psi_0 = \{\Omega'\}$ and for all \(i>0\), $\psi_i = \{\varnothing\}$
\vspace{-7pt}
\item\label{alg2i} Set $k = 0$
\vspace{-7pt}
\item\label{alg2ii} Select the interval with the maximum center value among the
intervals in the superset~$\Psi_k$
\vspace{-7pt}
\item\label{alg2iii}\label{alg2iv} If the interval selected by \eqref{alg1ii}  has a center value
greater than that of any \textit{larger} interval (i.e., intervals in $\Psi_l$
with $l<k$), divide  it and adds the new intervals to \(\psi\). Otherwise, reject the interval and skip this step. 
\vspace{-7pt}
\item\label{alg2v}  Set $k = k + 1$
\vspace{-7pt}
\item\label{alg2vi} Repeat \eqref{alg2ii}--\eqref{alg2v} until no smaller interval exists (i.e.,
until $\Psi_l=\{\varnothing\}$ for all $l\ge k$) or $k>\floor{\hmax(n)/w}$.
\item\label{alg10} Delete  all the intervals   already divided in \eqref{alg1iii} from \(\psi \) and repeat \eqref{alg1i}--\eqref{alg1vi}
\end{enumerate}

When compared with the SOO algorithm, the above steps are identical except
that LOGO processes the superset $\Psi_k$ instead of set $\psi_h$. The
superset $\Psi_k$ is reduced to $\psi_h$ with $k = h$ when $w = 1$
and thus LOGO is reduced to SOO.

We now explain this procedure using the example in Figure~\ref{fig1}. With
$w = 1$, the LOGO algorithm functions in the same fashion as the
SOO algorithm. See the last paragraph in the previous section for
the explanation as to how the SOO and LOGO algorithms function in
this example. For the case with $w = 4$, the difference arises during
iteration 3 when compared to the case with $w = 1$. At the beginning
of iteration 3, there are two sets $\psi_1$ and $\psi_2$ (i.e., there
are two sizes of rectangles in the second diagram on the bottom with
$w = 4$). However, there is only one superset consisting of the two
sets $\Psi_0=\psi_0\cup\psi_{0+1}\cup\psi_{0+4-1}$. Therefore, step
\eqref{alg2ii}--\eqref{alg2iv} is conducted only for $k = 0$ and the LOGO algorithm divides
only the one rectangle with the highest center value among those in
$\Psi_0$. Consequently, the algorithm has one additional iteration
(iteration 4) using the same number of function evaluations (\(N=9\)) as the case
with $w = 1$. It can be seen that as $w$ increases, the algorithm is more
biased to the local search and, in this example, this strategy turns out
to be beneficial as the algorithm divides the rectangle near the global
optima more when $w = 4$ than when $w = 1$.

The pseudocode for the LOGO algorithm is provided in
Algorithm~\ref{alg1}. Steps \eqref{alg2i}, \eqref{alg2v}, and \eqref{alg2vi} correspond to the for-loop in lines 10--19. Steps \eqref{alg2ii}--\eqref{alg2iii} correspond to line 11 and line 12--14, respectively. We use following notation. Each hyperrectangle,
$\omega\subseteq\Omega'$, is coupled with a function value at its
center point $f(c_\omega)$, where $c$ indicates the center point of
the rectangle. As explained earlier, we use $h$ to denote the number
of divisions and the index of the set as in $\psi_h$. We define
$\omega_{h,i}$ to be the $i^\text{th}$  element of a set $\psi_h$ (i.e.,
$\omega_{h,i}\in\psi_h$). Let $x_{h,i}$ and $c_{h,i}$ be an arbitrary point
and the center point in the rectangle $\omega_{h,i}$, respectively.
We denote \(val[\omega_{h,i}]\) to indicate a stored function value of the center point in the rectangle \(\omega_{h,i}\). As it can be seen in line 14, this paper considers a simple division procedure with a rescaled domain \(\Omega'\). If we  have prior knowledge about the domain of the function, we should leverage the information. For example, we could map the original input space to another so that we can obtain a better \(\ell\) based on the theoretical results in Section 4, or we could  employ a more elaborate division procedure based on the prior knowledge. 
  
\begin{algorithm}[t]
\caption{LOGO algorithm}\label{alg1}
\begin{algorithmic}[1] 
\makeatletter \addtocounter{ALG@line}{-1} \makeatother 
\State {\bf Inputs (problem):} an objective function $f\colon x\in\R^D\to\R$,
the search domain $\W$:~$x\in\nobreak \W$
\State {\bf Inputs (parameter):} the search depth function
$\hmax\colon\Z^+\to[1,\infty)$, the local weight $w\in\Z^+$, stopping
condition
\State Define the set $\psi_h$ as a set of hyperrectangles divided $h$ times
\State Define the superset $\Psi_k$ as the union of the $w$ sets:
$\Psi_k=\psi_{kw}\cup\psi_{kw+1}\cup\dots\cup\psi_{kw+w-1}$
\State Normalize the domain $\Omega$ to $\W' = [0, 1]^D$
\State Initialize the variables:%
\begin{tabular}[t]{l}
the set of hyperrectangles: $\psi_h=\{\varnothing\}$, $h=0,1,2,\dots$,\\
the current maximum index of the set: $h_\mathit{upper}=0$\\
the number of total divisions: $n=1$
\end{tabular}
\State Adds the initial hyperrectangle $\W'$ to the set:
$\psi_0\leftarrow\psi_0\cup\{\W'\}$ (i.e., $\omega_{0,0}=\W'$)
\State Evaluate the function $f$ at the center point of $\W'$, $c_{0,0}$:
$\val[\omega_{0,0}]\leftarrow f(c_{0,0})$
\For{iteration ${}= 1, 2, 3,\dots$}
\State $\valmax\leftarrow-\infty$, $h_\mathit{plus}\leftarrow
h_\mathit{upper}$
\For{$k=0,1,2,\dots,\max(\floor{\min(\hmax(n),h_\mathit{upper})/w},h_\mathit{plus})$}
\State \Select\ a hyperrectangle to be divided:
$(h,i)\in\arg\max_{h,i}\val[\omega_{h,i}]$ for $h,i:\omega_{h,i}\in\Psi_k$
\If{$\val[\omega_{h,i}]>\valmax$}
\State $\valmax\leftarrow\val[\omega_{h,i}]$, $h_\mathit{plus}\leftarrow0$,
$h_\mathit{upper}\leftarrow\max(h_\mathit{upper},h+1)$, $n\leftarrow n+1$
\State \Divide\ this hyperrectangle $\omega_{h,i}$ along the longest
coordinate direction
\Statex\qquad\qquad\quad\begin{tabular}[t]{l}
- three smaller hyperrectangles are created $\rightarrow$
$\omega_\mathit{left}$, $\omega_\mathit{center}$, $\omega_\mathit{right}$\\
- $\val[\omega_\mathit{center}]\leftarrow\val[\omega_{h,i}]$
\end{tabular}
\State \Evaluate\ the function $f$ at the center points of the
two new hyperrectangles:
\begin{center}
$
\val[\omega_\mathit{left}]\leftarrow f(c_{\omega_\mathit{left}}),\quad
\val[\omega_\mathit{right}]\leftarrow f(c_{\omega_\mathit{right}})
$
\end{center}
\State \Group\ the new hyperrectangles into the set $h + 1$
and remove the original rectangle:
\begin{center}
$
\psi_{h+1}\leftarrow\psi_{h+1}
\cup\{\omega_\mathit{center},\omega_\mathit{left},\omega_\mathit{right}\},\quad
\psi_h\leftarrow\psi_h\setminus \omega_{h,i}
$
\end{center}
\EndIf
\State {\bf if} stopping condition is met {\bf then Return} $(h,i)=\arg\max_{h,i}\val[\omega_{h,i}]$
\EndFor
\EndFor
\end{algorithmic}
\end{algorithm}

Having discussed how the LOGO algorithm functions, we now consider the
reason why the algorithm might work well. 
The key mechanism of the DIRECT and SOO algorithms is to divide all the
hyperintervals with potentially highest upper bounds w.r.t.\ unknown
smoothness at each iteration. The idea behind the LOGO algorithm is to
reduce the number of divisions per iteration by biasing the search toward
the local area with the concept of the supersets. Intuitively, this can be beneficial for two
reasons. First, by reducing the number of divisions per iteration, more
information can be utilized when selecting intervals to divide. For
example, one may simultaneously divide five or ten intervals per
iteration. In the former, when selecting the sixth to the tenth interval
to divide, one can leverage information gathered by the previous five
divisions (evaluations), whereas the latter makes it impossible. Because
the selection of intervals depends on the information, which in turn
provides the new information to the next selection, the minor difference
in availability of the information may make the two sequences of the
search very different in the long run. Second, by biasing the search
toward the local area, the algorithm likely converges faster in a certain
type of problem. In many practical problems, we do not aim to find a
position of global optima, but a position with a function value close to
global optima. In that case, the local bias is likely beneficial unless
there are too many local optima, the value of which is far from that
of global optima. Even though our local bias strategy is motivated to
solve the problem of the impractically slow convergence rate of most
global optimization methods, the algorithm maintains guaranteed loss
bounds w.r.t.\ global optima, as discussed below.

\section{Theoretical Results: Finite-Time Loss Analysis}\label{sec2.3}

We first derive the loss bounds for the LOGO algorithm when it uses any
division strategy that satisfies certain assumptions. Then, we provide
the loss bounds for the algorithm with the concrete division strategy
provided in Algorithm~\ref{alg1} and with the parameter values that
we use in the rest of this paper. The motivation in the first part is
to extend the existing framework of the theoretical analysis and thus
produce the basis for future work. The second part is to prove that
the LOGO algorithm maintains finite-time loss bounds for the parameter
settings that we actually use in the experiments.

\subsection{Analysis for General Division Method}\label{sec2.3.1}

In this section, we generalize the result obtained by \citeA{mun13}
in that the previous result is now seen as a special case of the new
result when $w = 1$. The previous work provided the loss bound of the
SOO algorithm with any division process that satisfied the following
two assumptions, which we adopt in this section.

\begin{assA}[Decreasing diameter]\label{assA1}
There exists a function $\delta(h)>0$ such that for
any hyperinterval $\omega_{h,i}\subseteq\W'$, we have
$\delta(h)\ge\sup_{x_{h,i}}\sem(x_{h,i},c_{h,i})$, while
$\delta(h-1)\ge\delta(h)$ holds for $h\ge 1$.
\end{assA}

\begin{assA}[Well-shaped cell]\label{assA2}
There exists a constant $\nu > 0$ such that any hyperinterval
$\omega_{h,i}$ contains a $\sem$-ball of radius $\nu\delta(h)$ centered in
$\omega_{h,i}$.
\end{assA}

Intuitively, Assumption~\ref{assA1} states that the unknown local
smoothness $\sem$ is upper-bounded by a monotonically decreasing function
of $h$. This assumption ensures that each division does not increase
the upper bound, $\delta(h)$. Assumption~\ref{assA2} ensures that every
interval covers at least a certain amount of space in order to relate
the number of intervals to the unknown smoothness $\sem$ (because $\sem$
is defined in terms of space). To present our analysis, we need to define
the relevant terms and variables. We define $\ops$-optimal set $X_\ops$ as
\[
X_\ops:=\{x\in\W':f(x)+\ops\ge f(x^*)\}.
\]
That is, the set of $\ops$-optimal set $X_\ops$ is the set of input
vectors whose function value is at least $\ops$-close to the value of the
global optima. In order to bound the number of hyperintervals relevant
to the $\ops$-optimal set $X_\ops$, we define \emph{near-optimality
dimension} as follows.

\begin{definition}[Near-optimality dimension]\label{def1}
The near-optimality dimension is the smallest $d\ge 0$ such that there
exists $C > 0$, for all $\ops > 0$, the maximum number of disjoint
$\sem$-balls of radius $\nu\ops$ centered in the $\ops$-optimal set $X_\ops$
is less than or equal to $C\ops^{-d}$.
\end{definition}

The \emph{near-optimality dimension} was introduced by \citeA{mun11} and
is closely related to a previous measure used by \citeA{kle08}. The value
of the near-optimality dimension $d$ depends on the objective function
$f$, the semi-metric $\sem$ and the division strategy (i.e., the constant
$\nu$ in Assumption~\ref{assA2}). If we consider a semi-metric~$\sem$
that satisfies Assumptions~\ref{ass1}, \ref{assA1}, and \ref{assA2},
then the value of $d$ depends only on such a semi-metric~$\sem$ and
the division strategy. In Theorem~\ref{the2}, we show that the division
strategy of the LOGO algorithm can let $d = 0$ for a general class of
semi-metric $\sem$.

Now that we have defined the relevant terms and variables used in previous
work, we introduce new concepts to advance our analysis. First, we define
the set of $\delta$-optimal hyperinterval $\psi_h(w)^*$ as
\[
\psi_h(w)^*:=\{\omega_{h,i}\subseteq \W': f(c_{\omega_{h,i}}) +
\delta(h-w+1)\ge f(x^*)\}.
\]
The $\delta$-optimal hyperinterval $\psi_h(w)^*$ is used to relate
the hyperintervals to  $\ops$-optimal set $X_\ops$. Indeed, the
$\delta$-optimal hyperinterval $\psi_h(w)^*$ is almost identical to
the $\delta(h-w+1)$-optimal set $X_{\delta(h-w+1)}$ ($\ops$-optimal set
$X_\ops$ with $\ops$ being $\delta(h-w+1)$), except that $\psi_h(w)^*$
only considers the hyperintervals and the values of their center points
while $X_{\delta(h-w+1)}$ is about the whole input vector space. In order
to relate $\psi_h(w)^*$ to $\psi_h(1)^*$, we define \emph{$\sem$-ball
ratio} as follows.

\begin{definition}[$\sem$-ball ratio]\label{def2}
For every $h$ and $w$, the $\sem$-ball ratio is the smallest
$\lambda_h(w)>0$ such that the volume of a $\sem$-ball of radius
$\delta(h-w+1)$ is no more than the volume of $\lambda_h(w)$ disjoint
$\sem$-balls of radius $\delta(h)$.
\end{definition}

In the following lemma, we bound the maximum cardinality of
$\psi_h(w)^*$. We use $\abs{\psi_h(w)^*}$ to denote the cardinality.

\begin{lemma}\label{lem1}
Let $d$ be the near-optimality dimension and $C$ denote the corresponding
constant in Definition~\ref{def1}. Let $\lambda_h(w)$ be the $\sem$-ball
ratio in Definition~\ref{def2}. Then, the $\delta$-optimal hyperinterval
is bounded as
\[
\abs{\psi_h(w)^*}\le C\lambda_h(w)\delta(h-w+1)^{-d}.
\]
\end{lemma}

\begin{proof}
The proof follows the definition of $\ops$-optimal set
$X_\ops$, Definition~\ref{def1}, Definition~\ref{def2}, and
Assumption \ref{assA2}. From the definition of $\ops$-optimal space
$X_\ops$, we can write $\delta(h-w+1)$-optimal set as
\[
X_{\delta(h-w+1)}=\{x\in\W':f(x)+\delta(h-w+1)\ge f(x^*)\}.
\]
The definition of the near-optimality dimension (Definition~\ref{def1})
implies that at most $C\delta(h-w+1)^{-d}$ centers of disjoint
$\sem$-balls of radius $\nu\delta(h-w+1)$ exist within space
$X_{\delta(h-w+1)}$. Then, from the definition of the $\sem$-ball
ratio (Definition~\ref{def2}), the space of $C\delta(h-w+1)^{-d}$
disjoint $\sem$-balls of radius $\nu\delta(h-w+1)$ is covered by
at most $C\lambda_h(w)\delta(h-w+1)^{-d}$ disjoint $\sem$-balls
of radius $\nu\delta(h)$. Notice that the set of space covered by
$C\delta(h-w+1)^{-d}$ disjoint $\sem$-balls of radius $\nu\delta(h-w+1)$
is a superset of $X_{\delta(h-w+1)}$. Therefore, we can deduce that
there are at most $C\lambda_h(w)\delta(h-w+1)^{-d}$ centers of disjoint
$\sem$-balls of radius $\nu\delta(h)$ within $X_{\delta(h-w+1)}$. Now,
recall the definition of the $h$-$w$-optimal interval,
\[
\psi_h(w)^*:=\{\omega_{h,i}\subseteq\W':f(c_{\omega_{h,i}}) +
\delta(h-w+1)\ge f(x^*)\}
\]
and notice that the number of intervals is equal to the
number of centers $c_{\omega_{h,i}}$ that satisfy the condition
$f(c_{\omega_{h,i}})+\delta(h-w+1)\ge f(x^*)$. Assumption \ref{assA2}
causes this number to be equivalent to the number of centers of disjoint
$\sem$-balls with radius $\nu\delta(h)$, which we showed to be upper bounded
by $C\lambda_h(w)\delta(h-w+1)^{-d}$.
\end{proof}

Next, we bound the maximum size of \emph{the optimal hyperinterval},
which contains a global optimizer $x^*$. In the following analysis,
we use the concept of the set and superset of hyperintervals. Recall
that set $\psi_h$ contains all the hyperintervals that have been divided
$h$ times thus far, and superset $\Psi_k$  is a union of the $w$ sets,
given as $\Psi_k=\psi_{kw}\cup\psi_{kw+1}\cup\dots\cup\psi_{kw+w-1}$
for $k=0,1,2,\dots$. We say that a hyperinterval is \emph{dominated}
by other intervals when the hyperinterval is not divided because its
center value is at most that of other hyperintervals in any set.

\begin{lemma}\label{lem2}
Let $k_n^*$ be the highest integer such that the optimal hyperinterval,
which contains a global optimizer $x^*$, belongs to the superset
$\Psi_{k_n^*}$ after $n$ total divisions (i.e., $k_n^*\le n$ determines
the size of the optimal hyperinterval, and hence the loss of the
algorithm). Then, $k_n^*$ is lower bounded as $k_n^*\ge K$ with any $K$
that satisfies $0\le K\le\floor{\hmax(n)/w}$ and
\[
n\ge
\bbfloor{\frac{\hmax(n)+w}{w}}\sum_{k=0}^K\biggl(\abs{\psi_{kw}(1)^*}+\sum_{l=1}^{w-1}\abs{\psi_{kw+l}(l+1)^*}\biggr).
\]
\end{lemma}

\begin{proof}
Let $\tau(\Psi_k)$ be the number of divisions, with which the algorithm
further divides the optimal hyperinterval in superset $\Psi_k$ and
places it into $\Psi_{k+1}$. In the example in Figure~\ref{fig1}
with $w = 1$, the optimal hyperinterval is initially the whole
domain $\W'\subseteq\Psi_0$. It is divided with the first division
and the optimal hyperinterval is placed into $\Psi_1$. Therefore,
$\tau(\Psi_0)=1$. Similarly, $\tau(\Psi_1)=2$. A division of non-optimal
interval occurs before that of the optimal one for $\tau(\Psi_2)$ and
hence $\tau(\Psi_2)=4$.  In other words, $\tau(\Psi_k)$ is the time when
the optimal hyperinterval in superset $\Psi_k$ is further divided
and escapes the superset $\Psi_k$, entering into $\Psi_{k+1}$.
Let $c_{kw+l,i^*}$ be the center point of the optimal hyperinterval in
a set $\psi_{kw+l}\subseteq\Psi_k$.

We prove the statement by showing that the quantity
$\tau(\Psi_k)-\tau(\Psi_{k-1})$ is bounded by the number of
$\delta$-optimal hyperintervals $\psi_h(w)^*$. To do so, let us
consider the possible hyperintervals to be divided during the time
$[\tau(\Psi_{k-1}),\tau(\Psi_k)-1]$. For the hyperintervals in the set
$\psi_{kw}$, the ones that can possibly further be divided during this
time must satisfy $f(c_{kw,i})\ge f(c_{kw,i^*})\ge f(x^*) -\delta(kw)$. The
first inequality is due to the fact that the algorithm does not divide an
interval that has center value less than the maximum center value of an
existing interval for each set, and there exists $f(c_{kw,i^*})$ during
the time $[\tau(\Psi_{k-1}),\tau(\Psi_k)-1]$. The second inequality follows
Assumption~\ref{ass1} and the definition of the optimal interval. Then,
from the definition of $\psi_h(w)^*$, the hyperintervals that can possibly
be divided during this time belong to $\psi_{kw}(1)^*\subseteq\Psi_k$.

In addition to set $\psi_{kw}$, in superset $\Psi_k$, there are
sets $\psi_{kw+l}$ with $l\colon w - 1 \ge l \ge 1$. For these sets,
we have $f(c_{kw+l,i})\ge f(c_{lw+l,i^*})\ge f(x^*)-\delta(kw)$
with similar deductions. Here, notice that during the time
$[\tau(\Psi_{k-1}),\tau(\Psi_k)-1]$, we can be sure that the center
value in the superset is lower bounded by $f(c_{kw,i^*})$ instead of
$f(c_{kw+l,i^*})$. In addition, we have $\delta(kw)=\delta(kw+l-l)$. Thus,
we can conclude that the hyperintervals in set $\psi_{kw+l}$ that
can be divided during time $[\tau(\Psi_{k-1}),\tau(\Psi_k)-1]$ belongs to
$\psi_{kw+l}(l+1)^*$ where $(w - 1)\ge l\ge 1$.

No hyperinterval in superset $k$ may be divided at iteration since the
intervals can be dominated by those in other supersets. In this case,
we have $f(c_{jw+l,i})\ge f(c_{kw,i^*})\ge f(x^*)-\delta(kw)$ for
some $j < k$ and $l\ge 0$. With similar deductions, it is easy to see
$f(x^*)-\delta(kw)\ge f(x^*)-\delta(jw+l)$. Thus, the hyperintervals
in a superset $\Psi_j$ with $j < k$ that can dominate those in
superset $\Psi_k$ during $[\tau(\Psi_{k-1}),\tau(\Psi_k)-1]$ belongs to
$\psi_{jw+l}(1)^*$.

Putting the above results together and noting that the algorithm
divides at most $\floor{\hmax(n)/w} + 1$ intervals during any iteration
($h_\mathit{plus}$ plays its role only when the algorithm divides at
most one interval), we have
\[
\tau(\Psi_k)-\tau(\Psi_{k-1}) \le \biggl(\bbfloor{\frac{\hmax(n)}{w}}+1\biggr)
\biggl(\abs{\psi_{kw}(1)^*}+\sum_{l=1}^{w-1}\abs{\psi_{kw+l}(l+1)^*}+\sum_{j=1}^{k-1}\sum_{l=0}^{w-1}\abs{\psi_{jw+l}(1)^*}\biggr).
\]
Then,
\[
\sum_{k=1}^{k_n^*}\tau(\Psi_k)-\tau(\Psi_{k-1}) \le
\bbfloor{\frac{\hmax(n)+w}{w}}
\sum_{k=1}^{k_n^*}\biggl(\abs{\psi_{kw}(1)^*}+\sum_{l=1}^{w-1}\abs{\psi_{kw+l}(l+1)^*}\biggr)
\]
since the last term for a superset $\Psi_j$ with $j\le k-1$ in
the previous inequality contains only the optimal intervals that
are subsets of the optimal intervals covered by the new summation
$\sum_{k=1}^{k_n^*}$.

If $k_n^*\ge\floor{\hmax(n)/w}$, then the statement always holds true for
any $0\,{\le}\, K\,{\le}\, \floor{\hmax(n)/w}$ since $k_n^*\ge\floor{\hmax(n)/w}\ge
K$. Accordingly, we assume $k_n^*<\floor{\hmax(n)/w}$ in the
following. Since $\tau(\Psi_0)$ is upper bounded by the term in the previous
summation on the right hand of the above inequality with $k = 0$,
\[
\tau(\Psi_{k_n^*+1})\le \bbfloor{\frac{\hmax(n)+w}{w}}
\sum_{k=0}^{k_n^*+1}\biggl(\abs{\psi_{kw}(1)^*}+\sum_{l=1}^{w-1}\abs{\psi_{kw+l}(l+1)^*}\biggr).
\]
By the definition of $k_n^*$, we have $n<\tau(\Psi_{k_n^*+1})$. Therefore,
for any $K\le\floor{\hmax(n)/w}$ such that
\begin{multline*}
\bbfloor{\frac{\hmax(n)+w}{w}}\sum_{k=0}^{K}\biggl(\abs{\psi_{kw}(1)^*}+\sum_{l=1}^{w-1}\abs{\psi_{kw+l}(l+1)^*}\biggr)\\
\le n <
\bbfloor{\frac{\hmax(n)+w}{w}}\sum_{k=0}^{k_n^*+1}\biggl(\abs{\psi_{kw}(1)^*}+\sum_{l=1}^{w-1}\abs{\psi_{kw+l}(l+1)^*}\biggr),
\end{multline*}
we have $k_n^*\ge K$.
\end{proof}

With Lemmas~\ref{lem1} and \ref{lem2}, we can now present the main
result in this section that provides the finite-time loss bound of the
LOGO algorithm.

\begin{theorem}\label{the1}
Let $\sem$ be a semi-metric such that Assumptions~\ref{ass1}, \ref{assA1},
and \ref{assA2} are satisfied. Let $h(n)$ be the smallest integer $h$
such that
\[
n\le
C\bbfloor{\frac{\hmax(n)+w}{w}}\sum_{k=0}^{\floor{h/w}}\biggl(\delta(kw)^{-d}+\sum_{l=1}^{w-1}\lambda_{kw+l}(l+1)\delta(kw)^{-d}\biggr).
\]
Then, the loss of the LOGO algorithm is bounded as
\[
r_n\le\delta\bigl(\min(w\floor{h(n)/w}-w,w\floor{\hmax(n)/w})\bigr).
\]
\end{theorem}

\begin{proof}
From Lemma~\ref{lem1} and the definition of $h(n)$,
\begin{align*}
n&>C\bbfloor{\frac{\hmax(n)+w}{w}}\sum_{k=0}^{\floor{h(n)/w}-1}\biggl(\delta(kw)^{-d}
+\sum_{l=1}^{w-1}\lambda_{kw+l}(l+1)\delta(kw+l-l)^{-d}\biggr)\\
&\ge
\bbfloor{\frac{\hmax(n)+w}{w}}\sum_{k=0}^{\floor{h(n)/w}-1}\biggl(\abs{\psi_{kw}(1)^*}+\sum_{l=1}^{w-1}\abs{\psi_{kw+l}(l+1)^*}\biggr).
\end{align*}
Therefore, we set $K$ as $K=\floor{h(n)/w}-1$ in the following to apply
the result of Lemma~\ref{lem2}. Then, it follows that $k_n^*\ge K(n)$ when
$K<\floor{\hmax(n)/w}$. Here, the number of divisions that an interval
in the superset $\Psi_K$ is at least $Kw = w\floor{h(n)/w}-w$. Therefore,
from Assumptions~\ref{ass1}, \ref{assA1}, and \ref{assA2}, we can deduce
that $r_n\le\delta(w\floor{h(n)/w}-w)$.

When $K\ge\floor{\hmax(n)/w}$, we have $\floor{(h_\mathit{upper})/w}\ge
k_n^*\ge\floor{\hmax(n)/w}$. Thus, in this case, we have $k_n^*$ being
equal to at least $\floor{\hmax(n)/w}$. From Assumptions~\ref{ass1},
\ref{assA1}, and \ref{assA2}, we can similarly deduce that
$r_n\le\delta(w\floor{\hmax(n)/w})$.
\end{proof}

The loss bound stated by Theorem~\ref{the1} applies to the LOGO algorithm
with any division strategy that satisfies Assumptions~\ref{assA1} and
\ref{assA2}. We add the following assumption about the division process
to derive more concrete forms of the loss bound.

\begin{assA}[Decreasing diameter revisit]\label{assA3}
The decreasing diameter defined in Assumption~\ref{ass1} can be
written as $\delta(h) = c\gb^{h/D}$ for some $c > 0$ and  $\gb
< 1$, and accordingly the corresponding $\sem$-ball ratio is
$\lambda_h(w)=(\delta(h-w+1)/\delta(h))^D$.
\end{assA}

Assumption~\ref{assA3} is similar to an assumption made by \citeA{mun13},
which is that $\delta(h) =\nobreak c\gb^h$. In contrast to the
previous assumption, our assumption explicitly reflects the fact that the size
of a hyperinterval decreases at a slower rate for higher dimensional
problems.  For the LOGO
algorithm, the validity of Assumptions~\ref{assA1}, \ref{assA2}, and
\ref{assA3} is confirmed in the next section.

We now present the finite-loss bound for the LOGO algorithm in the case
of the general division strategy with the above additional assumption
and with $d = 0$.

\begin{corollary}\label{cor1}
Let $\sem$ be a semi-metric such that Assumptions~\ref{ass1}, \ref{assA1},
\ref{assA2}, and \ref{assA3} are satisfied. If the near-optimality
dimension $d = 0$ and $\hmax(n)$ is set to $\sqrt{n}-w$, then the loss
of the LOGO algorithm is bounded for all $n$ as
\[
r_n\le
c\exp\Biggl(-\min\biggl(\sqrt{n}\frac{w}{C}\biggl(\frac{\gb^{-w}-1}{\gb^{-1}-1}\biggr)^{\!-1}-2,
\sqrt{n}-w\biggr)\frac{w}{D}\ln\frac1\gb\Biggr).
\]
\end{corollary}

\begin{proof}
Based on the definition of $h(n)$ in Theorem~\ref{the1}, we first relate
$h(n)$ to $n$ as
\begin{align*}
n&\le
C\frac{\hmax(n)+w}{w}\sum_{k=0}^{\floor{h(n)/w}}\biggl(\delta(kw)^{-d}+\sum_{l=1}^{w-1}\lambda_{kw+l}(l+1)\delta(kw+l-l)^{-d}\biggr)\\
&=C\frac{\hmax(n)+w}{w}\sum_{k=0}^{\floor{h(n)/w}}\biggl(1+\sum_{l=1}^{w-1}\gb^{-w}\biggr)
\le C\frac{\hmax(n)+w}{w}\biggl(\bbfloor{\frac{h(n)}w}+1\biggr)\sum_{l=0}^{w-1}\gb^{-w}.
\end{align*}
The first line follows the definition of $h(n)$, and the second line is
due to $d = 0$ and Assumption~\ref{assA3}. By algebraic manipulation,
\[
\bbfloor{\frac{h(n)}{w}}\ge
\frac{n}C\frac{w}{\hmax(n)+w}\biggl(\frac{\gb^{-w}-1}{\gb^{-1}-1}\biggr)^{\!-1}-1.
\]
Here, we use $\hmax(n)=\sqrt{n}-w$, and hence
\[
\bbfloor{\frac{h(n)}{w}}\ge
\sqrt{n}\frac{w}C\biggl(\frac{\gb^{-w}-1}{\gb^{-1}-1}\biggr)^{\!-1}-1.
\]
By substituting these results into the statement of Theorem~\ref{the1},
\[
r_n\le
\delta\Biggl(\min\biggl(\sqrt{n}\frac{w^2}C\biggl(\frac{\gb^{-w}-1}{\gb^{-1}-1}\biggr)^{\!-1}-2w,
w\sqrt{n}-w^2\biggr)\Biggr).
\]
From Assumption~\ref{assA3}, $\delta(h) = c\gb^{h/D}$. By using $\delta(h)
= c\gb^{h/D}$ in the above inequality, we have the statement of this
corollary.
\end{proof}

\begin{figure}
\centering
\hair2pt
\labellist
\pinlabel $\gb$ [r] at 0 75
\pinlabel $w$ [t] at 100 0
\endlabellist
\includegraphics[width=0.4\textwidth]{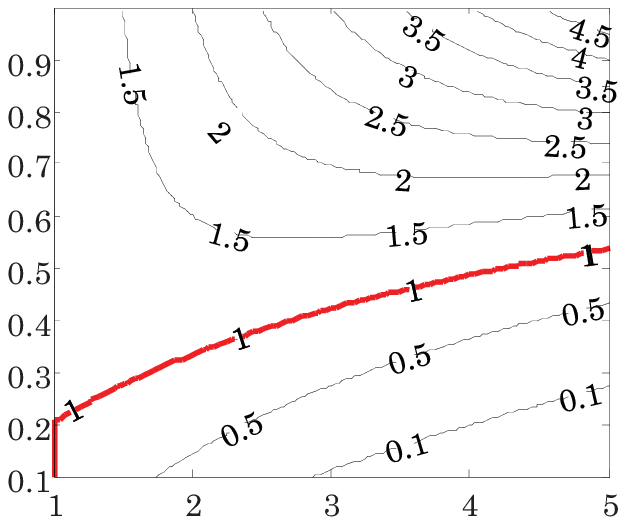}
\vskip5mm
\caption{Effect of local bias $w$ on loss bound in the case of $d =
0 :=w^2(\gb^{-1}-1)/(\gb^{-w}-1)$}\label{fig2}
\end{figure}

Corollary~\ref{cor1} shows that the LOGO algorithm guarantees an
exponential bound on the loss in terms of $\sqrt{n}$ (a stretched
exponential bound in terms of $n$). The loss bound in Corollary~\ref{cor1}
becomes almost identical to that of the SOO algorithm with $w =
1$. Accordingly, we illustrate the effect of $w$, when $n$ is
large enough to let us focus on the coefficient of $\sqrt{n}$, in
Figure~\ref{fig2}. The (red) bold line with label 1 indicates the
area where $w$ has no effect on the bound. The area with lines having
labels greater than one is where $w$ improves the bound, and the area
with labels less than one is where $w$ diminishes the bound. More
concretely, in the figure, we consider the ratio of the coefficient
of $\sqrt{n}$ in the loss bound with the various value of $w$ to that
with $w = 1$. The ratio is $w^2(\gb^{-1}-1)/(\gb^{-w}-1)$ or $w$,
depending on which element of the min in the bound is smaller. Since
$w^2(\gb^{-1}-1)/(\gb^{-w}-1)$ is at most $w$ (in the domain we consider),
we plotted $w^2(\gb^{-1}-1)/(\gb^{-w}-1)$ to avoid overestimating
the benefit of $w$. Thus, this is a rather pessimistic illustration
of the advantage of our generalization regarding $w$. For instance,
if the second element of the min in the bound is smaller and $n$ is
large enough, increasing $w$ always improves the bound, regardless of
the values in Figure~\ref{fig2}.

The next corollary presents the finite-loss bound for the LOGO algorithm
in the case of $d\neq 0$.

\begin{corollary}\label{cor2} 
Let  $\sem$ be a semi-metric such that
Assumptions~\ref{ass1}, \ref{assA1}, \ref{assA2}, and \ref{assA3} are
satisfied. If the near-optimality dimension $d > 0$ and $\hmax(n)$ is
set to be $\Theta((\ln n)^{c_1})-w$ for some $c_1 > 1$, the loss of the
LOGO algorithm is bounded as
\[
r_n\le \OO\Biggl(n^{-1/d}\biggl(w^2(\gb^{wd/D}-\gb^{2wd/D})\biggl(\frac{\gb^{-w}-1}{\gb^{-1}-1}\biggr)^{\!-1}\biggr)^{\!-1/d}\,\Biggr).
\]
\end{corollary}

\begin{proof}
In the same way as the first step in the proof of Corollary~\ref{cor1},
except for $d > 0$,
\[
n\le
Cc^{-d}\frac{\hmax(n)+w}{w}\sum_{k=0}^{\floor{h(n)/w}}\sum_{l=0}^{w-1}\gb^{-l-kwd/D}.
\]
The reason why we could not bound the loss in a similar rate as in the
case $d = 0$ is that the last summation term $\sum_{l=0}^{w-1}$ is no
longer independent of $k$. Since
\[
\sum_{l=0}^{w-1}\gb^{-l-kwd/D} =
\gb^{-kwd/D}\frac{\gb^{-w}-1}{\gb^{-1}-1},\quad
\sum_{k=0}^{\floor{h(n)/w}}\gb^{-kwd/D} =
\frac{\gb^{-(\floor{h(n)/w}+1)wd/D}-1}{\gb^{-wd/D}-1},
\]
with algebraic manipulation,
\[
c^{-d}(\gb^{-(\floor{h(n)/w}+1)wd/D}-1) \ge
\frac{n}C\frac{w}{\hmax(n)+w}\biggl(\frac{\gb^{-w}-1}{\gb^{-1}-1}\biggr)^{\!-1}(\gb^{-wd/D}-1).
\]
Therefore,
\[
c\gb^{(w\floor{h(n)/w}-w)/D} \ge
\Biggl(\frac{n}C\frac{w}{\hmax(n)+w}\biggl(\frac{\gb^{-w}-1}{\gb^{-1}-1}\biggr)^{\!-1}(\gb^{-wd/D}-1)\gb^{2wd/D}\Biggr)^{\!-1/d}.
\]
From Theorem~\ref{the1} and Assumption~\ref{assA3},
\[
r_n\le\max\Biggl(\biggl(\frac{n}C\frac{w}{\hmax(n)+w}
\biggl(\frac{\gb^{-w}-1}{\gb^{-1}-1}\biggr)^{\!-1}
(\gb^{wd/D}-\gb^{2wd/D})\biggr)^{\!-1/d},c\gb^{(w\floor{\hmax(n)/w}-w)/D}\Biggr).
\]
For $\hmax(n)=\Theta((\ln n)^{c_1})-w$ and for a sufficiently large $n$,
the first element of the previous max becomes larger than the second one,
and its order is equivalent to the one in the statement.
\end{proof}

We derived the loss bound for the SOO algorithm
with Assumption~\ref{assA3} in the case of $d\neq
0$ as well. The SOO version of the loss bound is
$r_n\le \OO(n^{-1/d}(\gb^{d/D}-\gb^{2d/D})^{-1/d})$, which is
equivalent to the loss bound of the LOGO algorithm with $w = 1$ in
Corollary~\ref{cor2}. In Figure~\ref{fig3}, we thereby illustrate the
effect of $w$ on the loss bound in the $\OO$ form. In the figure, we
plotted the ratio of the elements inside $\OO$ of the loss bounds. From
Figure~\ref{fig2} and Figure~\ref{fig3}, we can infer that the loss bound
is improved with $w > 1$ if $\gb$ is large and $d$ is small (when $n$
is sufficiently large). Intuitively, this makes sense, since there are
more of the different yet similar sizes of hyperintervals w.r.t.\ $\sem$
if $\gb$ is larger and $d$ is smaller. In that case, dividing all the
hyperintervals in the marginally different sizes would be redundant
and a waste of computational resources. Note that our discussion here
is limited to the loss bound that we have now, which may be tightened
in future work. We would then see the different effects of $w$ on such
tightened bounds.

\begin{figure}
\captionsetup[subfigure]{labelformat=empty}
\centering
\begin{subfigure}[b]{0.296\textwidth}
\hair2pt
\labellist
\pinlabel $\gb$ [r] at 12 55
\pinlabel $w$ [t] at 75 -5
\endlabellist
\includegraphics[width=1.13\textwidth]{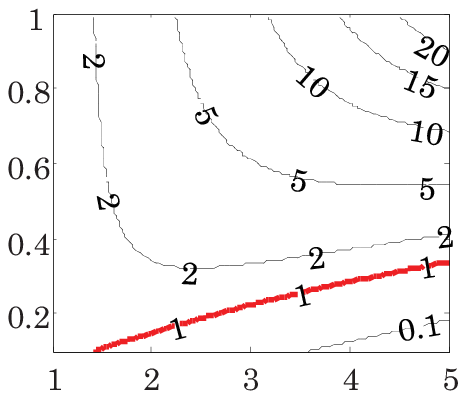}\hfill
\vskip5mm
\caption{(a) $d=0.01$}
\end{subfigure}\hfill
\begin{subfigure}[b]{0.31\textwidth}
\labellist
\pinlabel $\gb$ [r] at 5 55
\pinlabel $w$ [t] at 75 -5
\endlabellist
\includegraphics[width=1.01\textwidth]{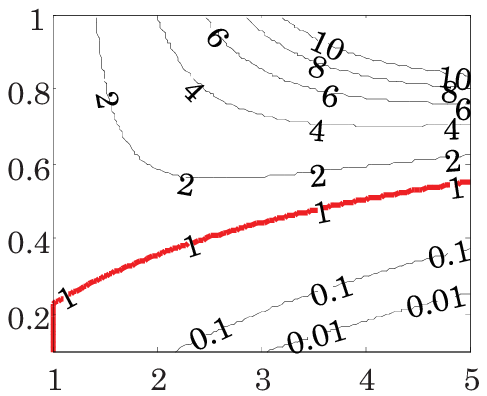}\hfill
\vskip5mm
\caption{(b) $d=0.5$}
\end{subfigure}\hfill
\begin{subfigure}[b]{0.31\textwidth}
\labellist
\pinlabel $\gb$ [r] at 2 55
\pinlabel $w$ [t] at 70 -5
\endlabellist
\includegraphics[width=0.98\textwidth]{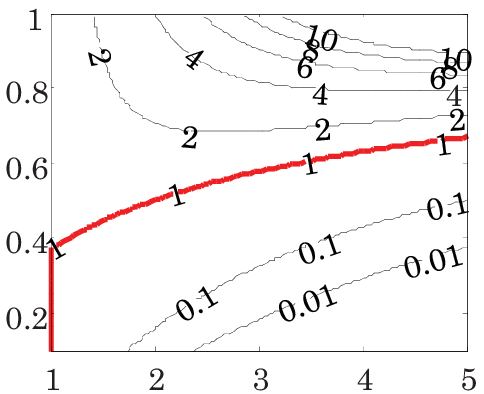}
\vskip5mm
\caption{(c) $d=1.0$}
\end{subfigure}
\caption{Effect of local bias $w$ on loss bound in the case of $d\neq 0
:=(w^2(\gb^{wd/D}-\gb^{2wd/D}(\frac{\gb^{-w}-1}{\gb^{-1}-1})^{-1})^{-1/d}/(\gb^{d/D}-\gb^{2d/D})^{-1/d}$}\label{fig3}
\end{figure}

\subsection{Basis of Practical Usage}\label{sec2.3.2}

In this section, we derive the loss bound of the LOGO algorithm with
the concrete division strategy presented in Section~\ref{sec2.1}. The
purpose of this section is to analyze the LOGO algorithm with the division
process and the parameter settings that are actually used in the rest
of this paper. The results of this section are directly applicable to
our experiments. In this section, we discard
Assumptions~\ref{assA1}, \ref{assA2}, and \ref{assA3}. We consider the
following assumption to present the loss bound in a concrete form.

\begin{assB}\label{assB1}
There exists a semi-metric $\sem$ such that that it satisfies
Assumption~\ref{ass1} and both of the following conditions hold:
\begin{itemize}
\item 
there exist $b >
0$, $\alpha > 0$ and $p\ge1$ such that for all \(x,y \in \Omega'\),  $\sem(x,y)=b\norm{x-y}^\alpha_p$ 
\item 
there exist \(\theta \in (0,1)\) such that for all \(x \in \Omega'\), $ f({ x^ * })  \ge  f( x) + \theta \ell \left( { x,{ x^ * }} \right)$. 
\end{itemize}
\end{assB}

First, we state that the loss bound of the algorithm with the practical
division process and parameter settings decreases at a stretched
exponential rate.

\begin{theorem}[worst-case analysis]\label{the2}
Let $\sem$ be a semi-metric such that Assumptions~\ref{ass1} and
\ref{assB1} are satisfied. The loss of the LOGO algorithm is bounded as
\[
r_n\le
c\exp\Biggl(-\min\biggl(\sqrt{n}\frac{w}{w'C}\biggl(\frac{\gb^{-w}-1}{\gb^{-1}-1}\biggr)^{\!-1}-2,
w'\sqrt{n}-w\biggr)\frac{w}{D}\ln\frac1{\gb}\Biggr)
\]
where $\gb=3^{-\alpha}$ and $c=b3^\alpha D^{\alpha/p}$. Here, $w'=1$
if we set the parameter as $\hmax(n)=\sqrt{n}-w$. On the other hand,
$w'=w$ if we set the parameter as $\hmax(n)=w\sqrt{n}-w$.
\end{theorem}

\begin{proof}
From Assumption \ref{assB1} and the division strategy,
\[
{\sup}_{x_{h,i}}\sem(x_{h,i},c_{h,i})\le b(3^{-\floor{h/D}}D^{1/p})^\alpha
= bD^{\alpha/p}3^{-\alpha\floor{h/D}}
\]
which corresponds to the \emph{diagonal} length of each
rectangle, while $3^{-\alpha\floor{h/D}}$ corresponds to the
length of the \emph{longest} side. This quantity is upper bounded
by $3^{-\alpha h/D +\alpha}$. Thus, we consider the case where
$\delta(h)=b3^{\alpha} D^{\alpha/p}3^{-\alpha h/D}$, which satisfies
Assumption~\ref{assA1}. Also, Assumption~\ref{assA3} is satisfied for
$\delta(h)$ with $\gb=3^{-\alpha}$ and $c=b3^\alpha D^{\alpha/p}$.

Every rectangle contains at least a $\sem$-ball of radius corresponding to
the length of the \emph{shortest} side for the rectangle. Consequently,
we have at least a $\sem$-ball of radius $\nu\delta(h)=b3^{-\alpha}3^{-\alpha
h/D}$ for any rectangle where $\nu=3^{-2\alpha}D^{-\alpha/p}$, which
satisfies Assumption~\ref{assA2}.

From Assumption~\ref{assB1}, the volume $V$ of a $\sem$-ball of
radius $\nu\delta(h)$ is proportional to $(\nu\delta(h))^D$ as the
following: $V^p_D(\nu\delta(h))=(2\nu\delta(h)\Gamma(1+1/p))^D /
\Gamma(1+D/p)$. Therefore, Assumption~\ref{assA3} is satisfied for
the $\sem$-ball ratio $\lambda_h(w)$. In addition, the \emph{$\delta(h)$-optimal
set} $X_{\delta(h)}$ is covered by a $\sem$-ball of radius $\delta(h)$ by
Assumption \ref{assB1}, and thereby contains at most $(\delta(h)/\nu\delta(h))^D =
\nu^{-D}$ disjoint $\sem$-balls of radius $\nu\delta(h)$. Hence, the number of the
$\sem$-balls does not depend on $\delta(h)$, which means $d = 0$.

Now that we have satisfied Assumptions~\ref{assA1}, \ref{assA2},
and \ref{assA3} with $\gb=3^{-\alpha}$, $c=b3^\alpha D^{\alpha/p}$,
and $d = 0$, we obtain the statement by following the proof of
Corollary~\ref{cor1}.
\end{proof}

Regarding the effect of local orientation $w$, Theorem~\ref{the2} presents
the worst-case analysis. Recall that $w$ is introduced in this paper to
restore the practicality of global optimization methods. Thus, focusing
on the worst case is likely too pessimistic. To mitigate this problem,
we present the following optimistic analysis.

\begin{theorem}[best-case analysis in terms of $w$]\label{the3}
Let $\sem$ be a semi-metric such that Assumptions~\ref{ass1} and
\ref{assB1} are satisfied. For $1\le l\le w$, let $\omega_{h+l-1,i'}$
be any hyperinterval that may dominate other intervals in
the set $\psi_h$ during the algorithm's execution. Assume that
$\omega_{h+l-1,i'}\subseteq\psi_h(1)^*$. Then, the loss of the LOGO
algorithm is bounded as
\[
r_n\le c\exp\biggl(-\min\biggl(\sqrt{n}\frac1{w'C}-2,
w'\sqrt{n}-w\biggr)\frac{w}D\ln\frac1{\gb}\biggr)
\]
where $\gb=3^{-\alpha}$ and $c=b3^\alpha D^{\alpha/p}$. Here, $w'=1$
if we set the parameter as $\hmax(n)=\sqrt{n}-w$. On the other hand, $w'=w$
if we set the parameter as $\hmax(n)=w\sqrt{n}-w$.
\end{theorem}

\begin{proof}
The statement of Lemma~\ref{lem2} is modified as
\[
n\ge\bbfloor{\frac{\hmax(n)+w}{w}}\sum_{k=0}^K\biggl(\abs{\psi_{kw}(1)^*}+\sum_{l=1}^{w-1}\abs{\psi_{kw}(1)^*}\biggr).
\]
The statement of Theorem~\ref{the1} is modified as
\begin{align*}
n&\le
C\bbfloor{\frac{\hmax(n)+w}{w}}\sum_{k=0}^{\floor{h/w}}(w\delta(kw)^{-d}),\\
r_n&\le\delta\bigl(\min(w\floor{h(n)/w}-w,w\floor{\hmax(n)/w})\bigr).
\end{align*}
Then, we can follow the proof of Theorem~\ref{the2} and
Corollary~\ref{cor1}, obtaining
\[
\bbfloor{\frac{h(n)}{w}}\ge\frac{n}C\frac{1}{\hmax(n)+w}-1.
\]
With $\hmax(n)=\sqrt{n}-w$, from the modified statement of
Theorem~\ref{the1}, we obtain the statement of this theorem.
\end{proof}

As Theorem~\ref{the3} makes a strong assumption to eliminate the
negative effect of the local orientation in the bound, increasing $w$
always improves the loss bound in the theorem when $n$ is sufficiently
large. This may seem to be overly optimistic, but we show an instance
of this case in our experiment.

\begin{figure}[t!]
\captionsetup[subfigure]{labelformat=empty}
\centering
\begin{subfigure}[b]{0.45\textwidth}
\hair2pt
\labellist
\pinlabel $\alpha$ [r] at 18 90
\pinlabel $w$ [t] at 115 0
\endlabellist
\includegraphics[width=.98\textwidth]{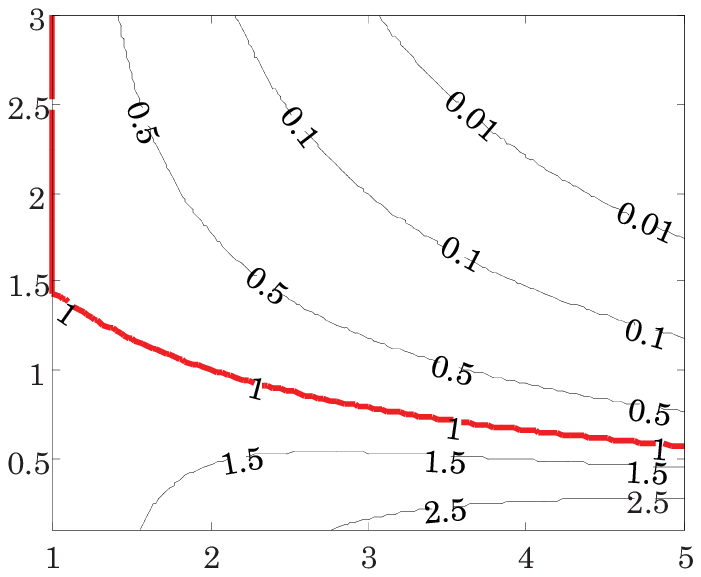}
\vskip5mm
\caption{(a) Pessimistic $w^2(\gb^{-1}-1)/(\gb^{-w}-1)$}
\end{subfigure}
\hfill
\begin{subfigure}[b]{0.45\textwidth}
\hair2pt
\labellist
\pinlabel $\alpha$ [r] at 1 90
\pinlabel $w$ [t] at 100 0
\endlabellist
\includegraphics[width=0.9\textwidth]{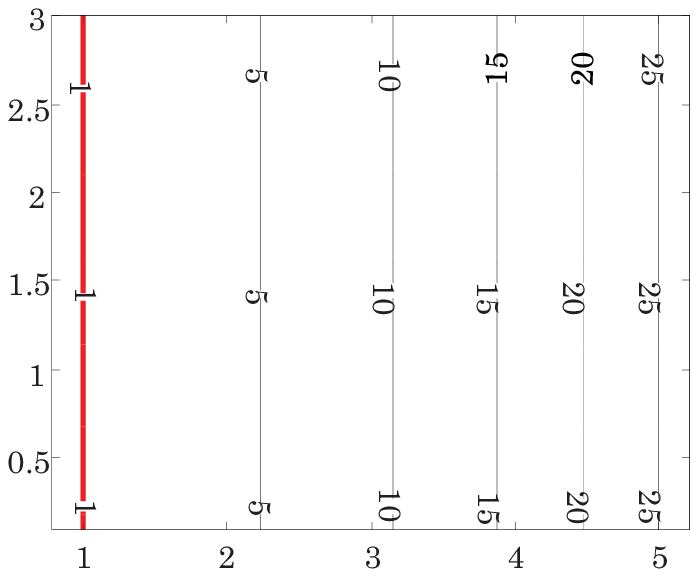}
\vskip5mm
\caption{(b) Optimistic $w^2$}
\end{subfigure}
\caption{Effect of local bias $w$ on loss bound with practical
setting. The real effect would exist somewhere in-between.}\label{fig4}
\end{figure}

More realistically, the effect of $w$ with large $n$ would exist somewhere
between the left and the right diagrams in Figure~\ref{fig4}. As in
the previous figures, the (red) bold line with label~$1$ is where
$w$ has no effect on the bound, the area with labels greater than
one is where $w$ improves the bound, and the area with labels less
than one is where $w$ diminishes the bound. The left diagram shows
the effect of $w$ in the worst case of Theorem~\ref{the2} by plotting
$w^2(\gb^{-1}-1)/(\gb^{-w}-1)$ with $\gb=3^{-\alpha}$. The reason why
plotting $w^2(\gb^{-1}-1)/(\gb^{-w}-1)$ represents the worst case is
discussed in the previous section. The right diagram presents the effect
of $w$ in the best case of Theorem~\ref{the2} or Theorem~\ref{the3}
by simply plotting $w^2$. Notice that in both Theorem~\ref{the2} and
Theorem~\ref{the3}, the best scenario for the effect of $w$ is when we
use $\hmax(n)=w\sqrt{n}-w$ and the second element of the min dominates
the bound. In this case, the coefficient of $\sqrt{n}$ is $w^2$, which
is the effect of $w$ on the bound when $n$ is large enough to ignore
the other term.

In conclusion, we showed that the LOGO algorithm provides a stretched
exponential bound on the loss with the algorithm's division strategy,
which is likely more practical than the one used in the analysis of the
SOO algorithm, and with the parameter setting $\hmax(n)=\sqrt{n}-w$
or $\hmax(n)=w\sqrt{n}-w$. We also discussed how the local bias $w$
affects the loss bound. Based on these results, we use the LOGO algorithm
in the following experiments.

\section{Experimental Results}\label{sec2.4}

In this section, we test the LOGO algorithm with a series of
experiments. In the main part of the experiments, we compared the LOGO
algorithm with its direct predecessor, the SOO algorithm \cite{mun11}
and its latest powerful variant, the Bayesian Multi-Scale Optimistic
Optimization (BaMSOO) algorithm \cite{wan14}. The BaMSOO algorithm
combines the SOO algorithm with a Gaussian Process (GP) to leverage
the GP's estimation of the upper confidence bound. It was shown to
outperform the traditional Bayesian Optimization method that uses a
GP and the DIRECT algorithm \cite{wan14}. Accordingly, we omitted the
comparison with the traditional Bayesian Optimization method.
We also compare LOGO with popular heuristics, simulated annealing (SA) and genetic algorithm
(GA) (see \citeR{rus09} for a brief introduction). 

In the experiments, we rescaled the domains to the $[0, 1]^D$
hypercube. We used the same division process for SOO, BaMSOO and LOGO, which is the one presented in Section~\ref{sec2.2}
and proven to provide stretched exponential bounds on the loss in
Section~\ref{sec2.3.2}. Previous algorithms have also been used with this
division process in experiments \cite{jon93,gab01,mun13,wan14}. For
the SOO and LOGO algorithms, we set $\hmax(n) =w\sqrt{n}-w$. This
setting guarantees a stretched exponential bound for LOGO, as proven in
Section~\ref{sec2.3.2}, and for SOO \cite{mun13}. For the
LOGO algorithm, we used a simple adaptive procedure to set the parameter
$w$. Let $f(x_i^+)$ be the best value found thus far in the end of
iteration $i$. Let $W = \{3, 4, 5, 6, 8, 30\}$. The algorithm begins with
$w = W_1 = 3$. At the end of iteration $i$, the algorithm set $w = W_k$
with $k=\min(j+1,6)$ if $f(x_i^+)\ge f(x_{i-1}^+)$, and $k=\max(j-1,1)$
otherwise, where $W_j$ is the previous parameter value $w$ before this
adjustment occurs. Intuitively, this adaptive procedure is to encourage
the algorithm to be locally biased when it seems to be making progress,
forcing it to explore a more global region when this does not seem to
be the case. Although the values in the set $W = \{3, 4, 5, 6, 8, 30\}$
are arbitrary, this simple setting was used in all the experiments in this
paper, including the real-world application in Section~\ref{sec3.4}. The
results demonstrate the robustness of this setting. As discussed later,
a future work would be to replace this simple adaptive mechanism to
improve the performance of the proposed algorithm. For the BaMSOO
algorithm, the previous work of \citeA{wan14} used a pair of a good kernel
and hyperparameters that were handpicked for each test function. In our
experiments, we assumed that such a handpicking procedure was unavailable,
which is typically the case in practice. We tested several pairs of a
kernel and hyperparameters; however, none of the pairs performed robustly
well for all the test functions (e.g., one pair performed well for one
test function, although not others). Thus, we used the empirical Bayes
method to adaptively update the hyperparameters\footnote{We implemented BaMSOO by ourselves to use the empirical Bayes method, which was not done in the original implementation. The original implementation of BaMSOO was not available for us  as well.  }. We selected the isotropic Matern kernel with  $\nu=5/2$, which is given
by $\kappa(x,x')=g(\sqrt{5\norm{x-x'}^2/l})$, where the function $g$
is defined to be $g(z)=\sigma^2(1+z+z^3/3)$. The hyperparameters were
initialized to $\sigma=1$ and $l=0.25$. We updated the hyperparameters
every iteration until $1{,}000$ function evaluations were executed and
then per $1{,}000$ iterations afterward (to reduce the computational
cost).
For SA and GA, we used the same settings as those of the Matlab standard subroutines \texttt{simulannealbnd} and \texttt{ga}, except that we specified the domain bounds.  
   
\begin{table} 
\centering
\begin{scriptsize}
\renewcommand{\tabcolsep}{1.45pt}
\renewcommand{\arraystretch}{1.4}
\begin{tabular}{|l|c|c|c|c|c|c|c|c|c|c|c|}
\hline
\multicolumn{1}{|c|}{\multirow{2}{*}{$f$}} & \multirow{2}{*}{$D$} & \multirow{2}{*}{$\W$} &
\multicolumn{3}{|c|}{SOO} & 
\multicolumn{3}{|c|}{BaMSOO} & 
\multicolumn{3}{|c|}{LOGO}\\
\cline{4-12}
&&&$N$ & $\mathit{Time}\,(s)$ & $\mathit{Error}$ & 
$N$ & $\mathit{Time}\,(s)$ & $\mathit{Error}$ &
$N$ & $\mathit{Time}\,(s)$ & $\mathit{Error}$\\
\hline
Sin~1 & $1$ & $[0,1]$ & $57$ & $5.3\,\E{-}02$ & $2.3\,\E{-}06$ &
$30$ & $2.0\,\E{+}00$ & $2.3\,\E{-}06$ &
$17$ & $3.9\,\E{-}02$ & $2.3\,\E{-}06$\\
\hline
Sin~2 & $2$ & $[0,1]^2$ & $271$ & $1.7\,\E{-}01$ & $4.6\,\E{-}06$ &
$181$ & $7.5\,\E{+}00$ & $4.6\,\E{-}06$ &
$45$ & $5.4\,\E{-}02$ & $4.6\,\E{-}06$\\
\hline
Peaks & $2$ & $[-3,3]^2$ & $141$ & $1.0\,\E{-}01$ & $9.0\,\E{-}05$ &
$37$ & $3.5\,\E{+}00$ & $9.0\,\E{-}05$ &
$35$ & $6.1\,\E{-}02$ & $9.0\,\E{-}05$\\
\hline
Branin & $2$ & $[-5,10]\times[0,15]$ & $339$ & $2.1\,\E{-}01$ & $9.0\,\E{-}05$ &
$121$ & $8.1\,\E{+}00$ & $9.0\,\E{-}05$ &
$85$ & $7.0\,\E{-}02$ & $8.7\,\E{-}05$\\
\hline
Rosenbrock~2 & $2$ & $[-5,10]^2$ & $491$ & $3.1\,\E{-}01$ & $9.7\,\E{-}06$ &
\gra${>}4000$ & \gra$5.8\,\E{+}04$ & \gra$5.5\,\E{-}03$ &
$137$ & $1.3\,\E{-}01$ & $9.7\,\E{-}06$\\
\hline
Hartman~3 & $3$ & $[0,1]^3$ & $359$ & $2.3\,\E{-}01$ & $7.91\,\E{-}05$ &
$126$ & $8.9\,\E{+}00$ & $7.9\,\E{-}05$ &
$65$ & $7.1\,\E{-}02$ & $5.1\,\E{-}05$\\
\hline
Shekel~5 & $4$ & $[0,10]^4$ & $1101$ & $6.6\,\E{-}01$ & $8.4\,\E{-}05$ &
$316$ & $3.1\,\E{+}01$ & $8.4\,\E{-}05$ &
$157$ & $1.2\,\E{-}01$ & $8.4\,\E{-}05$\\
\hline
Shekel~7 & $4$ & $[0,10]^4$ & $1117$ & $7.1\,\E{-}01$ & $9.4\,\E{-}05$ &
$95$ & $1.2\,\E{+}01$ & $9.4\,\E{-}05$ &
$157$ & $1.2\,\E{-}01$ & $9.4\,\E{-}05$\\
\hline
Shekel~10 & $4$ & $[0,10]^4$ & $1117$ & $6.4\,\E{-}0.1$ & $9.68\,\E{-}05$ &
\gra${>}4000$ & \gra$4.5\,\E{+}04$ & \gra$8.1\,\E{+}00$ &
$197$ & $1.5\,\E{-}01$ & $9.7\,\E{-}05$\\
\hline
Hartman~6 & $6$ & $[0,1]^6$ & $1759$ & $1.2\,\E{+}00$ & $7.51\,\E{-}05$ &
\gra${>}4000$ & \gra$4.0\,\E{+}04$ & \gra$2.3\,\E{-}03$ &
$161$ & $1.3\,\E{-}01$ & $6.8\,\E{-}05$\\
\hline
Rosenbrock~10 & $10$ & $[-5,10]^{10}$ & \gra${>}8000$ & \gra$7.8\,\E{+}00$ & \gra$3.83\,\E{-}03$ &
\gra${>}8000$ & \gra$5.8\,\E{+}04$ & \gra$9.6\,\E{+}00$ &
$1793$ & $1.7\,\E{+}00$ & $4.8\,\E{-}05$\\
\hline
\end{tabular}
\end{scriptsize}
\caption{Performance comparison in terms of the number of evaluations
($N$) and CPU time ($\mathit{Time}$) to achieve $\mathit{Error} <
10^{-4}$. The grayed cells indicate the experiments where we could not
achieve $\mathit{Error} < 10^{-4}$ even with a large number of function
evaluations ($4000$ or $8000$).}\label{tab1}
\end{table}

Table~\ref{tab1} shows the results of the comparison with $11$ test
functions in terms of the number of evaluations and CPU time to achieve
a small error. The first two test functions, Sin~1 and Sin~2, were used
to test the SOO algorithm \cite{mun13}, and have the form
$f(x) = (\sin(13x)\sin(27x) + 1) / 2$  and  $f (x_1, x_2) = f (x_1) f
(x_2)$ respectively. The form of the third function, Peaks, is given
in Equation (16) and illustrated in Figure 2 of \citeauthor{mcd07}'s paper \citeyear{mcd07}. The rest
of the test functions are common benchmarks in global optimization
literature; \citeauthor{sur13} present detailed information about the functions \citeyear{sur13}. In the table, $\mathit{Time}\,(s)$ indicates CPU time
in second and $\mathit{Error}$ is defined as
\[
\mathit{Error}=
\begin{cases}
\abs{(f(x^*)-f(x^+))/f(x^*)} & \text{if }f(x^*)\neq0,\\
\abs{f(x^*)-f(x^+)} & \text{otherwise}.\\
\end{cases}
\]
In the table, $N = 2n$ is the number of function evaluations needed to
achieve $\mathit{Error} < 10^{-4}$, where $n$ is the total number of
divisions and is the one used as the main measure in the analyses in
the previous sections. Here, $N$ is equal to $2n$ because of the adopted
division process. Thus, the lower the value of $N$ becomes, the better
the algorithm's performance is. We continued iterations until $4000$
function evaluations for all the functions with dimensionality less than
$10$, and $8000$ for the function with dimensionality equal to $10$.

As can be seen in Table~\ref{tab1}, the LOGO algorithm outperformed
the other algorithms. The superior performance of the LOGO algorithm
with the small number of function evaluations is attributable to its
focusing on the promising area discovered during the search. Conversely,
the SOO algorithm continues to search the global domain and tends to be
similar to a uniform grid search. The BaMSOO algorithm also follows the
tendency toward a grid search as it is based on the SOO algorithm. The
BaMSOO algorithm chooses where to divide based on the SOO algorithm;
however, it omits function evaluations when the upper coincidence bound
estimated by the GP indicates that the evaluation is not likely to be
beneficial. Although this mechanism of the BaMSOO algorithm seems to be
beneficial to reduce the number of function evaluations in some cases,
it has two serious disadvantages. The first disadvantage is computational
time due to the use of GP. Notice that it requires $O(N^3)$ every time
to re-compute the upper confidence bound.%
\footnote{Although there are several methods to mitigate its computational
burden by approximation, the effect of the approximation on the
performance of the BaMSOO algorithm is unknown and left to a future work.}
A more serious disadvantage is the possibility of not determining a
solution at all. From Table~\ref{tab1}, we can see that BaMSOO improves
the performance of SOO in $7/11$ cases; however, it \emph{severely}
degrades the performance in $4/11$ cases. Moreover, not only may BaMSOO
reduce the performance but also it may not guarantee convergence even in
the limit \emph{in practice}. The BaMSOO algorithm reduces the number of
function evaluations by relying on the estimation of the upper confidence
bound. However, the estimation can be wrong, and if it is wrong, it may
never explore the region where the global optimizer exists.  Notice
that these limitations are not unique to BaMSOO but also apply to
many GP-based optimization methods. In terms of the first limitation
(computational time), BaMSOO is a significant improvement when compared
to other traditional GP-based optimization methods \cite{wan14}.

Although the LOGO algorithm has a bias toward local search, it maintains
a strong theoretical guarantee, similar to the SOO algorithm, as
proven in the previous sections. In terms of theoretical guarantee,
the SOO algorithm and the LOGO algorithm share a similar rate on the
loss bound and base their analyses on the same set of assumptions that
hold in practice. On the other hand, the BaMSOO algorithm has a worse
rate on the loss bound (an asymptotic loss of the order $n^{-(1-\ops)/d})$
and its bound only applies to a restricted class of a metric $\sem$ (the
Euclidean norm to the power $\alpha = \{1,2\}$). It also requires several
additional assumptions to guarantee the bound. Some of the additional
assumptions would be impractical, particularly
the assumption of the objective function being always well-captured by the
GP with a chosen kernel and hyperparameters. As discussed above, this
assumption would cause BaMSOO to not only lose the loss bound but also
the consistency guarantee (i.e., convergence in the limit) in practice.

Figure~\ref{fig5} presents the performance comparison for each number
of function evaluations and Figure~\ref{fig6} plots the corresponding
computational time. In both figures, a lower plotted value along the
vertical axis indicates improved algorithm performance. For SA and GA, each figure shows the mean over 10 runs.
We report the mean of the standard deviation over time in the following. For SA, it was 1.19 (Sin 1), 1.32 (Sin 2), 0.854 (Peaks), 0.077 (Branin), 1.06 (Rosenbrock 2), 0.956 (Hartman 3), 0.412 (Shekel 5), 0.721 (Shekel 7), 1.38 (Shekel 10), 0.520 (Hartman 6), and 0.489 (Rosenbrock 10). For GA, it was 0.921 (Sin 1), 0.399 (Sin 2), 0.526 (Peaks), 0.045 (Branin), 1.27 (Rosenbrock 2), 0.493 (Hartman 3), 0.216 (Shekel 5), 0.242 (Shekel 7), 1.19 (Shekel 10), 0.994 (Hartman 6), and 0.181 (Rosenbrock 10).  \begin{figure}[t]
\centering
\begin{subfigure}[b]{0.32\textwidth}
\includegraphics[width=\textwidth]{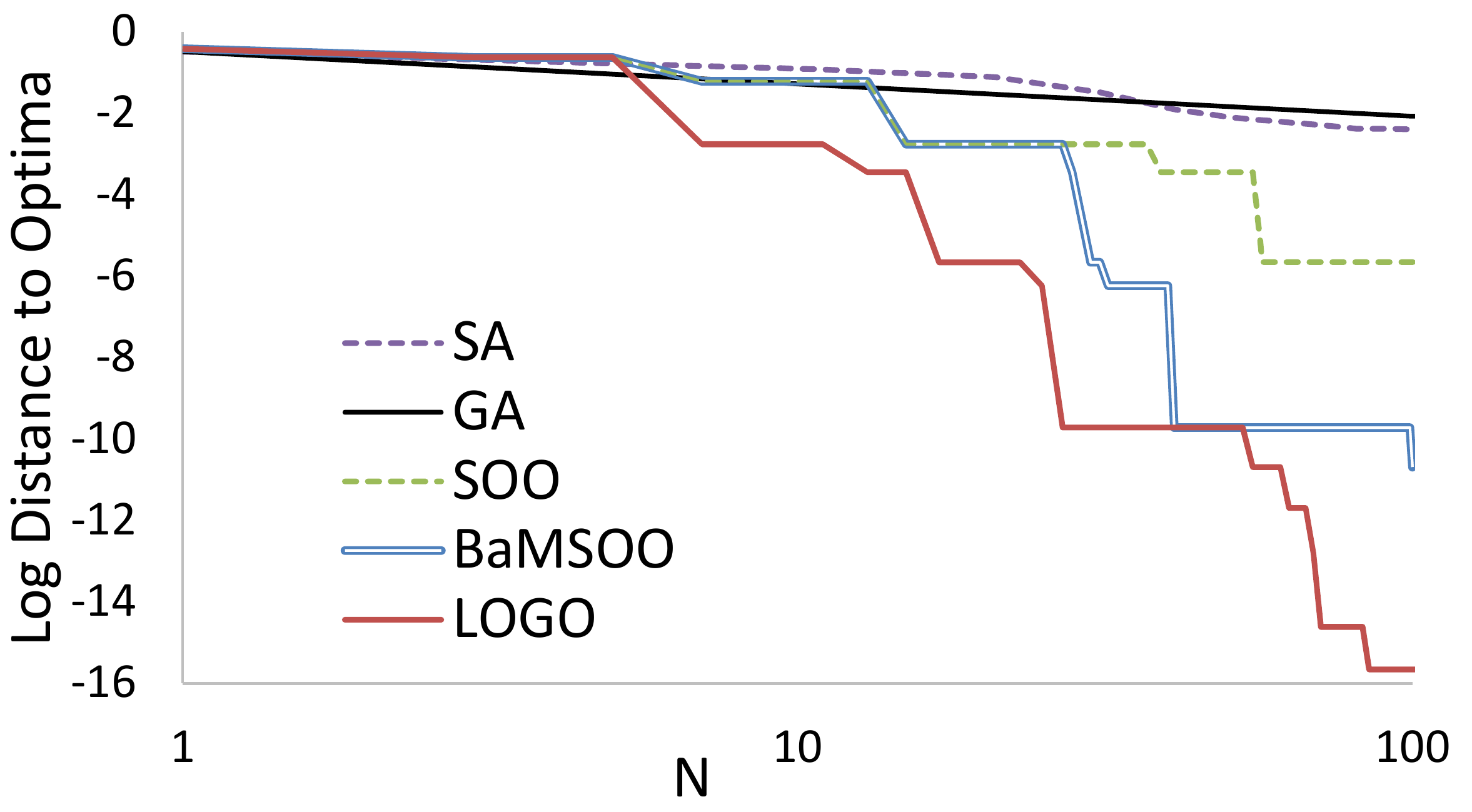}
\caption{Sin 1}
\end{subfigure}\hfill
\begin{subfigure}[b]{0.32\textwidth}
\includegraphics[width=\textwidth]{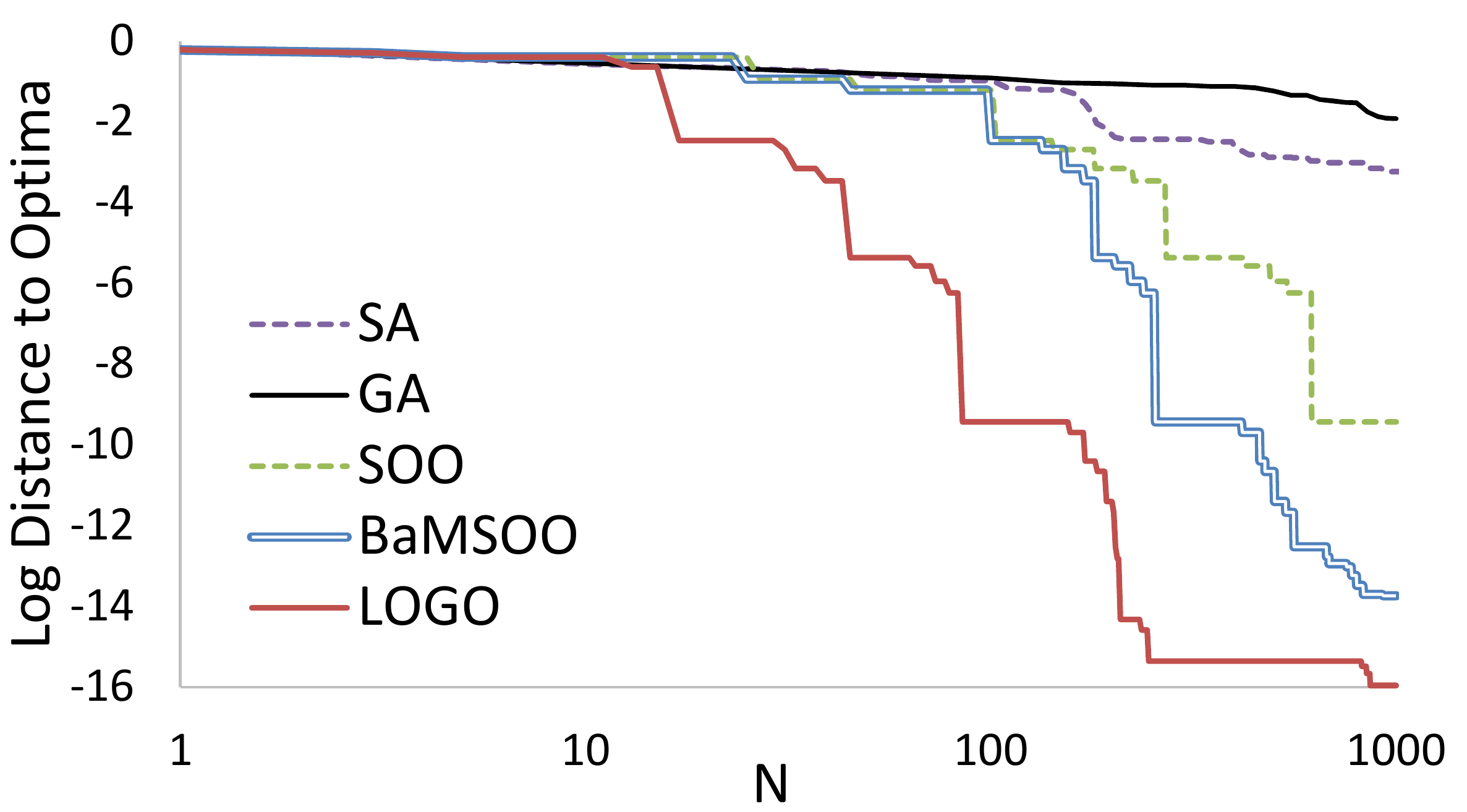}
\caption{Sin 2}
\end{subfigure}\hfill
\begin{subfigure}[b]{0.32\textwidth}
\includegraphics[width=\textwidth]{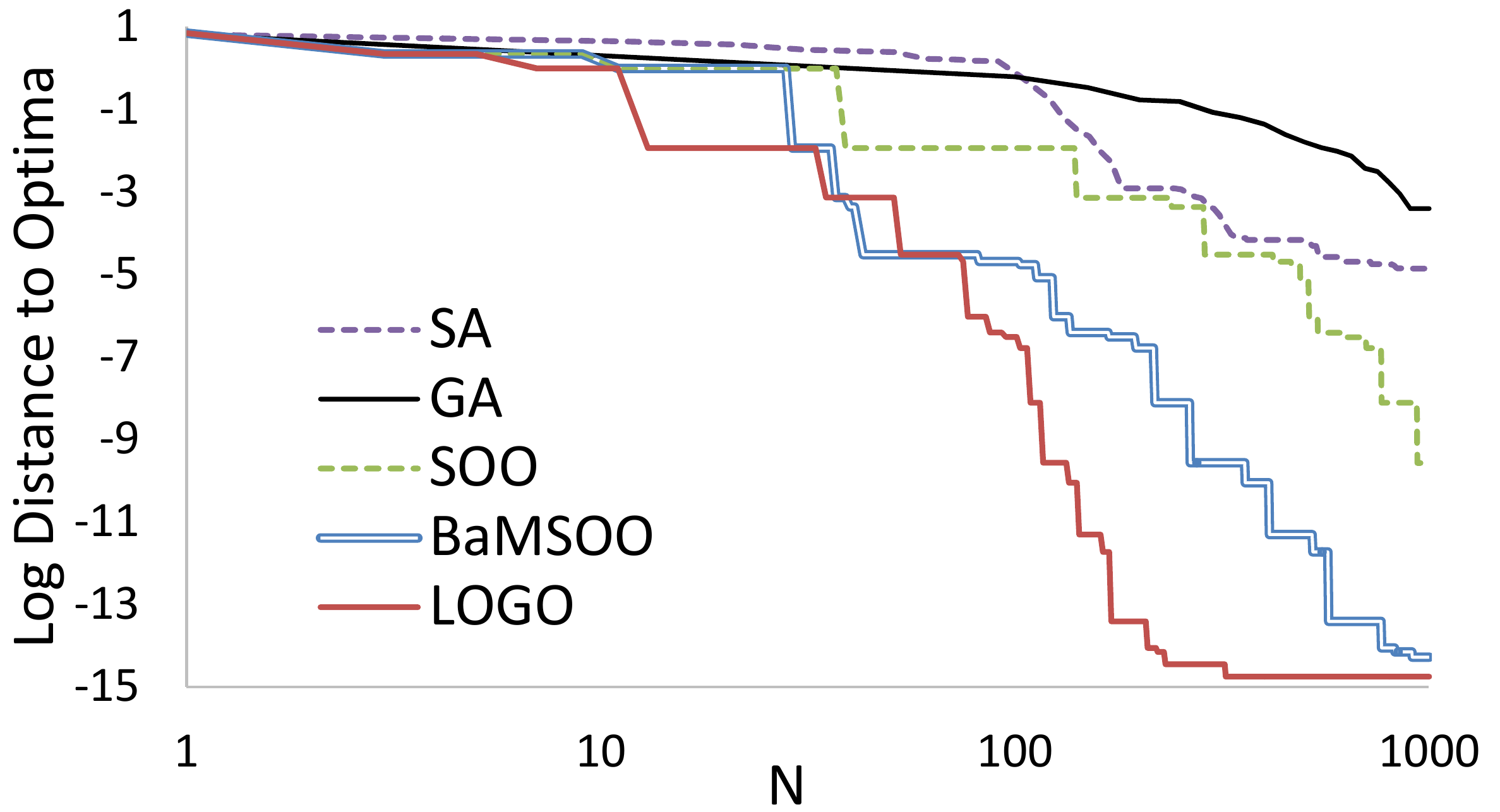}
\caption{Peaks}
\end{subfigure}
\vskip3mm
\begin{subfigure}[b]{0.32\textwidth}
\includegraphics[width=\textwidth]{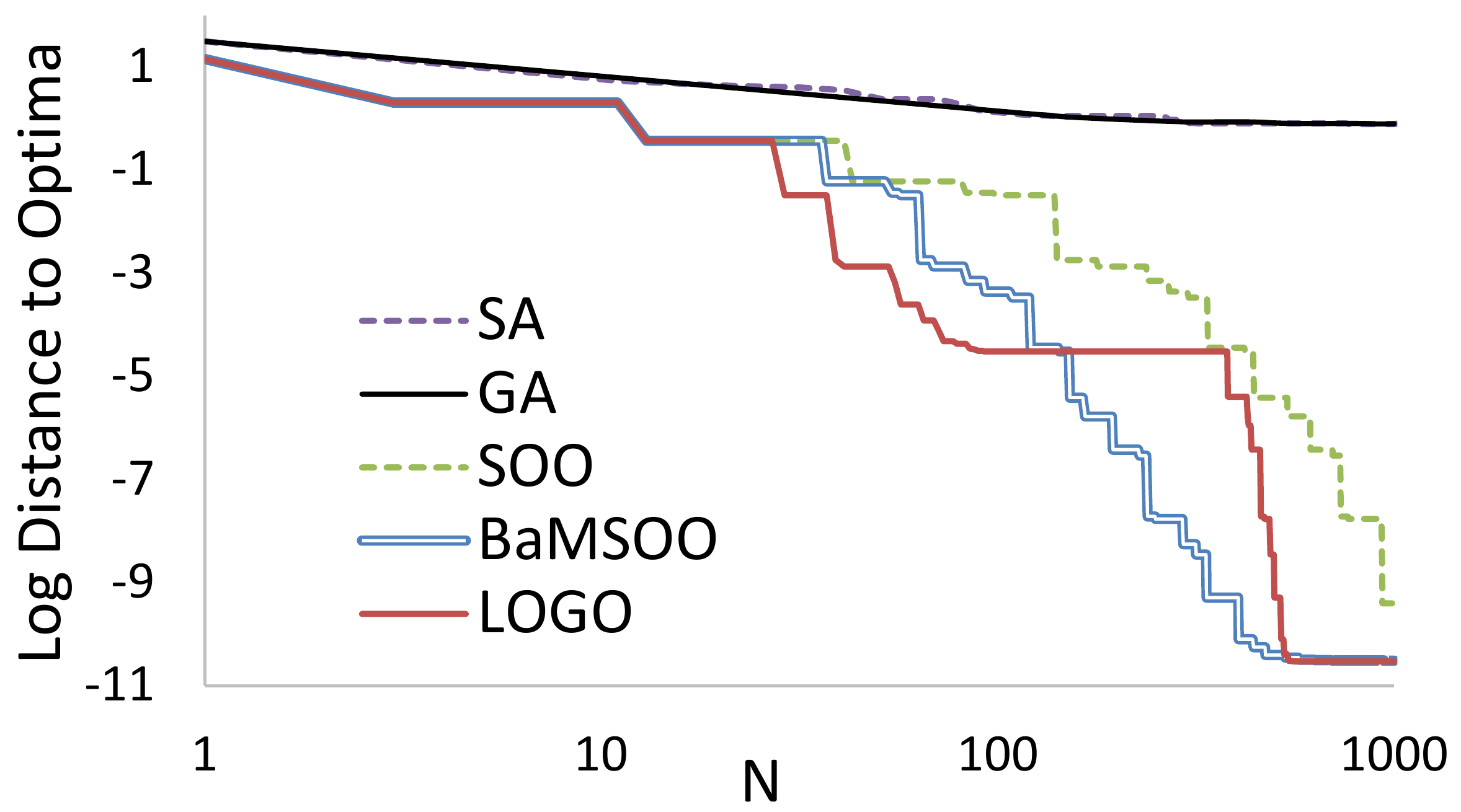}
\caption{Branin}
\end{subfigure}\hfill
\begin{subfigure}[b]{0.32\textwidth}
\includegraphics[width=\textwidth]{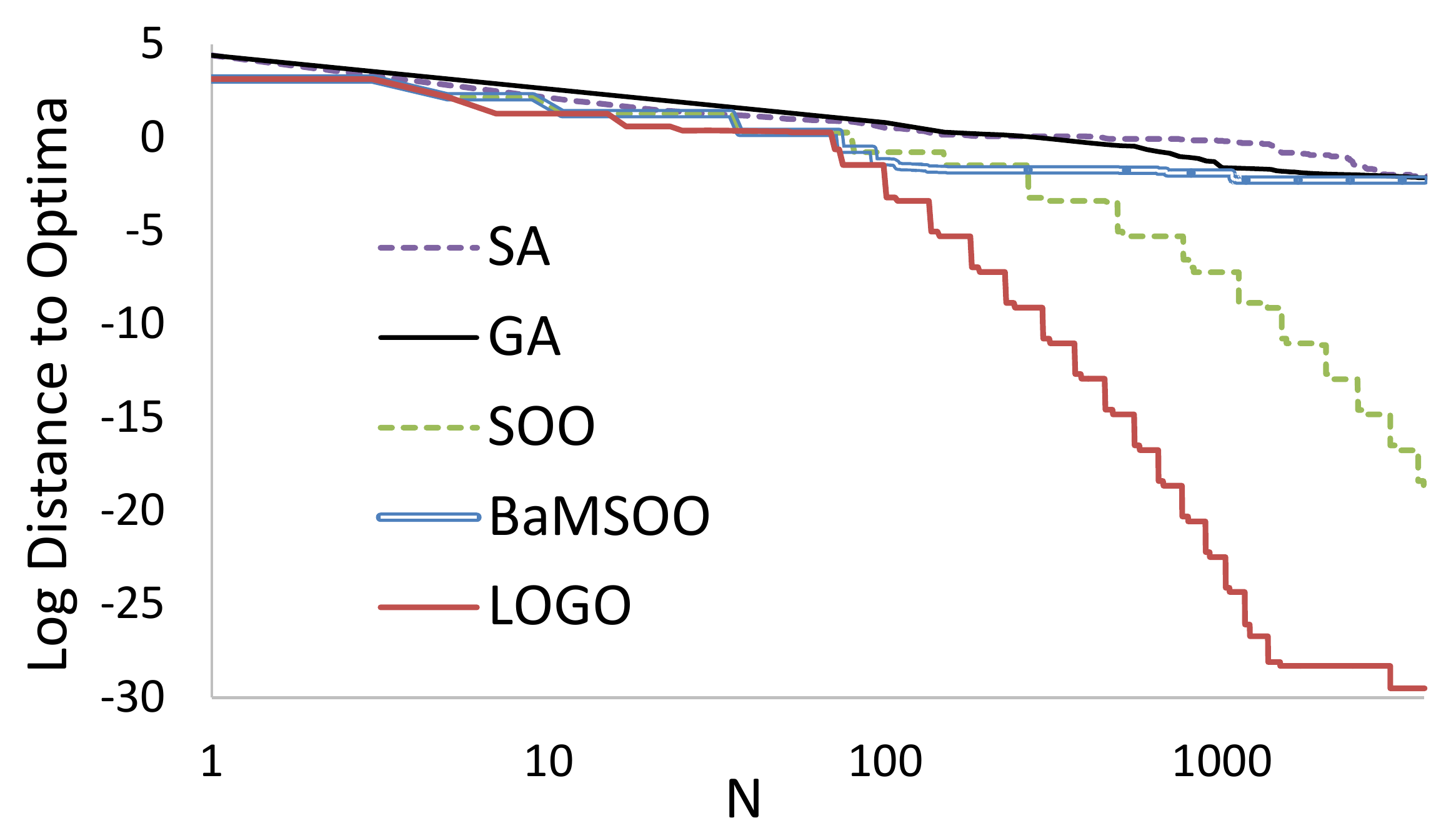}
\caption{Rosenbrock 2}
\end{subfigure}\hfill
\begin{subfigure}[b]{0.32\textwidth}
\includegraphics[width=\textwidth]{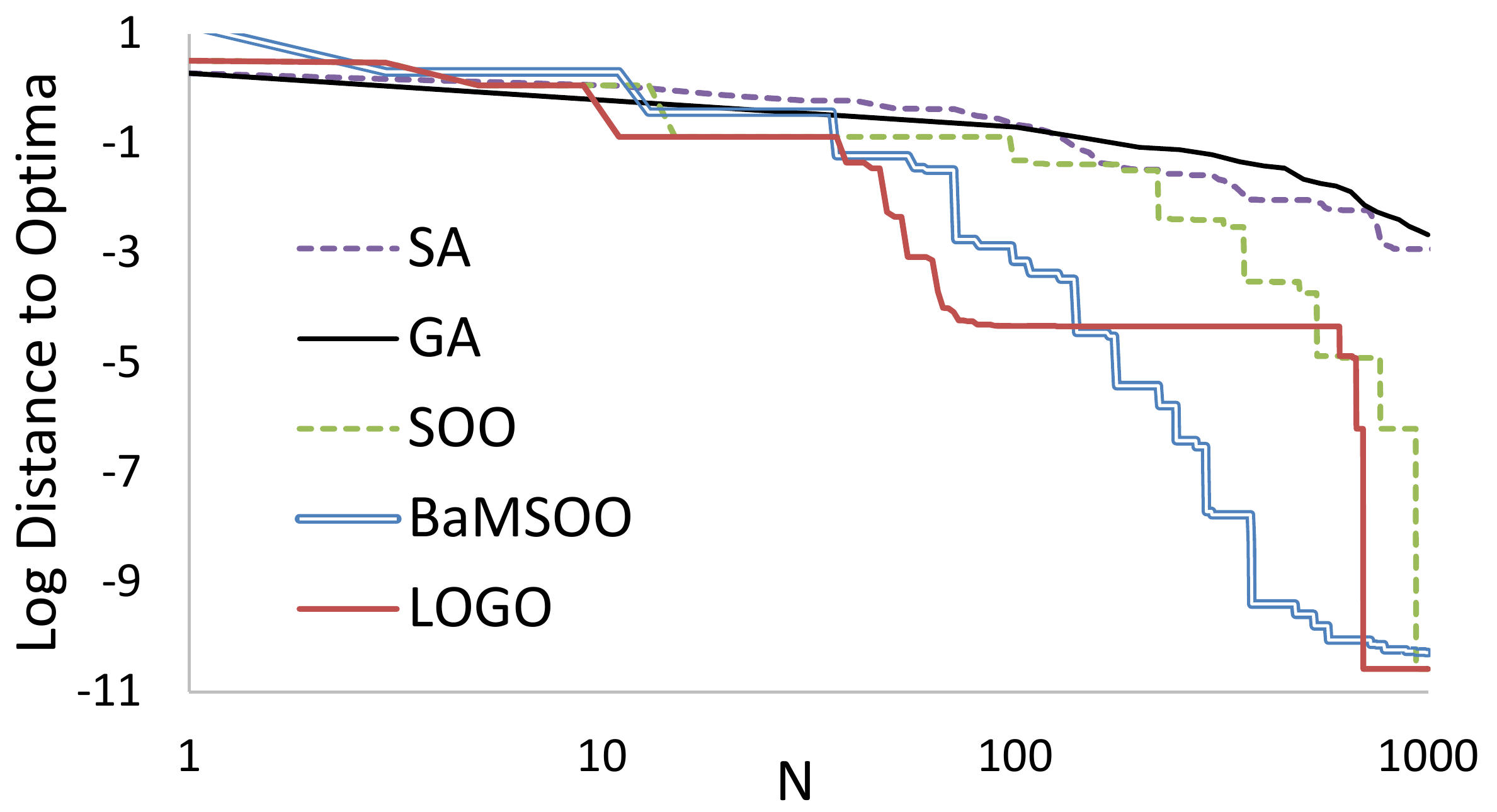}
\caption{Hartman 3}
\end{subfigure}
\vskip3mm
\begin{subfigure}[b]{0.32\textwidth}
\includegraphics[width=\textwidth]{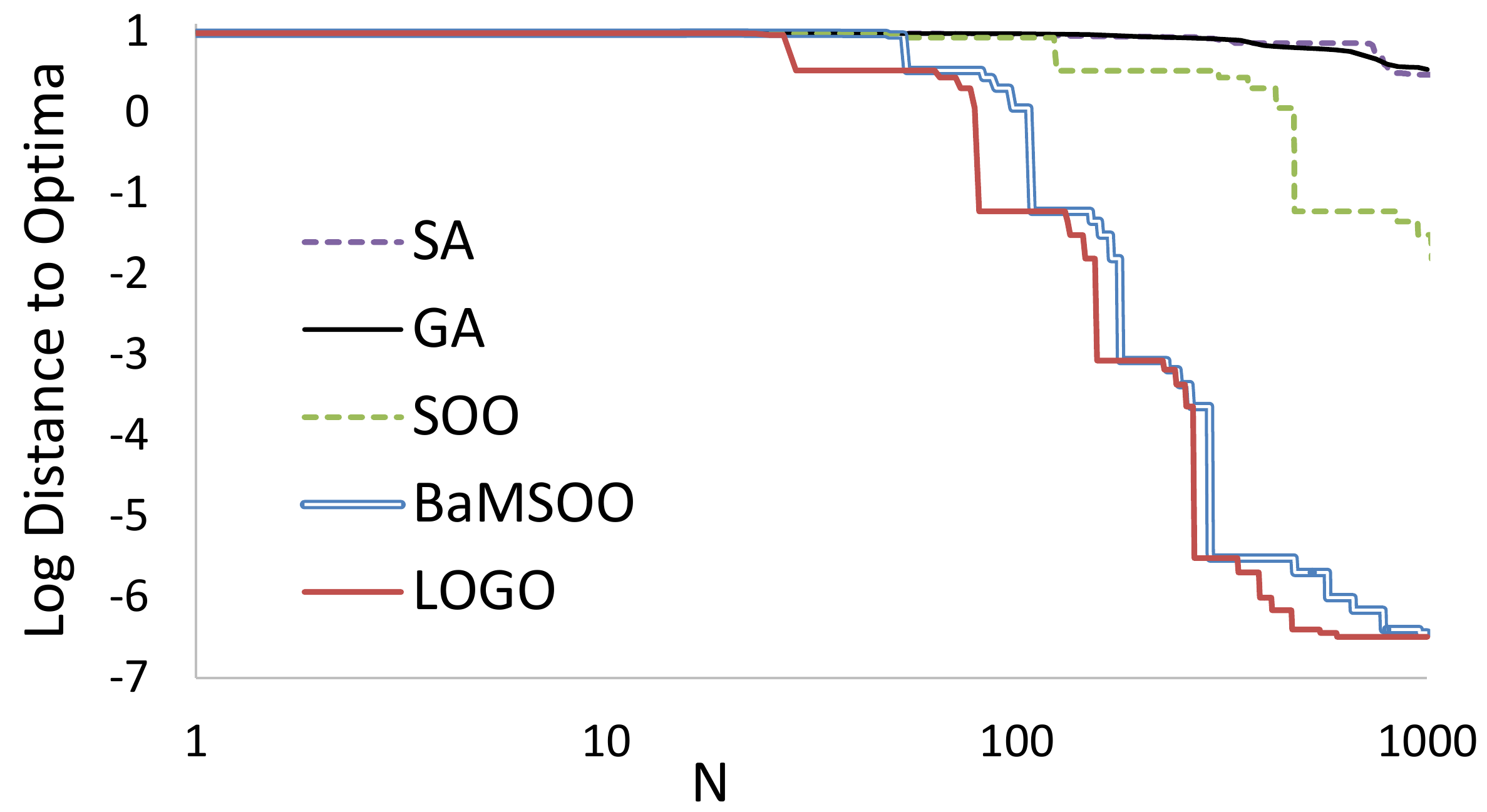}
\caption{Shekel 5}
\end{subfigure}\hfill
\begin{subfigure}[b]{0.32\textwidth}
\includegraphics[width=\textwidth]{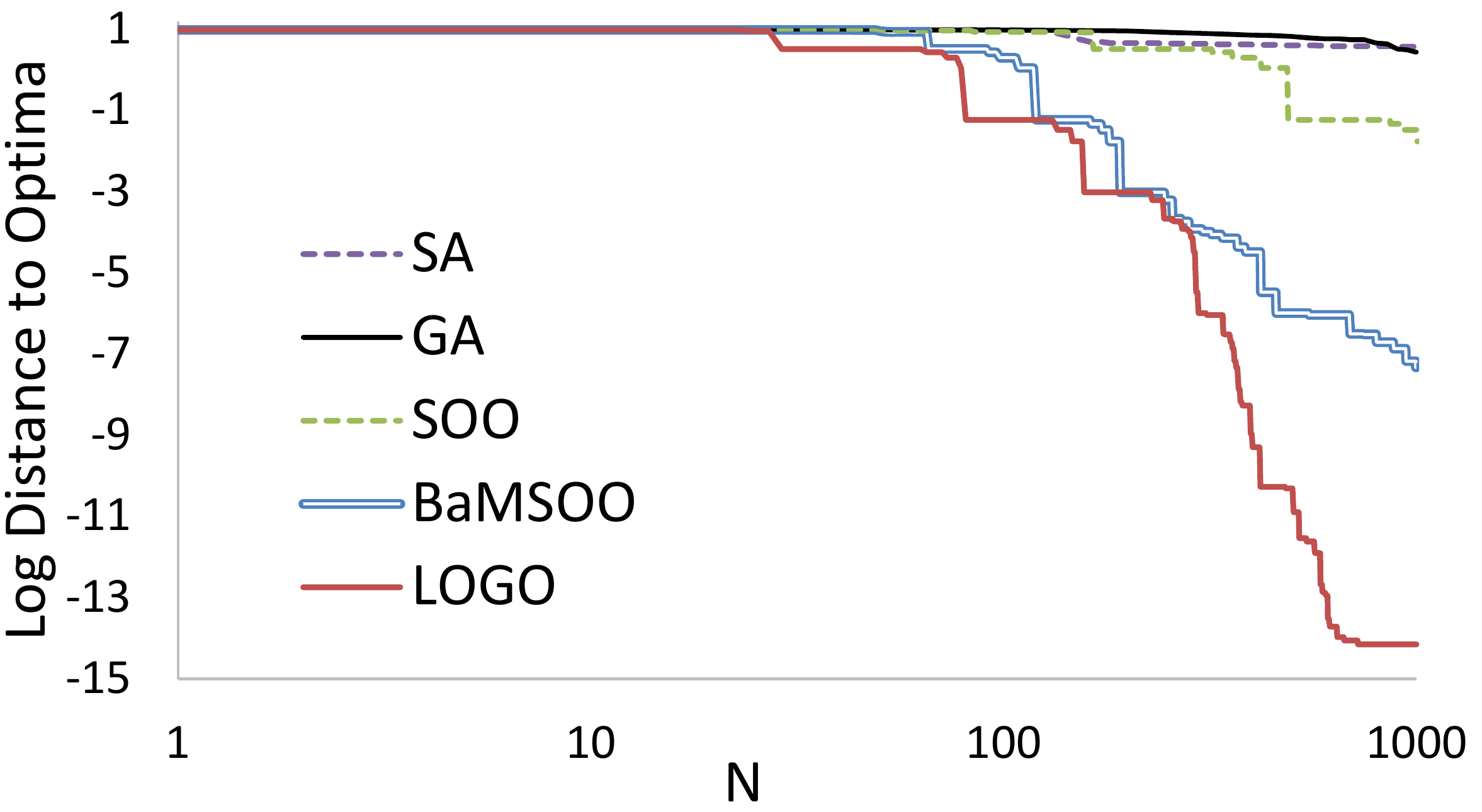}
\caption{Shekel 7}
\end{subfigure}\hfill
\begin{subfigure}[b]{0.32\textwidth}
\includegraphics[width=\textwidth]{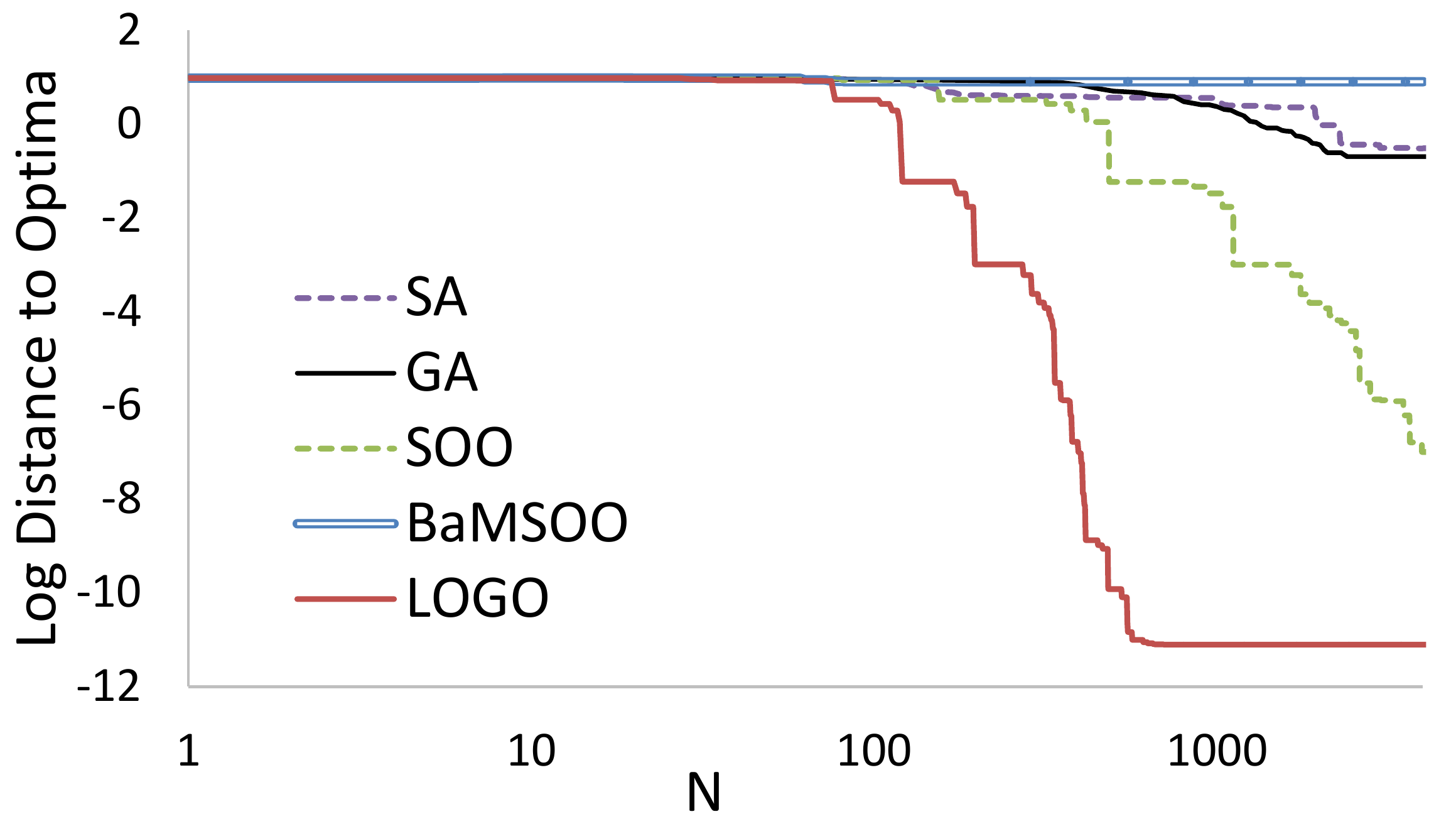}
\caption{Shekel 10}
\end{subfigure}
\vskip3mm
\begin{subfigure}[b]{0.32\textwidth}
\includegraphics[width=\textwidth]{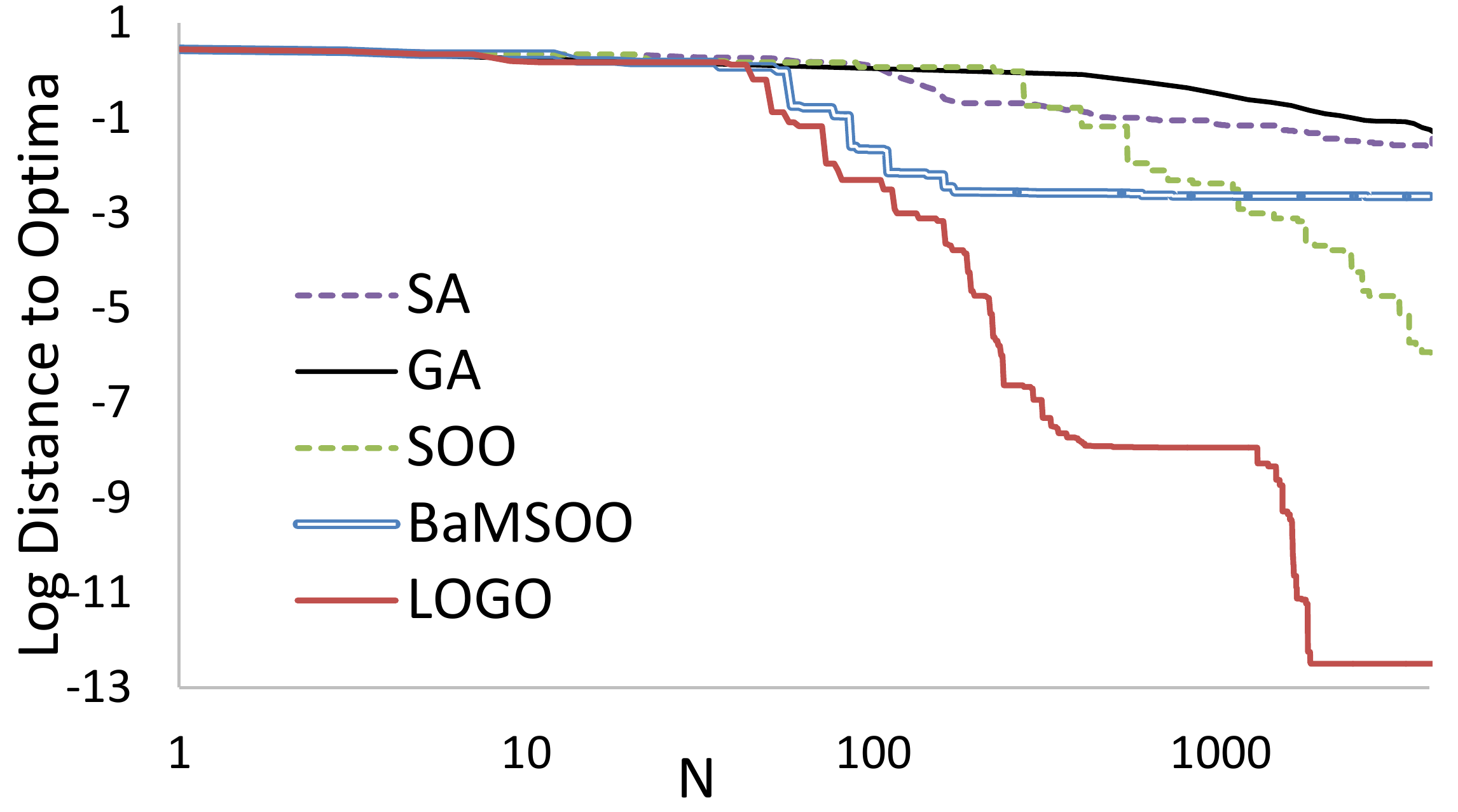}
\caption{Hartman 6}
\end{subfigure}\hfill
\begin{subfigure}[b]{0.32\textwidth}
\includegraphics[width=\textwidth]{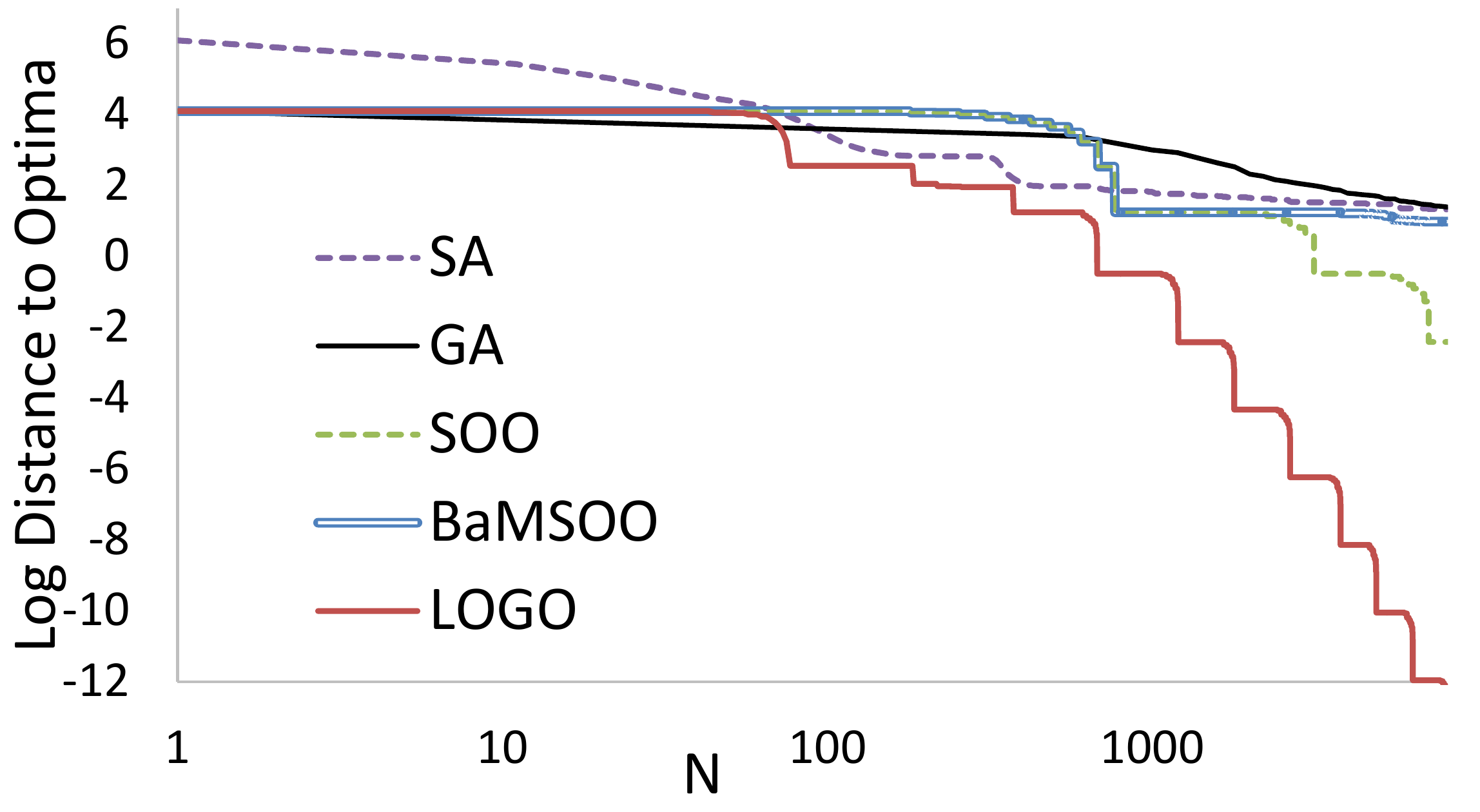}
\caption{Rosenbrock 10}
\end{subfigure}\hfill
\rule{0.32\textwidth}{0pt}
\caption{Performance comparison: the number of evaluations $N$ vs.\
the log error computed as $\log_{10}\abs{f(x^*)-f(x^+)}$. $f(x^*)$
indicates the true optimal value of the objective function and $f(x^+)$
is the best value determined by each algorithm.}\label{fig5}
\end{figure}

\begin{figure}[t]
\centering
\begin{subfigure}[b]{0.32\textwidth}
\includegraphics[width=\textwidth]{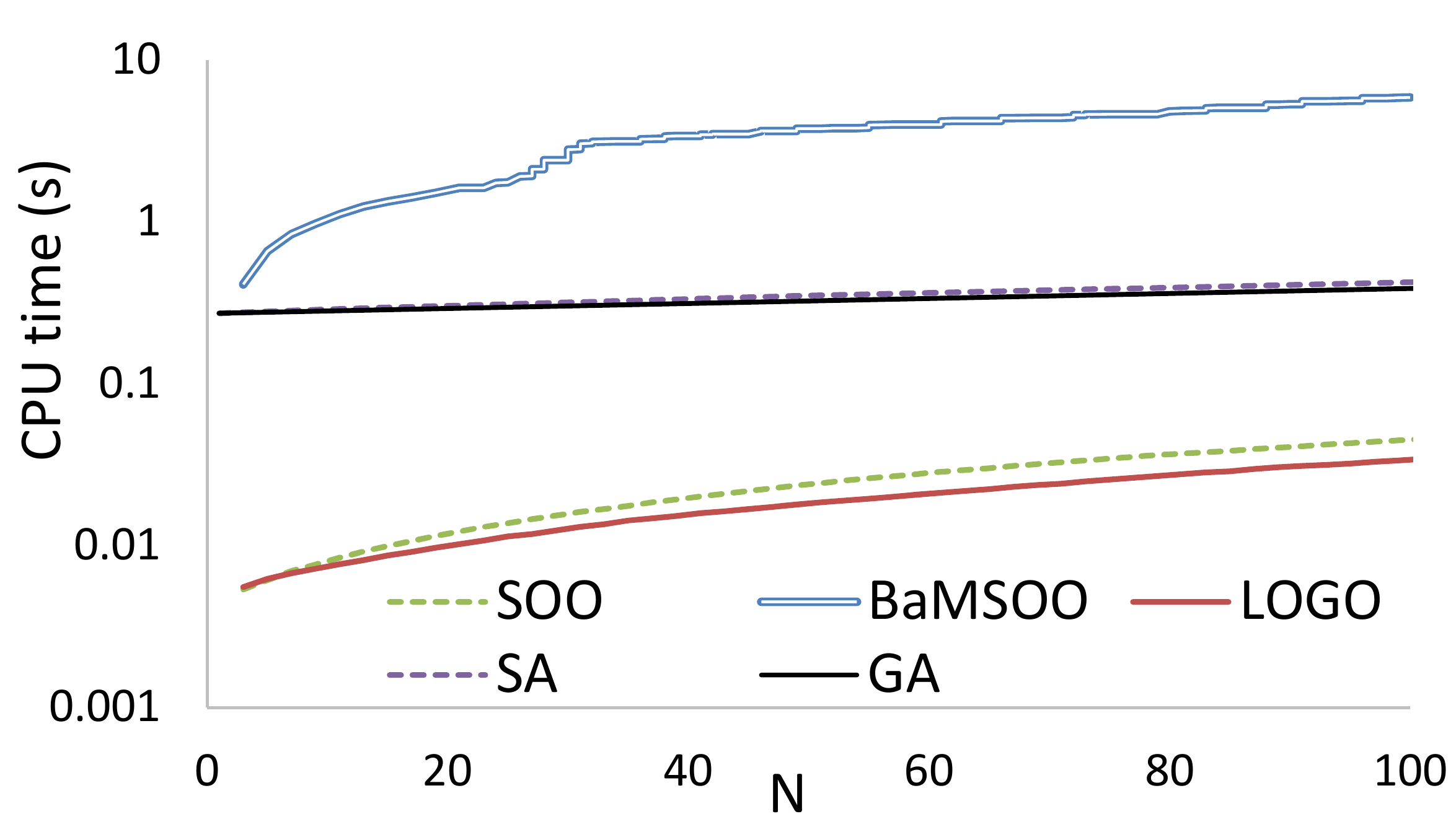}
\caption{Sin 1}
\end{subfigure}\hfill
\begin{subfigure}[b]{0.32\textwidth}
\includegraphics[width=\textwidth]{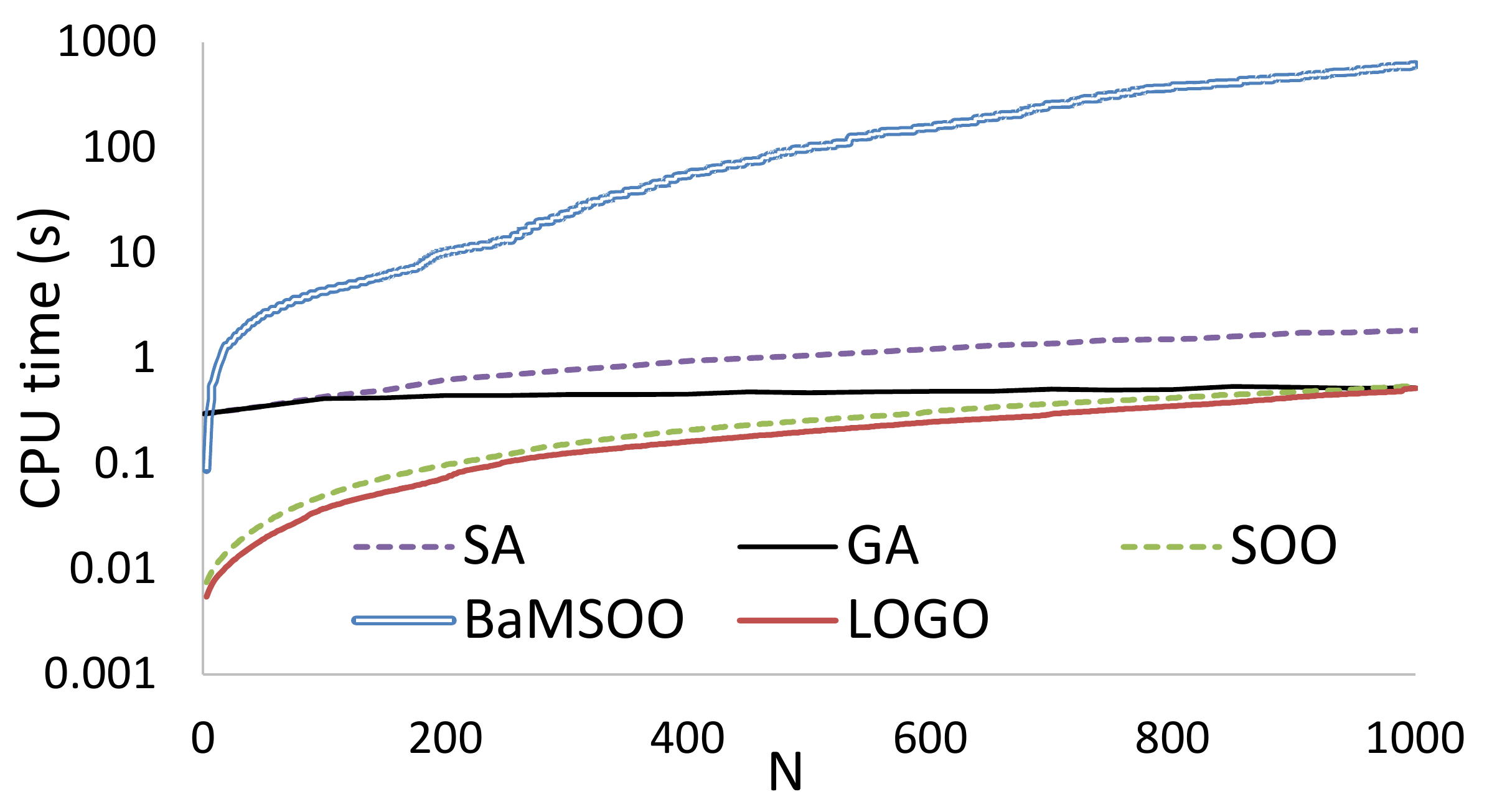}
\caption{Sin 2}
\end{subfigure}\hfill
\begin{subfigure}[b]{0.32\textwidth}
\includegraphics[width=\textwidth]{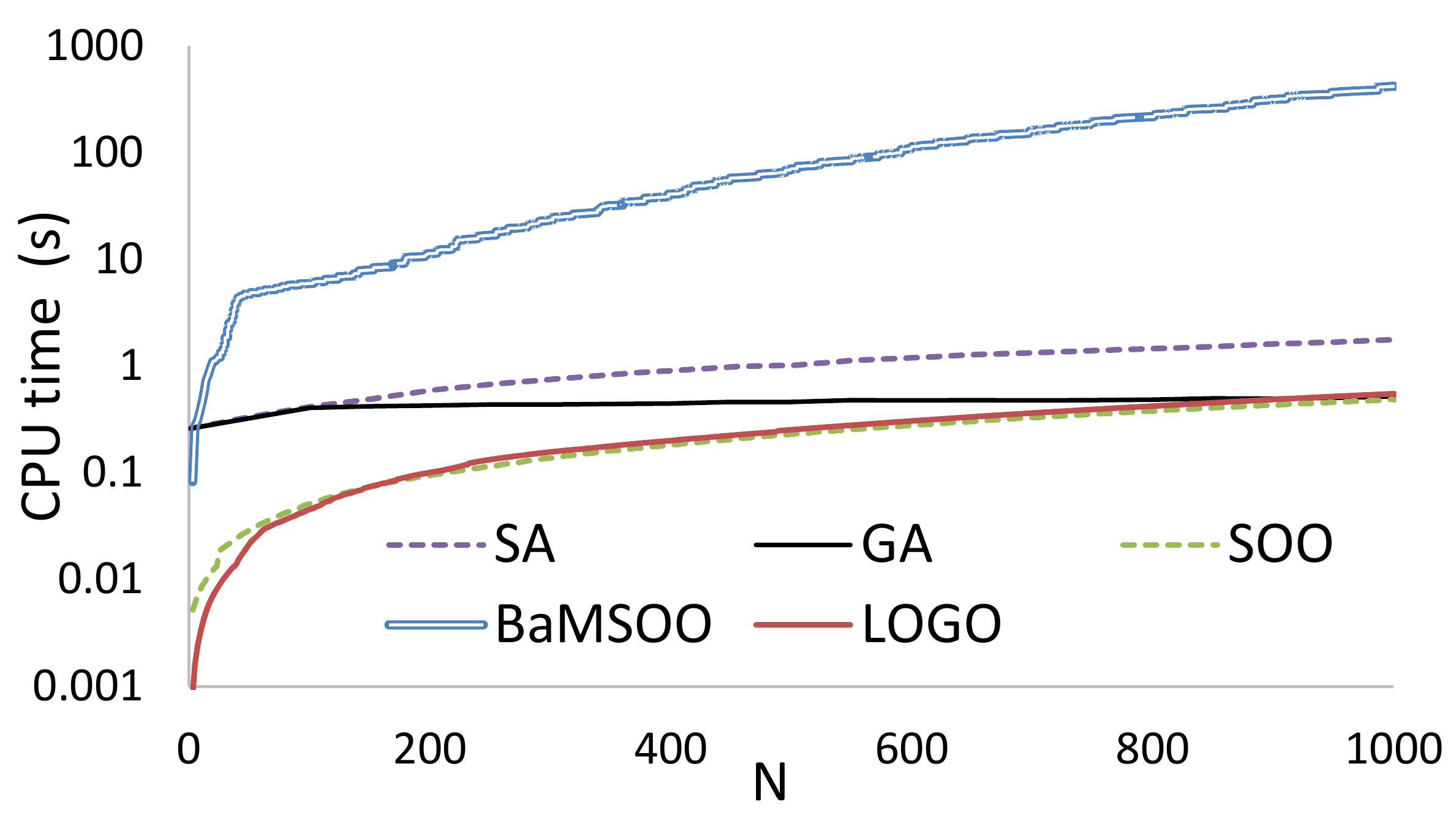}
\caption{Peaks}
\end{subfigure}
\vskip3mm
\begin{subfigure}[b]{0.32\textwidth}
\includegraphics[width=\textwidth]{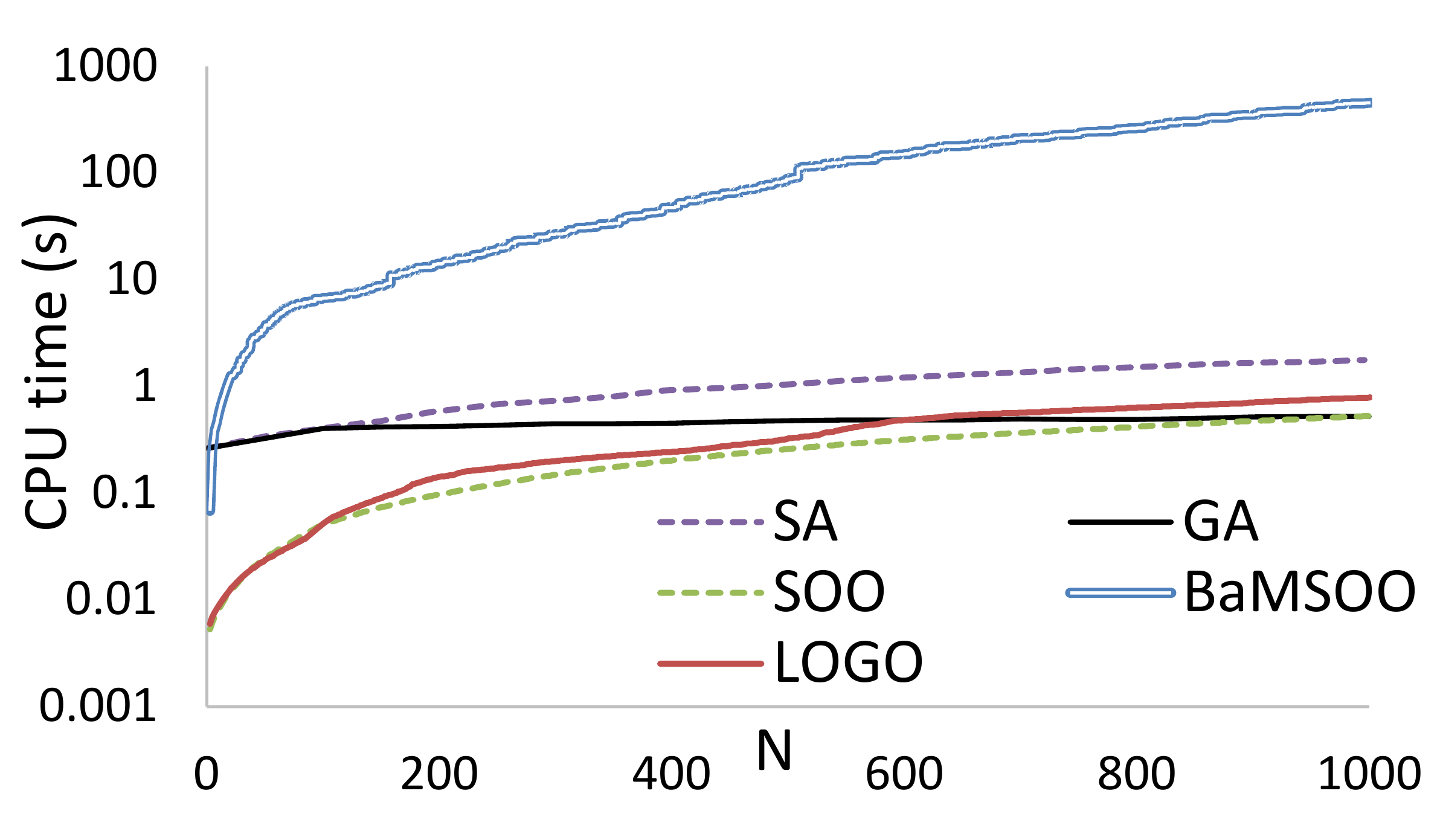}
\caption{Branin}
\end{subfigure}\hfill
\begin{subfigure}[b]{0.32\textwidth}
\includegraphics[width=\textwidth]{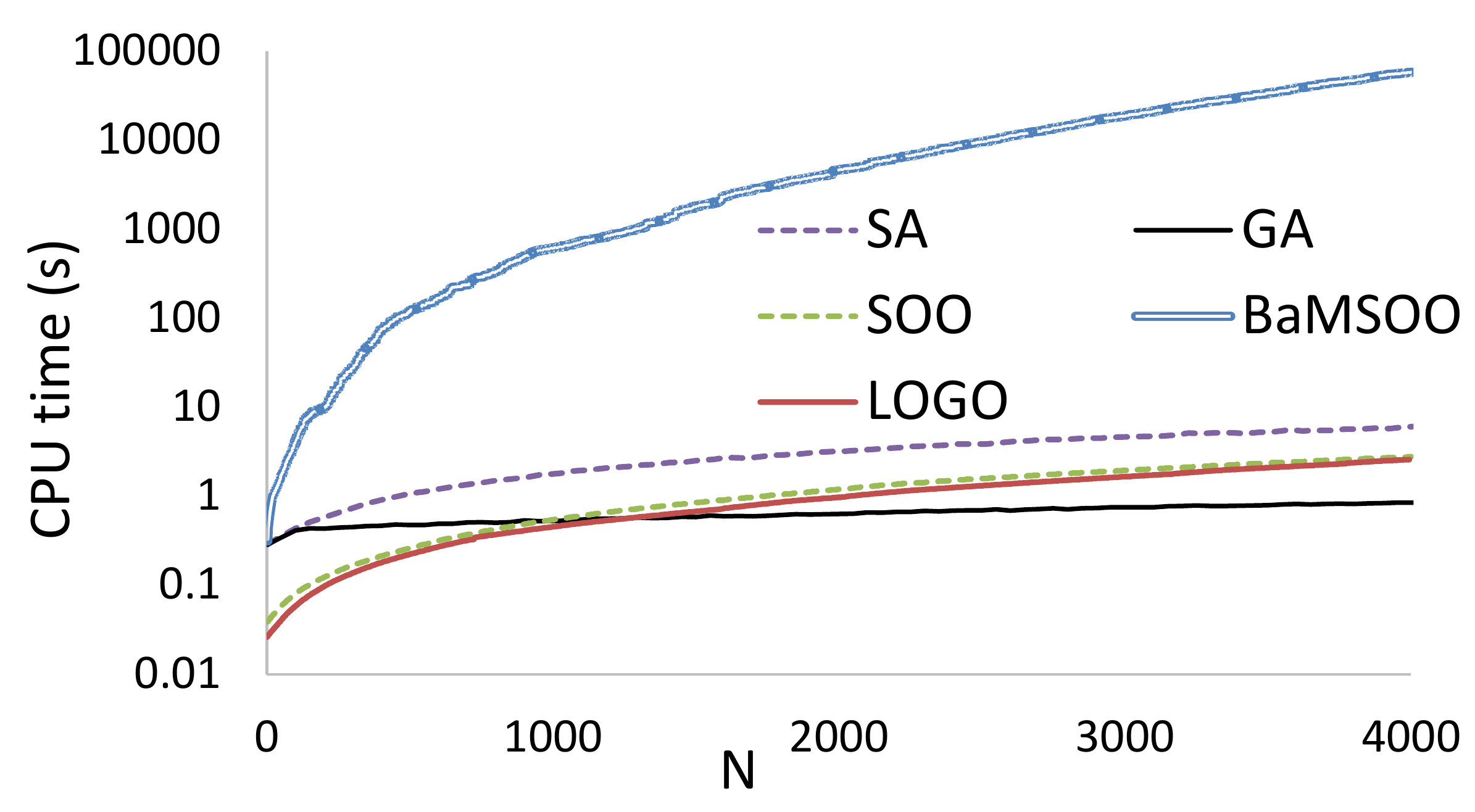}
\caption{Rosenbrock 2}
\end{subfigure}\hfill
\begin{subfigure}[b]{0.32\textwidth}
\includegraphics[width=\textwidth]{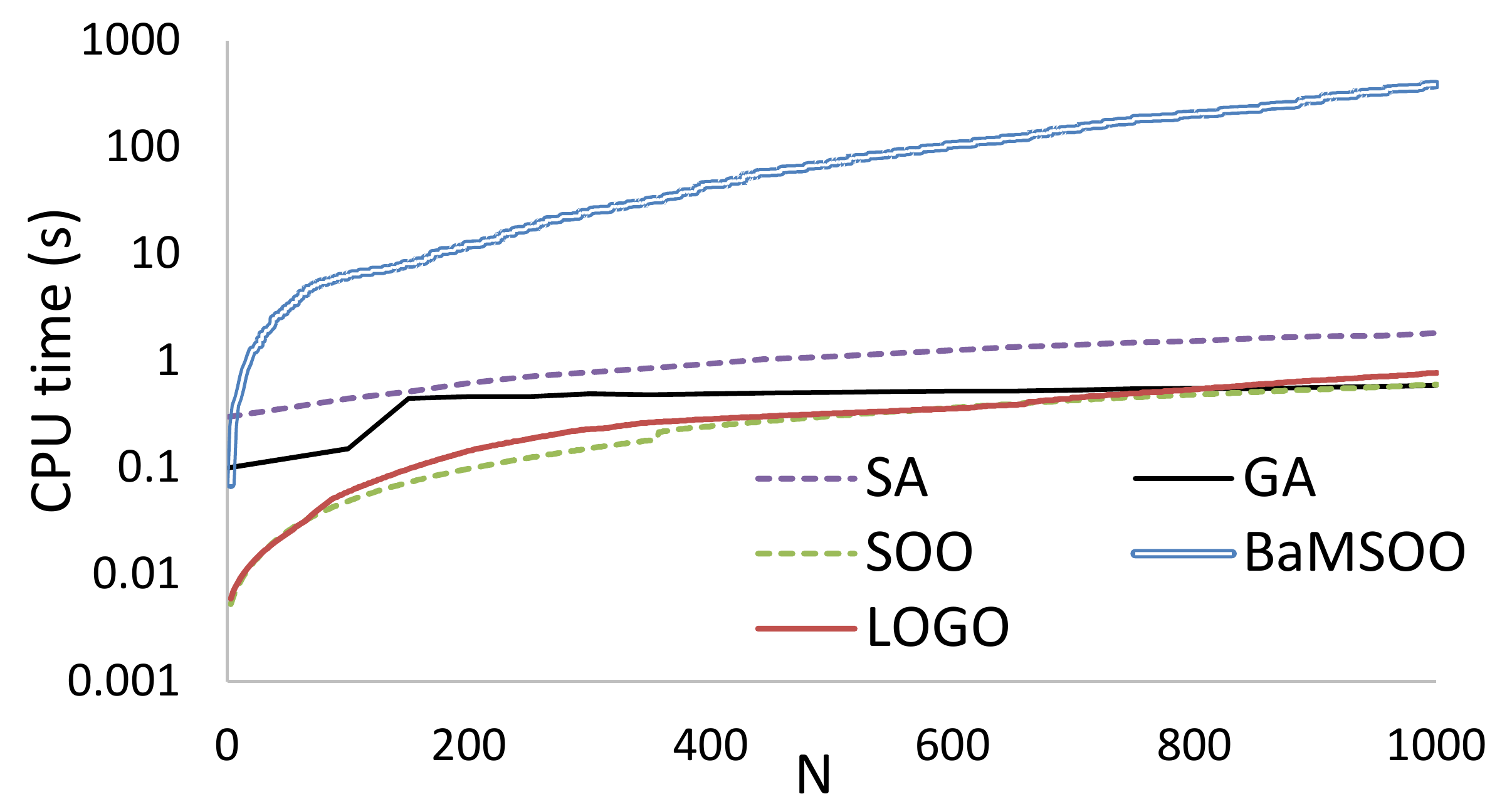}
\caption{Hartman 3}
\end{subfigure}
\vskip3mm
\begin{subfigure}[b]{0.32\textwidth}
\includegraphics[width=\textwidth]{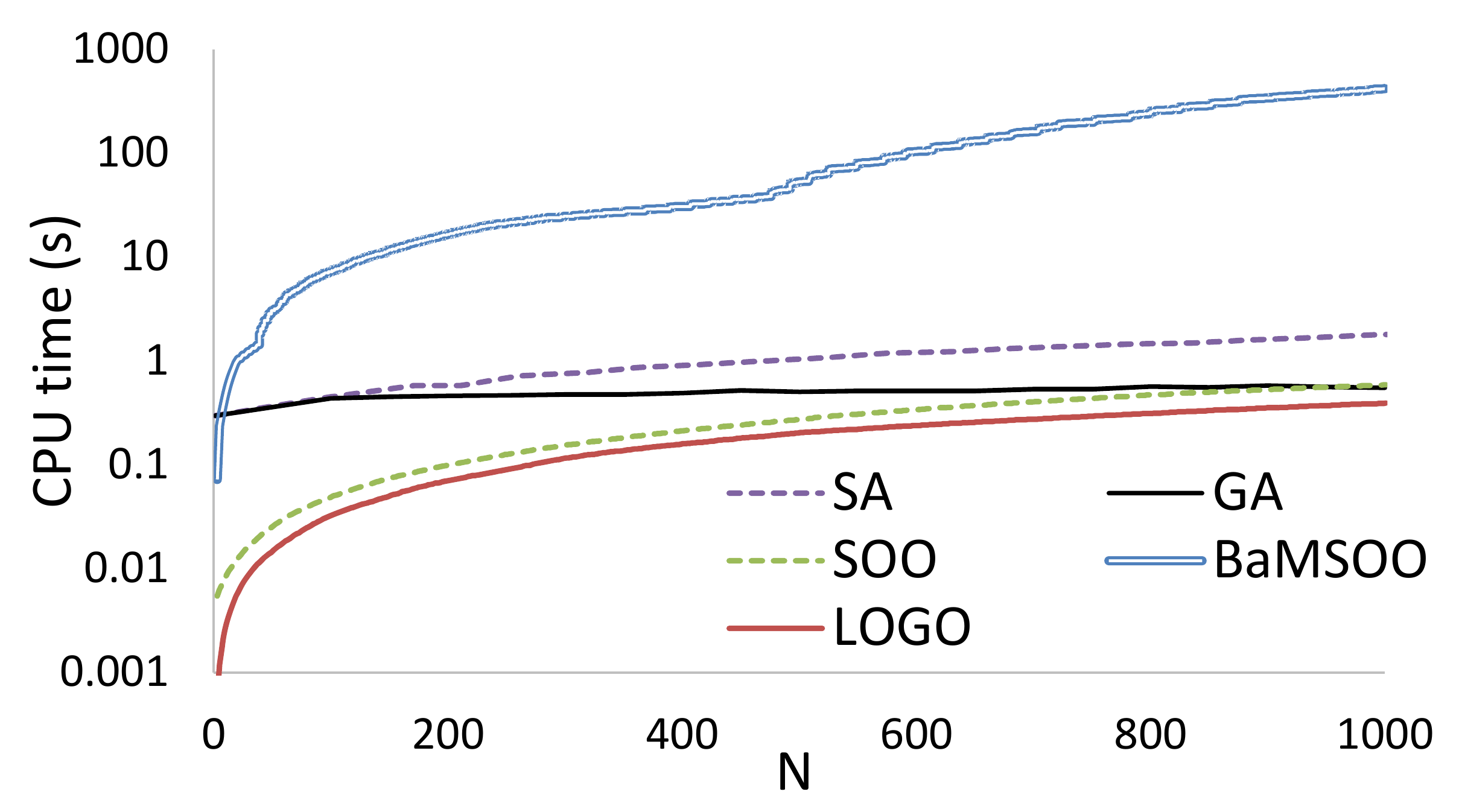}
\caption{Shekel 5}
\end{subfigure}\hfill
\begin{subfigure}[b]{0.32\textwidth}
\includegraphics[width=\textwidth]{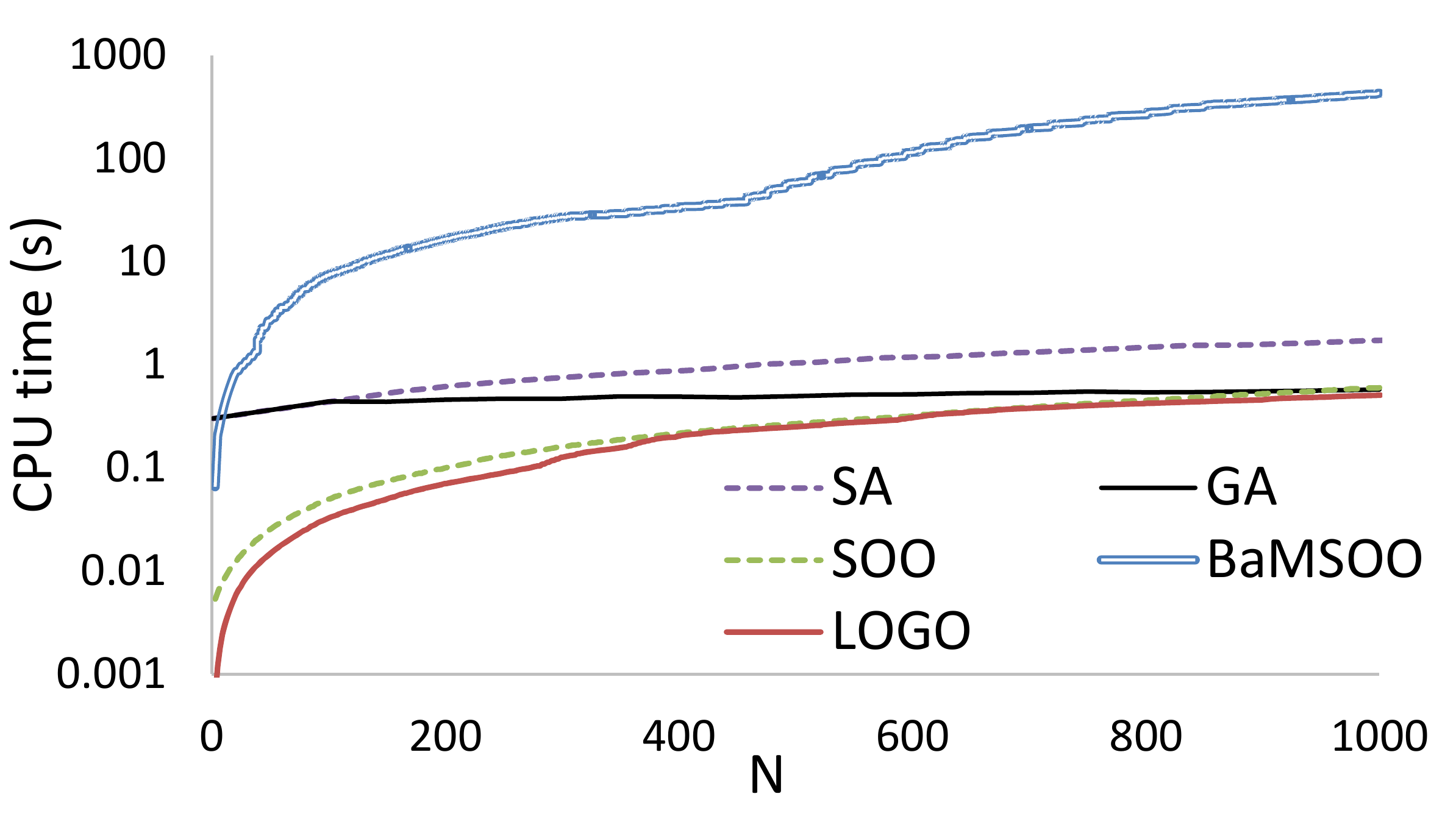}
\caption{Shekel 7}
\end{subfigure}\hfill
\begin{subfigure}[b]{0.32\textwidth}
\includegraphics[width=\textwidth]{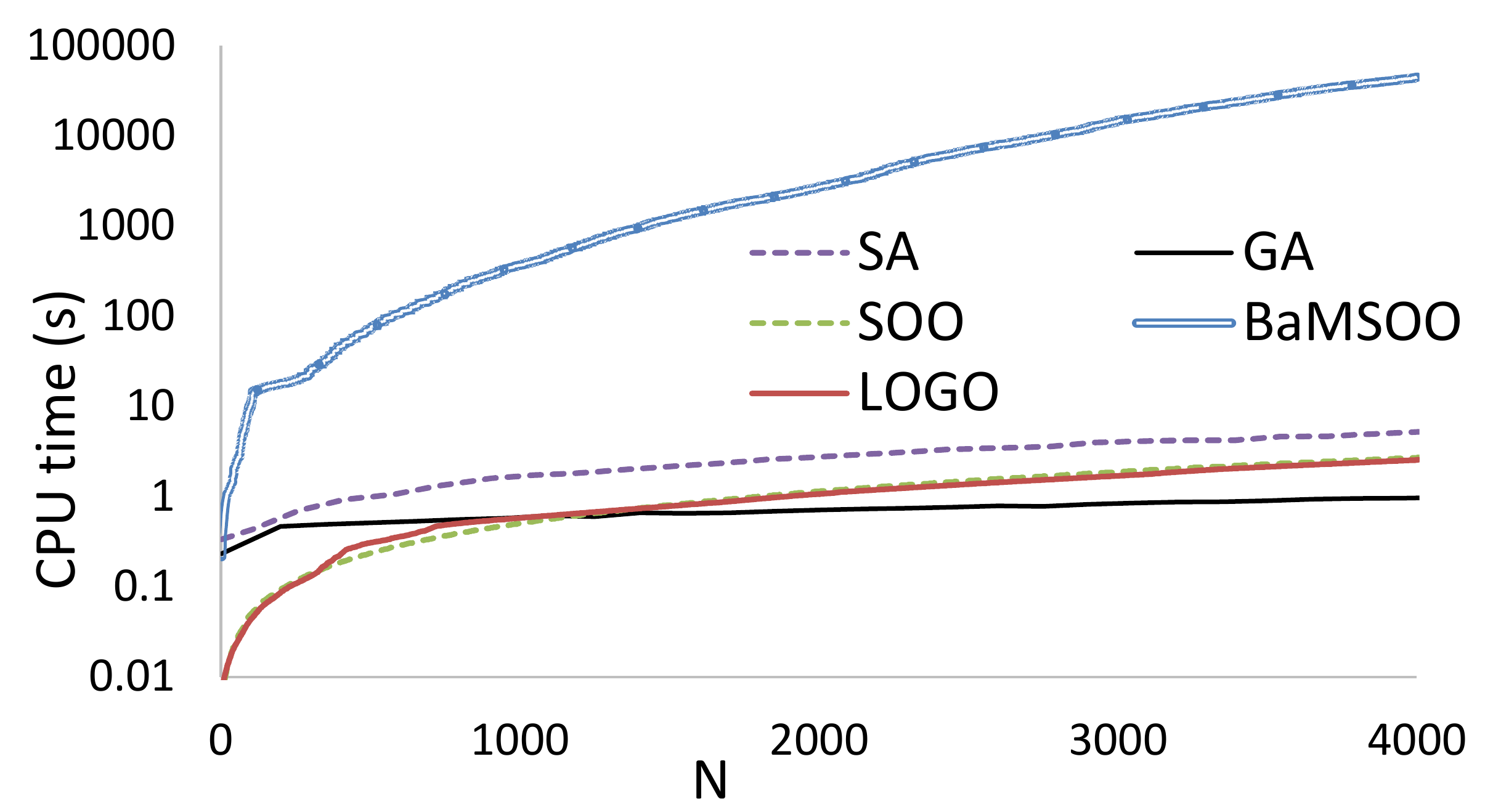}
\caption{Shekel 10}
\end{subfigure}
\vskip3mm
\begin{subfigure}[b]{0.32\textwidth}
\includegraphics[width=\textwidth]{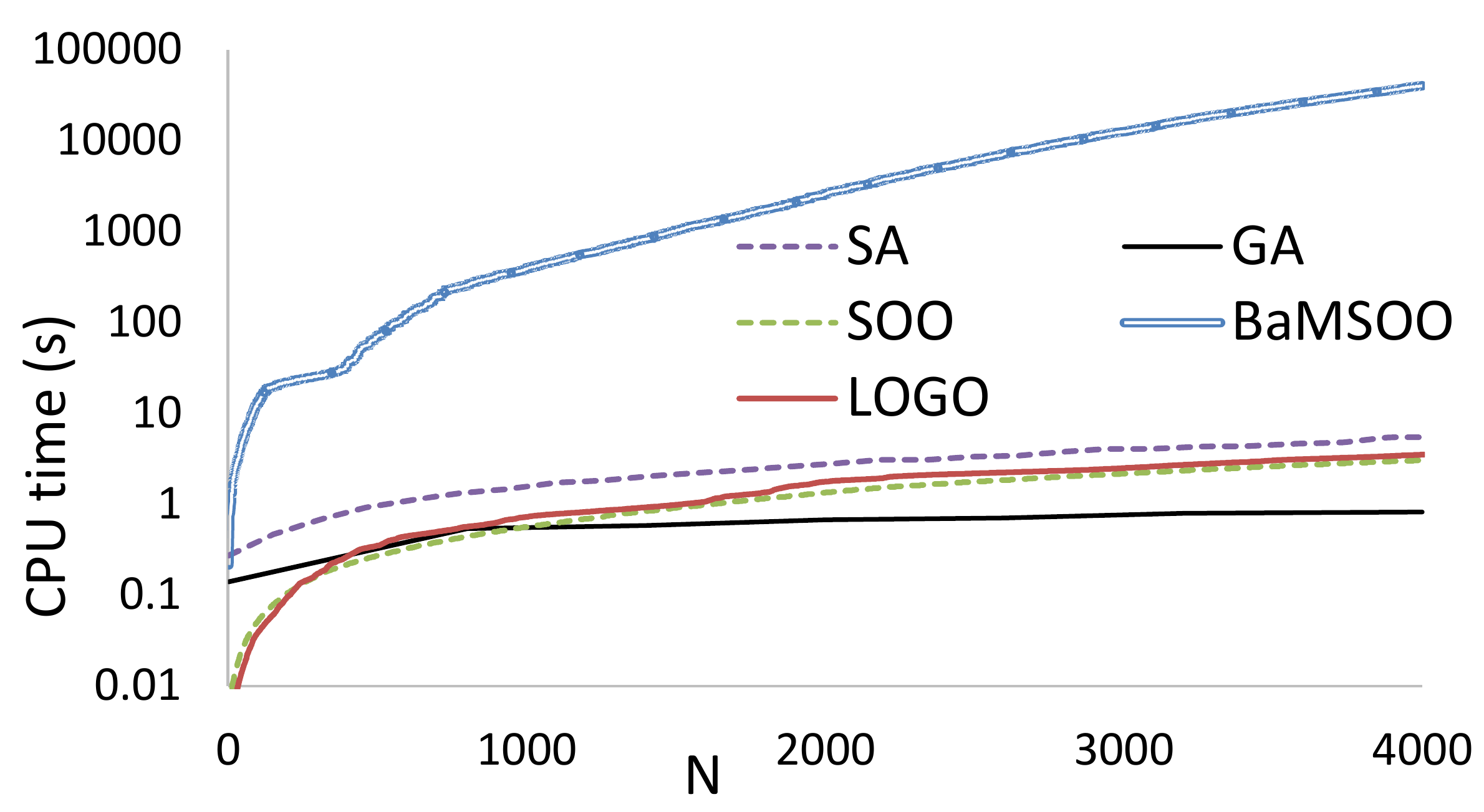}
\caption{Hartman 6}
\end{subfigure}\hfill
\begin{subfigure}[b]{0.32\textwidth}
\includegraphics[width=\textwidth]{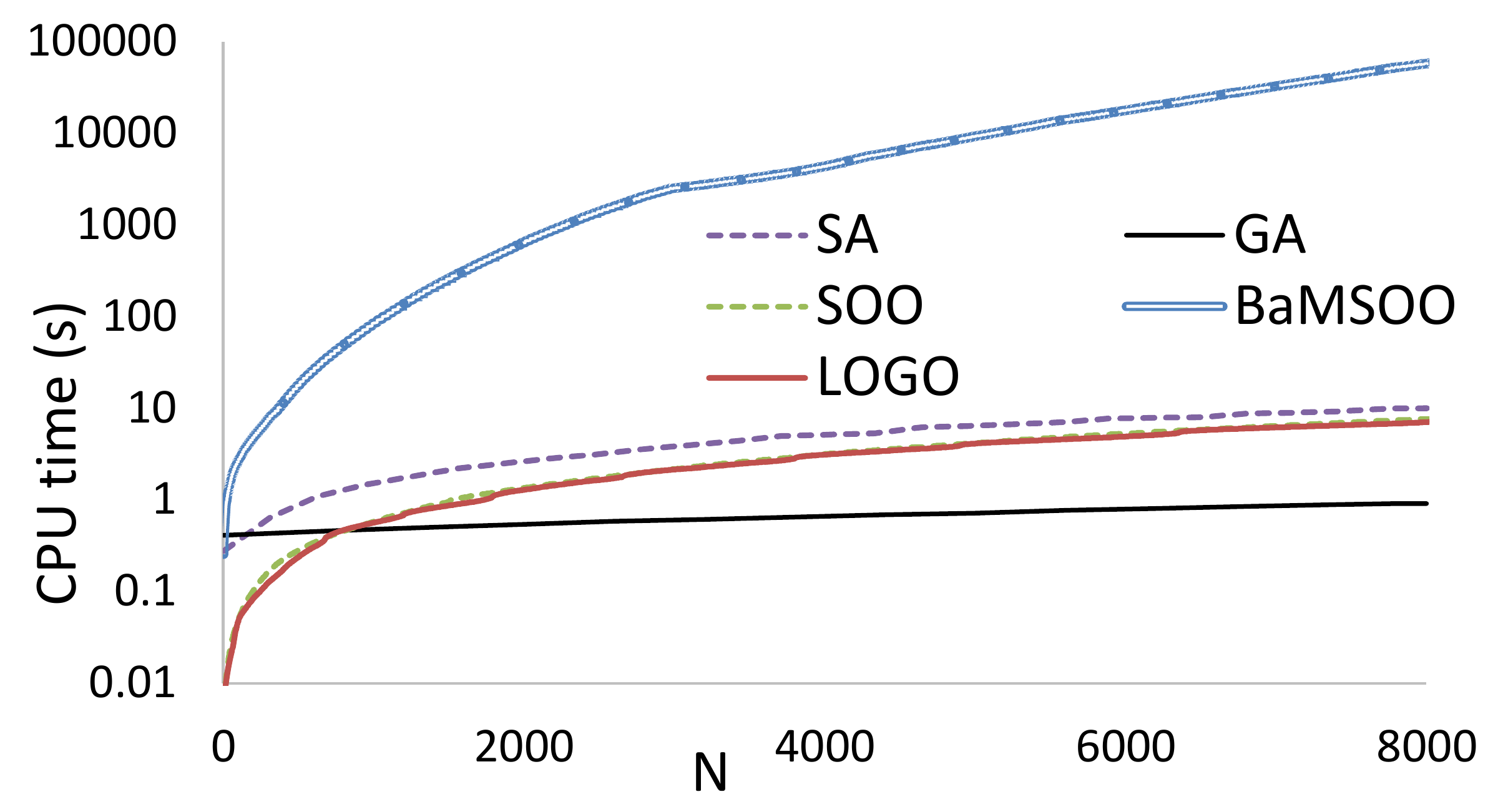}
\caption{Rosenbrock 10}
\end{subfigure}\hfill
\rule{0.32\textwidth}{0pt}
\caption{CPU time comparison: CPU time required to achieve the performance
indicated in Figure~\ref{fig5}}\label{fig6}
\end{figure}

As illustrated in Figure~\ref{fig5}, the LOGO algorithm generally
delivered improved performance compared to the other algorithms. A
particularly impressive result for the LOGO algorithm was its robustness
for the more challenging functions, Shekel~10 and Rosenbrok~10. The
function Shekel~$m$ has $m$ local optimizers and the slope of the surface
generally becomes larger as $m$ increases. Therefore, Shekel~10 and
Rosenbrok~10, which have $10$-dimensionality, are generally more difficult
functions when compared with the others in our experiment. Indeed,
only the LOGO algorithm achieved acceptable performance on these. From
Figure~\ref{fig6}, we can see that the LOGO algorithm and the SOO
algorithm were fast. The LOGO algorithm was often marginally slower
than the SOO algorithm owing to the additional computation required to
maintain the supersets. The reason why the BaMSOO algorithm required
a large computational cost at some horizontal axis points is that it
continued skipping to conduct the function evaluations (because the
evaluations were judged to be not beneficial based on GP). This is an
effective mechanism of BaMSOO to avoid wasteful function evaluations;
however, one must be careful to make sure that the function evaluations
are costly, relative to this mechanism.

In summary, compared to the BaMSOO algorithm, the LOGO algorithm was
faster and considerably simpler (in both implementation and parameter
selection) and had stronger theoretical bases while delivering superior
performance in the experiments. When compared with the SOO algorithm,
the LOGO algorithm decreased the theoretical convergence rate in the worst
case analysis, but exhibited significant improvements in the experiments.

\begin{figure}
\centering
\hair2pt
\begin{subfigure}[b]{0.49\textwidth}
\includegraphics[width=\textwidth]{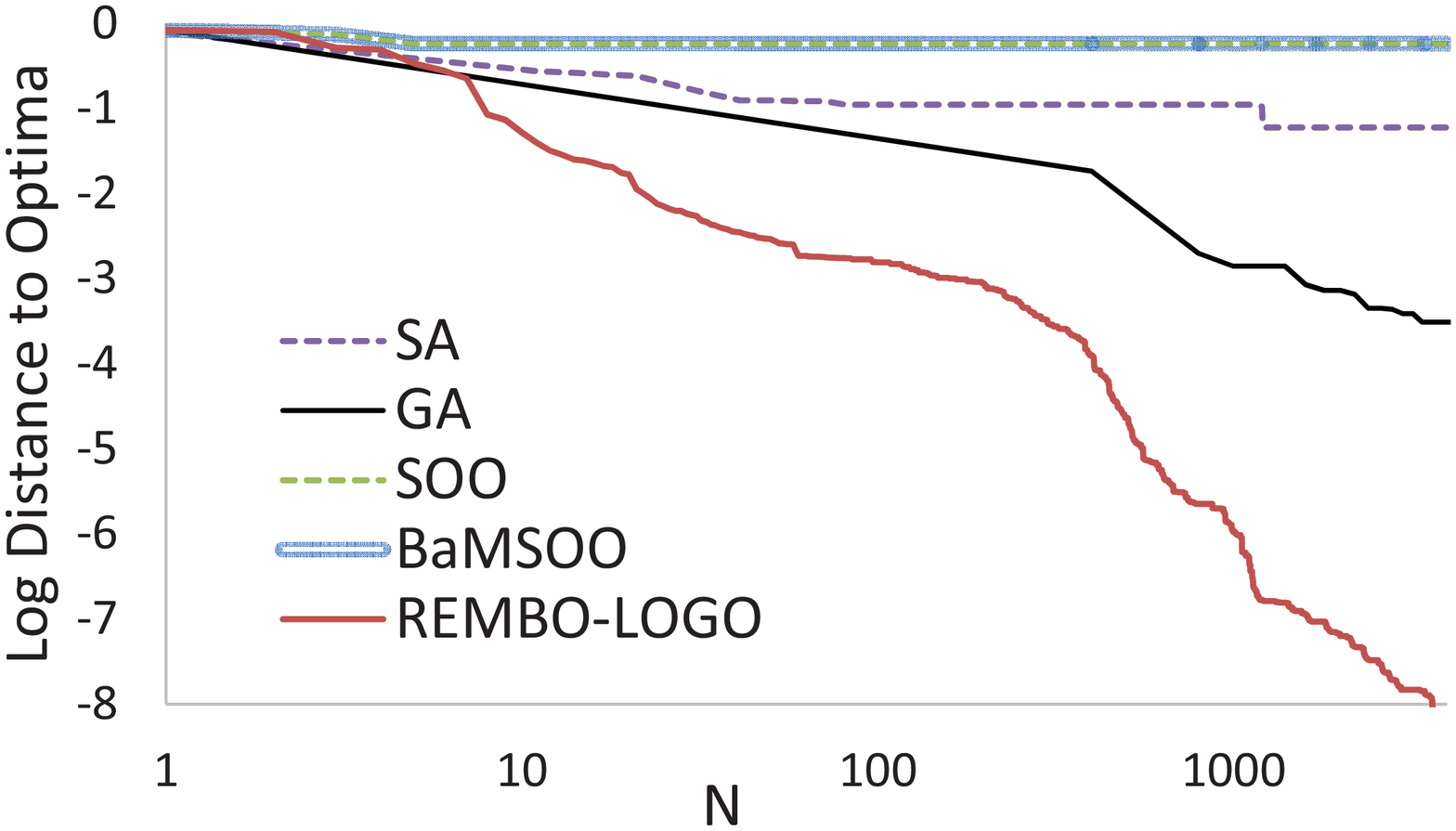}
\caption{ Scalability: a 1000-dimensional function} \label{fig7a}
\end{subfigure}
\labellist
\pinlabel \rotatebox{90}{\scriptsize Log Distance to Optimal \normalsize} [r] at 12 180
\pinlabel {$N$} [t] at 270 25
\endlabellist
\begin{subfigure}[b]{0.49\textwidth}
\includegraphics[width=\textwidth]{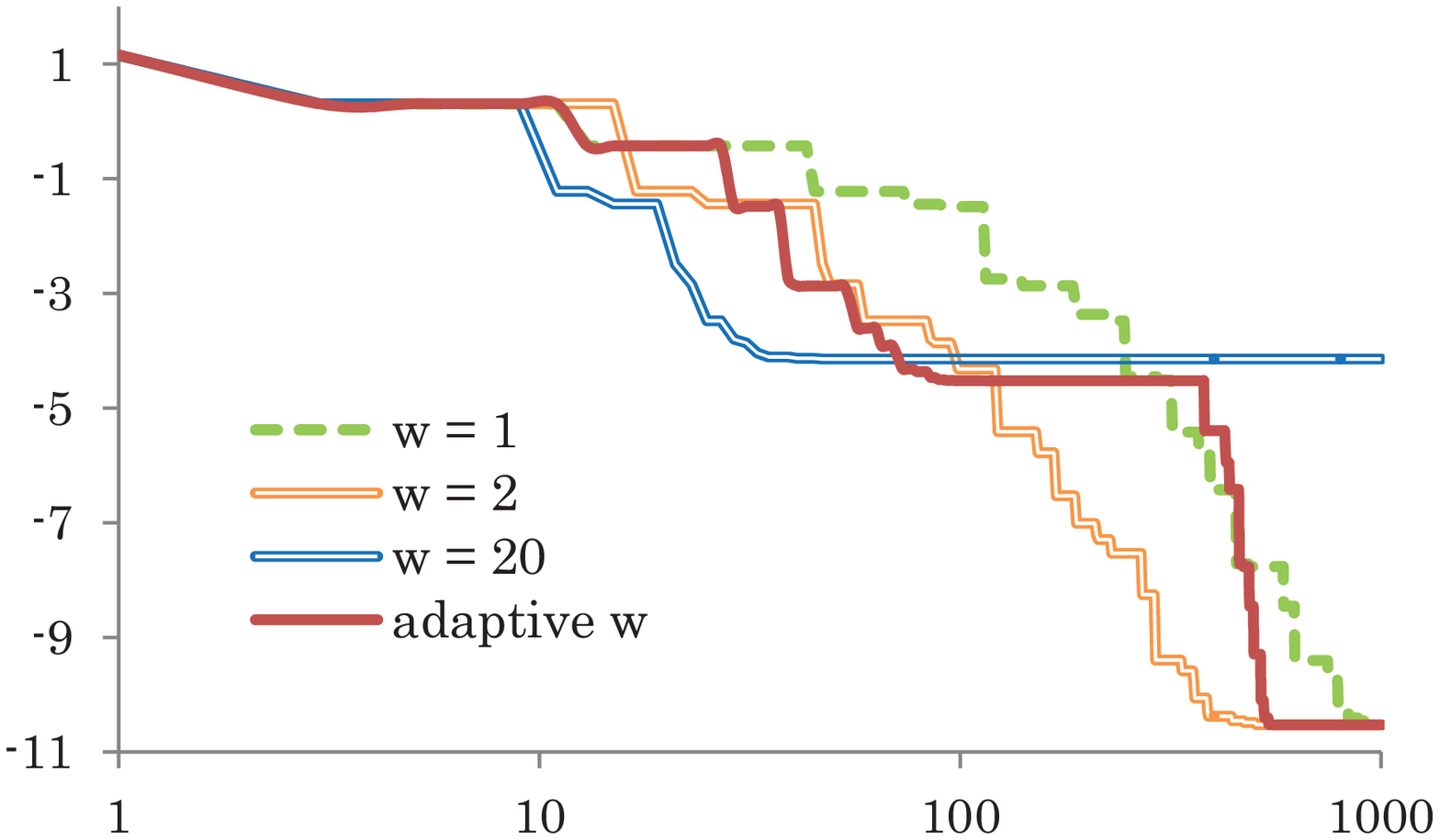}
\caption{Sensitivity to local bias parameter $w$ }\label{fig7b}
\end{subfigure}
\caption{On the current possible limitations of LOGO}\label{fig7}
\end{figure}

Now that we have confirmed the advantages of the LOGO algorithm, we discuss its possible limitations: scalability and parameter sensitivity. The scalability for high dimensions   is a challenge for non-convex optimization in general as the  search space grows exponentially in space. However, we may achieve the scalability by leveraging additional structures of the objective function that are present for some applications. For example, \citeauthor{kawaguchi2016deep} \citeyear{kawaguchi2016deep} showed an instance of deep learning models, in which the objective function has such an additional structure: the nonexistence of poor local minima. As an illustration, we combine LOGO with a random embedding method, REMBO \cite{wang2013bayesian}, to account for another structure: a low effective dimensionality. In Figure \ref{fig7} (a), we report the algorithms' performances for a 1000 dimensional  function: Sin 2  embedded in 1000 dimensions in the same manner described in Section 4.1 in the previous study \cite{wang2013bayesian}.        

Another possible limitation of LOGO is the sensitivity of its performance
to the free parameter $w$. Even though we provided theoretical analysis
and insight on the effect of the parameter value in the previous section,
it is yet unclear how to set $w$ in a principle manner. We illustrate
this current limitation in Figure~\ref{fig7} (b). The result labeled with ``adaptive w'' indicates the result  with the fixed adaptive mechanisms of \(w\) that we use in all the other experiments except ones in Figure 7 (b) and 8. In the illustration, we
use the Branin function because the experiment conducted with it clearly
illustrated the limitation. As can be seen in the figure, the performance
in the early stage is always improved as $w$ increases because the
algorithm finds a local optimum faster with higher $w$. However, if $w$
is too large, such as $w = 20$ in the figure, the algorithm gets stuck at
the local optimum for a long time. Thus, the best value (or sweet spot)
exists between too large and too small values of $w$. In the results of
this experiment, it can be seen that the choice of $w = 2$ is the best,
which finds the global optima with high precision within only $200$
function evaluations.

However, this limitation would not be a serious problem in practice for
the following four reasons. First, a similar limitation exists, to the
best of our knowledge, for any algorithms that are successfully used
in practice (e.g., simulated annealing, genetic algorithm, swarm-based
optimization, the DIRECT algorithm, and Bayesian optimization). Second,
unlike any other previous algorithm, the finite-time loss bound
always applies even for a bad choice of $w$. Third, we demonstrated in
the previous experiments that a very simple adaptive rule may suffice
to produce a good result. Also, future work may further mitigate this
limitation by developing different methods to adaptively determine
the value of $w$. Also, another possibility would be to conduct optimization over \(w\) with a cheaper surrogate model. 
Finally, the limitation
may not apply to some of the target objective functions at all.

For the fourth and final reason, recall that we speculated in the
algorithm's analysis that increasing $w$ would always have beneficial
effects in some problems, as illustrated in Figure~\ref{fig4}. Clearly,
any problems within the scope of local optimization fall into this
category. In Figure~\ref{fig8}, we show a rather unobvious instance
of such problems, and thus an example, for which the limitation of the
parameter sensitivity does not apply. As can be seen in the diagram on the
left in Figure~\ref{fig8}, this test function has many local optima, only
one of which is the global optimum. Nevertheless, as in the diagram on
the right, the performance of the LOGO algorithm improves as $w$ increases,
with no harmful effect.

\begin{figure}[ht]
\centering
\small\hair2pt
\labellist
\pinlabel $f$ [r] at -200 600
\pinlabel $x_1$ [tr] at 0 300
\pinlabel $x_2$ [t] at 400 330
\endlabellist
\includegraphics[width=0.47\textwidth]{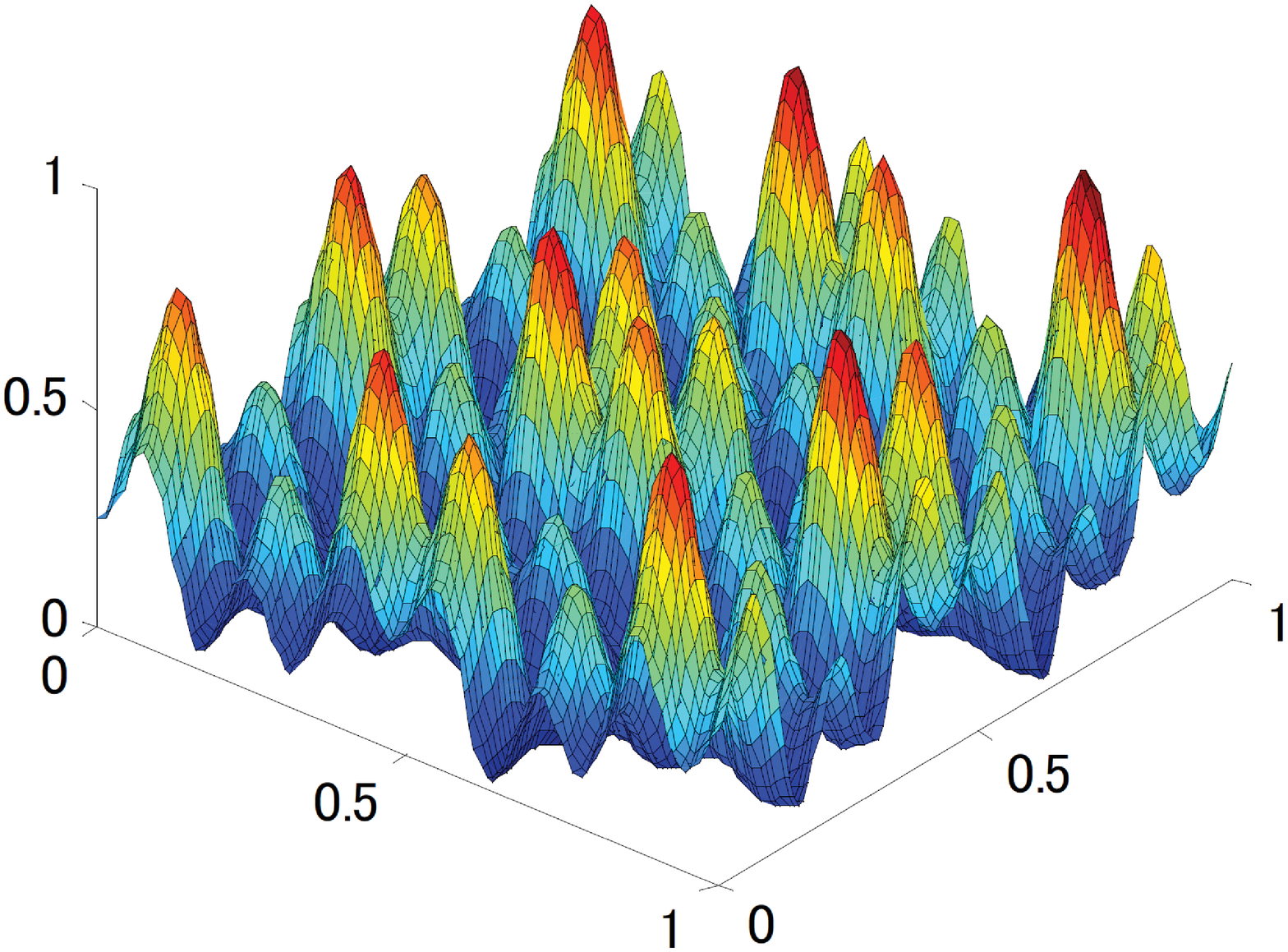}\hfill
\small\hair2pt
\labellist
\pinlabel \rotatebox{90}{Log Distance to Optimal} [r] at -160 420
\pinlabel {Local Bias Parameter: $w$} [t] at 230 142
\endlabellist
\includegraphics[width=0.44\textwidth]{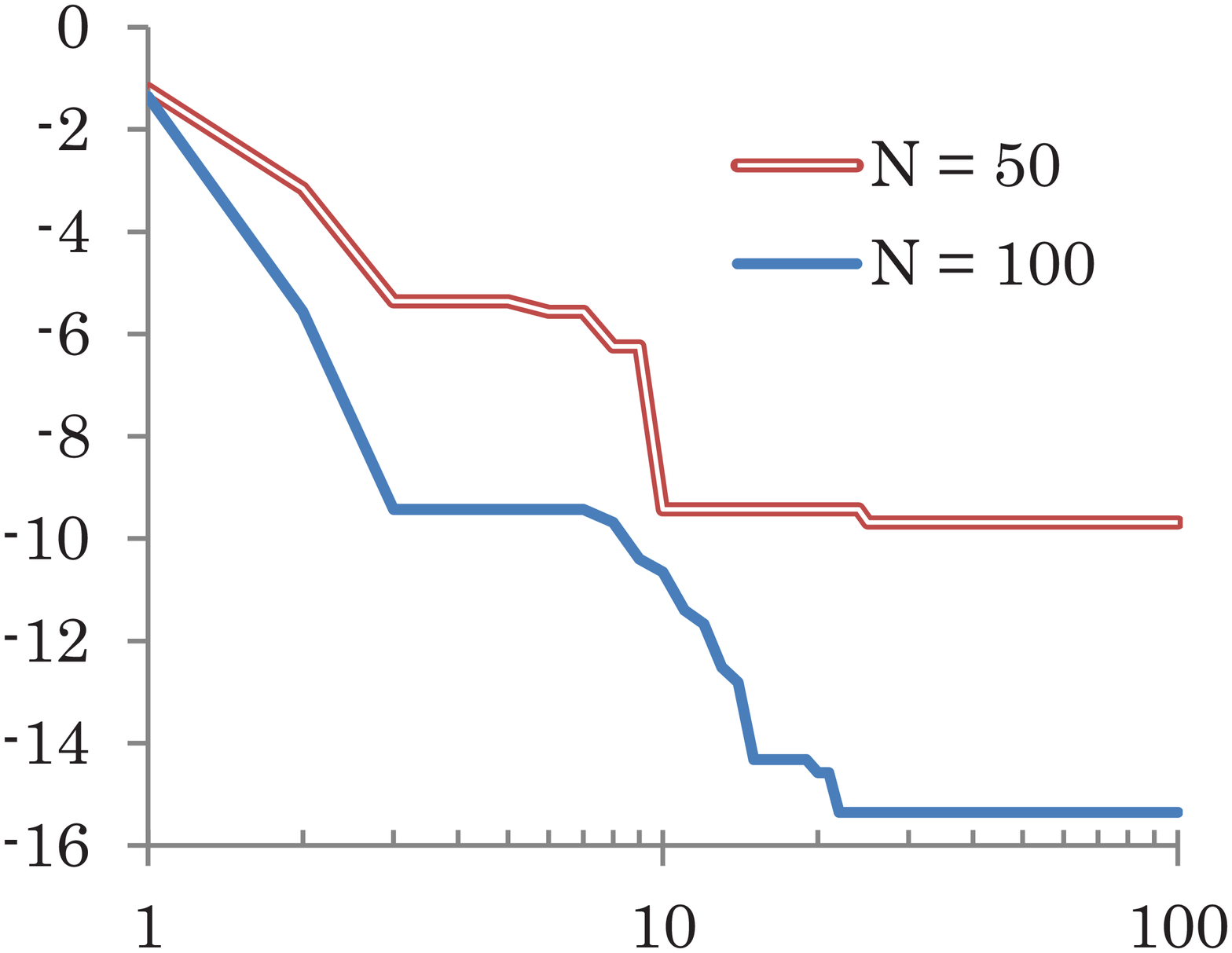}
\vskip3mm
\caption{An example of problems where increasing $w$ is always
better. The diagram on the left shows the objective function, and the
diagram on the right presents the performance at $N = 50$ and $100$
for each $w$.}\label{fig8}
\end{figure}

\section{Planning with LOGO via Policy Search}\label{sec3}

We now apply the LOGO algorithm to planning, which is an important area
in the field of AI. The goal of our planning problem is to find an action
sequence that maximizes the total return over the infinite discounted
horizon or finite horizon (unlike classical planning problem, we do not
consider constraints that specify the \emph{goal} state). In this paper,
we discuss the formulations for the case of the infinite discounted
horizon, but all the arguments are applicable to the case of a finite
horizon with straightforward modifications. We consider the case where
the state/action space is continuous, the planning horizon is long,
and the transition and reward functions are known and deterministic.

The planning problem can be formulated as
follows. Let $S\subset\R^{D_S}$ be a set of
states, $A\subset\R^{D_A}$ be a set of actions,
$T\colon\R^{D_S} \to \R^{D_S}$ be a transition function,
$R\colon\R^{D_S}\times\R^{D_A}\to\R$ be a return
or reward function, and $\gamma\le 1$ be a discount factor. A planner
considers to take an action $a \in A$ in a state $s \in S$, which triggers
a transition to another state based on the transition function $T$,
while receiving a return based on the reward function $R$. The discount
factor $\gamma$ discounts the future rewards to fulfill either or both
of the following two roles: accounting for the relative importance of
immediate rewards compared to future rewards, and obviating the need
to think ahead toward the infinite horizon. An action sequence can be
represented by a policy $\pi$ that maps the state space to the action
space: $\pi\colon\R^{D_S} \to \R^{D_A}$.

The value of an action sequence or a policy $\pi$, $V^\pi$, is the sum
of the rewards over the infinite discounted horizon, which is
\[
V^\pi(s_0)=\sum_{j=0}^\infty\gb^t R(s_j,\pi(s_j)).
\]
The value of a policy can be also written with a recursive form as
\begin{equation}\label{eq3}
V^\pi(s)=R(s,\pi(s)) + \gb V^\pi\bigl(T(s,\pi(s))\bigr).
\end{equation}
Here, we are interested in finding the optimal policy $\pi^*$. In
the dynamic programing approach, we can compute the optimal policy,
by solving the following Bellman's optimality equation:
\begin{equation}\label{eq4}
V^*(s)=\max_a R(s,a) + \gb V^*(T(s,a))
\end{equation}
where $V^*$ is the value of the optimal policy. In Equation~\eqref{eq4},
the optimal policy $\pi^*$ is the set of the actions defined by the
max. A major problem with this approach is that the efficiency of the
computation depends on the size of the state space. In a real-world
application, the state space is usually very large or continuous, which
often makes it impractical to solve Equation~\eqref{eq4}.

A successful approach to avoid the state size dependency is to focus
only on the state space that is reachable from the current state within
the planning time horizon. In this way, even with an infinitely large
state space, a planner only needs to consider a finitely sized subset
of the space. This approach is called \emph{local planning}. Unlike
local optimization vs.\ global optimization, the optimal solution of
local planning is indeed globally optimal, given the initial state. It
is called local planning because it does not cover all the states and
its solution changes for different initial states. Accordingly, as the
initial state changes, a planner may need to conduct re-planning.

A natural way to solve local planning is to use tree search
methods, which construct a tree rooted in an initial state toward the
future possible states in the depth of the planning horizon. This tree
search can be conducted using any traditional search method, including
both uninformed search (e.g., breadth-first and depth-first search) and
informed (heuristic) search (e.g., $\rm A^*$ search). Also, recent studies
have developed several tree-based algorithms that are specialized to
local planning. Among those, the SOO algorithm, the direct predecessor of
the LOGO algorithm, was applied to local planning with the tree search
approach \cite{bus13}. Most of the new algorithms, for example, HOLOP
\cite{bub10,wei12}, operate with stochastic transition functions.

However efficient these proposed algorithms are, the search space in the
tree search approach grows exponentially in the planning time horizon,
$H$. Therefore, local planning with the tree search approach would not work
well with a very long time horizon. In some applications, a small $H$
is justified, but in other applications, it is not. If an application
problem requires a long time tradeoff between immediate and future
rewards, then the tree search approach would be impractical. Here,
we are motivated to solve such a real-world application, and therefore
need another approach.

In this paper, we consider policy search \cite{dei13} as an effective
alternative to solve the planning problem with continuous state/action
space and with a long time horizon. Policy search is a form of local
planning. Thus, like the tree search approach,
it operates even with infinitely large or continuous state space. In
addition, unlike the tree search approach, policy search significantly
reduces the search space by naturally integrating the domain-specific
expert knowledge into the structure of the policy. More concretely,
the search space of policy search is a set of policies $\{\pi_x :
x \in \W\}$, which are parameterized by a vector $x$ in $\W\subset
\R^D$. Therefore, the search space is no longer dependent on the
planning time horizon $H$, the state space $S$, nor the action space
$A$, but only on the parameter space $\W$. Here, parameter space $\W$
can be determined by expert knowledge, which can significantly reduce
the search space.

We use the regret $r_m$ as the measure of the policy search algorithm's
performance:
\[
r_m = V^{\pi_x^*}(s_0) - V^{\pi_x^+(m)}(s_0)
\]
where $\pi_x^*$ is the optimal policy in the given set of policies
$\{\pi_x : x \in \W\}$, and $\pi_x^+(m)$ is the best policy found by an
algorithm after the $m$ steps of planning. An evaluation of each policy
takes $mH$ steps if we consider a fixed planning horizon $H$. Here,
$\pi_x^*$ may differ from the optimal policy $\pi^*$ when $\pi^*$ is
not covered in the set $\{\pi_x : x \in\W\}$.

The policy search approach is usually adopted with gradient methods
\cite{bax01,kob13,wei14}. While a gradient method is fast, it converges
to local optima \cite{sut99}. Further, it has been observed that it
may result in a mere random walk when large plateaus exist in the
surface of the policy space \cite{hei08}. Clearly, these problems
can be resolved using global optimization methods at the cost of scalability
\cite{bro09,aza14}.  Unlike previous policy search methods, our method guarantees finite-time regret bounds w.r.t.\ global optima in $\{\pi_x : x \in\W\}$ without strong additional assumption, and provides a practically useful convergence speed.

\subsection{LOGO-OP Algorithm: Leverage (Unknown) Smoothness in Both
Policy Space and Planning Horizon}\label{sec3.1}

In this section, we present a simple modification of the LOGO algorithm
to leverage not only the unknown smoothness in policy space but also the
known smoothness over the planning horizon. The former is accomplished
by the direct application of the LOGO algorithm to policy search, and
the latter is what the modification in this section aims to do without
losing the advantage of the original LOGO algorithm. We call the modified
version, Locally Oriented Global Optimization with Optimism in Planning
horizon (LOGO-OP).  As a result of this modification, we add a new free
parameter $L$.

The pseudocode for the LOGO-OP algorithm is provided in
Algorithm~\ref{alg2}. By comparing Algorithms~\ref{alg1} and \ref{alg2},
it can be seen that the LOGO-OP algorithm functions in the same manner
as the LOGO algorithm, except for line 15 (the function evaluation or,
equivalently, the policy evaluation in the policy search) and line
20. Notice that the LOGO algorithm is directly applicable to the policy
search by considering $V$ to be $f$ in Algorithm~\ref{alg1}. While
the LOGO algorithm does not assume the structure of the function $f$,
the LOGO-OP algorithm functions with and exploits the given structure
of the value function $V$ (i.e., MDP model). The algorithm functions
as follows. The policy evaluation is performed for each policy $\pi_x$
with a parameter $x$ specified by each of the two new hyperrectangles
(from line 15-1 to \hbox{15-11}). Given the initial condition $s_0 \in S$,
the transition function $T$, the reward function~$R$, a~discount factor
$\gamma \le 1$, and the policy $\pi_x$, the algorithm computes the value
of the policy as in Equation~\eqref{eq3} (from line 15-2 to line 15-10,
except line 15-6).

The main modification appears in line 15-6 where the algorithm leverages
the known smoothness over the planning horizon. Remember that \emph{the
unknown smoothness in policy space} (or input space $x$) is specified
as $f(x^*)-f(x)\le \sem(x,x^*)$ (from Assumption~\ref{ass1}) and thus
it infers the upper bound of the value of a policy \emph{that is not
yet evaluated but similar (close in policy space w.r.t.\ $\sem$) to
already evaluated polices}. Conversely, \emph{the known smoothness
over the planning horizon} renders the upper bound on the value of a
policy \emph{while the particular policy is being evaluated}. That is,
the known smoothness over the planning horizon can be written as
\[
\sum_{j=0}^\infty \gb^j R(s_j,\pi_x(s_j)) - \sum_{j=0}^t \gb^j
R(s_j,\pi_x(s_j))\le \frac{\gb^{t+1}}{1-\gb}\Rmax
\]
where $0\le t\le\infty$ is a arbitrary point in the planning horizon as
in line 15-3 and $\Rmax$ is the maximum reward. This known smoothness is
due to the definition of $\Rmax$ and the sum of a geometric series. In
the case of the finite horizon with $H$, we have the same formula with
$(\gb^t / (1-\gb))\Rmax$ being replaced by $(H-t)\Rmax$. In line 15-6,
unlike the original LOGO algorithm, the LOGO-OP algorithm terminates the
evaluation of a policy when the continuation of evaluating the policy
is judged to be a misuse of the computational resources based on the
known smoothness over the planning horizon. Concretely, it terminates
the evaluation of a policy when the upper bound of the value of the
policy becomes less than $(V^+ - L)$, where $V^+$ is the value of the
best policy found thus far and $L$ is the algorithm's parameter.

When the upper bound of the value of policy becomes less than $V^+$,
the planner can know that the policy is not the best policy. Thus,
it is tempting to simply terminate the policy evaluation with this
criterion. However, the essence of the LOGO algorithm is the utilization
of the unknown smoothness embedded in the surface of the value function
in the policy space. In other words, the algorithm makes use of the
result of each policy evaluation, whether the policy is the best one or
not. Any interruption of the policy evaluation changes the shape of the
surface of the value function, which interferes with the mechanism of
the LOGO algorithm. Nevertheless, the some degree of the interruption
is likely to be beneficial since our goal is to find the optimal policy
instead of surface analysis.

\begin{algorithm}[t]
\caption{LOGO-OP algorithm}\label{alg2}
\expandafter\patchcmd\csname\string\algorithmic\endcsname 
{\labelwidth 1.2em}{\labelwidth6mm}{}{}
\begin{algorithmic}[1]
\makeatletter
\addtocounter{ALG@line}{-1} 
\makeatother
\State {\bf Inputs (problem):} the initial condition $s_0 \in S$,  
the transition function $T$, 
the reward function $R$, 
a discount factor $\gamma \le 1$
with convergence criteria (or finite horizon $H$), 
the policy space $\pi_x: x \in \W\subset \R^D$.
\State {\bf Inputs (parameter):} the search depth function $\hmax\colon\Z^+\to[1,\infty)$, 
the local weight $w\in\Z^+$, stopping condition, the maximum reward $\Rmax$, a parameter $L$.
\let\algln\alglinenumber
\algrenewcommand{\alglinenumber}[1]{\footnotesize 2--5:}
\State lines 2--5 are exactly the same as lines 2--5 in Algorithm~\ref{alg1}
\makeatletter
\addtocounter{ALG@line}{3}
\makeatother
\algrenewcommand{\alglinenumber}[1]{\footnotesize #1:}
\State Adds the initial hyperrectangle $\W'$ to the set: $\psi_0\leftarrow\psi_0\cup\{\W'\}$ (i.e., $\omega_{0,0} = \W'$)
\State Evaluate the value function $V$ at the center point of $\W'$, $c_{0,0}$: $\val[\omega_{0,0}]\leftarrow V(c_{0,0})$, 
$V^+\leftarrow\val[\omega_{0,0}]$
\For{iteration ${}= 1, 2, 3,\dots$}
\State $\valmax\leftarrow -\infty$, $h_\mathit{plus}\leftarrow h_\mathit{upper}$
\For{$k=0,1,2,\dots,\max(\floor{\min(\hmax(n),h_\mathit{upper})/w},h_\mathit{plus})$}
\State \Select\ a hyperrectangle to be divided: $(h,i)\in\arg\max_{h,i}\val[\omega_{h,i}]$ for $h,i : \omega_{h,i}\in\Psi_k$
\If{$\val[\omega_{h,i}]>\valmax$}
\State $\valmax\leftarrow\val[\omega_{h,i}]$, $h_\mathit{plus}\leftarrow 0$, 
$h_\mathit{upper}\leftarrow\max(h_\mathit{upper},h+1)$, $n\leftarrow n+1$
\State \Divide\ this hyperrectangle $\omega_{h,i}$ along the longest coordinate direction
\Statex \qquad\qquad\quad\begin{tabular}[t]{l}
- three smaller hyperrectangles are created $\rightarrow$ $\omega_\mathit{left}$, $\omega_\mathit{center}$, $\omega_\mathit{right}$\\
- $\val[\omega_\mathit{center}]\leftarrow\val[\omega_{h,i}]$
\end{tabular}
\color{blue!30!black}
\State \Evaluate\ the value function $V$ at the center points of the
two new hyperrectangles:
\newcounter{foo}
\makeatletter
\setcounter{foo}{\value{ALG@line}}
\setcounter{ALG@line}{0}
\algrenewcommand{\alglinenumber}[1]{\footnotesize\hskip1.3cm \llap{\arabic{foo}--\arabic{ALG@line}}:}
\makeatother
\For{each policy $\pi_x$ corresponding $c_{\omega_\mathit{left}}$ and $c_{\omega_\mathit{right}}$}
\State $z_1\leftarrow0$, $z_2\leftarrow1$, $s\leftarrow s_0$
\For{$t=0,1,2,\dots$,}
\State $z_1\leftarrow z_1+z_2R(s,\pi_x(s))$
\State $z_2\leftarrow \gb z_2$, $s\leftarrow T(s,\pi_x(s))$           ,
\State {\bf if} $z_1 + (\gb^{t+1} / (1-\gb))\Rmax < (V^+ - L)$ {\bf then Exit loop}
\State {\bf if}  convergence criteria is met {\bf then Exit loop}
\EndFor
\State save $z_1$ as the value of the corresponding rectangle
\State $\val[\omega_\mathit{left}]\leftarrow z_1$ or $\val[\omega_\mathit{right}]\leftarrow z_1$
\EndFor
\State $V^+\leftarrow\max(V^+,\val[\omega_\mathit{left}],\val[\omega_\mathit{center}],\val[\omega_\mathit{right}])$
\makeatletter
\setcounter{ALG@line}{\value{foo}}
\makeatother
\algrenewcommand{\alglinenumber}[1]{\footnotesize #1:}
\color{black}
\State \Group\ the new hyperrectangles into the set $h + 1$ and
remove the original rectangle:
\begin{center}
$
\psi_{h+1}\leftarrow\psi_{h+1}
\cup\{\omega_\mathit{center},\omega_\mathit{left},\omega_\mathit{right}\},\quad
\psi_h\leftarrow\psi_h\setminus \omega_{h,i}
$
\end{center}
\EndIf
\State {\bf if} stopping condition is met {\bf then Return} $(h,i)=\arg\max_{h,i}\val[\omega_{h,i}]$
\EndFor
\color{blue!30!black}
\State {\bf for} all intervals $\omega$ with $\val[\omega]<(V^+-L)$ {\bf do} $\val[\omega]\leftarrow (V^+ - L)$
\color{black}
\EndFor
\end{algorithmic}
\end{algorithm}

The LOGO-OP algorithm uses $L$ to determine the degree of the
interruption. Because $V^+$ is monotonically increasing along the
execution, the value of a policy that is not fully evaluated owing to
line 15-6 in early iterations tends to be greater than the value of a
policy that is not fully evaluated in the later iterations. The algorithm
resolves this problem in line 20 such that it is not biased to divide
the interval evaluated in an early iteration.

With smaller $L$, the LOGO-OP algorithm can stop the evaluation of a
non-optimal policy earlier, at the cost of accuracy in the evaluation
of the value function's surface. With larger $L$, the algorithm needs to
spend more time on the evaluation of a non-optimal policy, but can obtain
a more accurate estimate of the value function's surface. In the regret
analysis, we show that a certain choice of $L$ ensures a tighter regret
bound when compared to the direct application of the LOGO algorithm.

\subsection{A Parallel Version of the LOGO-OP Algorithm}\label{sec3.2}

The LOGO-OP algorithm presented in Algorithm~\ref{alg2} has four main
procedures: \Select\ (line 11), \Divide\ (line 14), \Evaluate\ (line 15),
and \Group\ (line 16). A natural way to parallelize the algorithm
is to decouple \Select\ from the other three procedures. That is,
let the algorithm first \Select\ $z$ hyperrectangles to be divided,
and then allocate the $z$ number of \Divide, \Evaluate, and \Group\ to
$z$ parallel workers. However, this natural parallelization has data
dependency from one \Select\ to another \Select. In other words, the
procedure of the next \Select\ cannot start before \Divide, \Evaluate, and
\Group\ for the previous \Select\ are finalized. As a result, the parallel
overhead tends to be non-negligible. In addition, if \Select\ chooses less
hyperrectangles than parallel workers, then the available resources of
the parallel workers are wasted. Indeed, the latter problem was tackled
by creating multiple initial rectangles in a recent parallelization study
of the DIRECT algorithm \cite{he09}. While the use of multiple initial
rectangles can certainly mitigate the problem, it still allows the
occasional occurrence of the resource wastage, in addition to requiring
the user to specify the arrangement of the initial rectangles.

To solve these problems, we instead decouple the \Evaluate\ procedure
from the other three procedures and allocate only the \Evaluate\ task
to each parallel worker. We call the parallel version, the pLOGO-OP
algorithm. The algorithm uses one master process to conduct \Select,
\Divide, and \Group\ operations and an arbitrary number of parallel
workers to execute \Evaluate. The main idea is to temporarily use the
artificial value assignment to the center point of a hyperrectangle
in the master process, which is overwritten by the true value when the
parallel worker finishes evaluating the center point. With this strategy,
there is no data dependency and all the parallel workers are occupied with
tasks almost all the time. In this paper, we use the center value of the
original hyperrectangle before division as the temporary artificial value,
but the artificial value may be computed using a more advanced method
(e.g., methods in surface analysis) in future work. For the center point
of the initial hyperrectangle, we simply assign the worst possible value
(if we have no knowledge regarding the worst value, we can use $-\infty$).

The master process keeps selecting new hyperrectangles unless all the
parallel workers are occupied with tasks. This logic ensures that all
the parallel workers always have tasks assigned by the master process,
but the master process does not select too many hyperrectangles based on
the artificial information. Note that this parallelization makes sense
only when \Evaluate\ is the most time consuming procedure, and it is very
likely true for policy evaluation.

\subsection{Regret Analysis}\label{sec3.3}

Under a certain condition, all the finite-loss bounds of the LOGO
algorithm are directly translated to the regret bound of the LOGO-OP
algorithm. The condition that must be met is that $(V^+ - L)$ is
less than the center value of the optimal hyperinterval during the
algorithm's execution. We state the regret bound more concretely
below. For simplicity, we use the notion of a planning horizon $H$,
which is the effective (non-negligible) planning horizon for LOGO in
accordance with the discount factor, $\gamma$. Let $H'$ be the effective
planning horizon of the LOGO-OP algorithm. Then, the planning horizon for
LOGO-OP, $H'$, becomes smaller than that for LOGO, $H$, as the algorithm
finds improved function values. This is because the LOGO-OP algorithm
terminates each policy evaluation at line 15-6 when the upper bound on
the policy value is determined to be lower than $(V^+ - L)$.

\begin{corollary}\label{cor3}
Let $H'\le H$ be the planning horizon used by the LOGO-OP algorithm at
each policy evaluation. Let $V^+$ be the value of the best policy found
by the algorithm at any iteration. Assume that the value function of
the policy satisfies Assumptions~\ref{ass1} and \ref{assB1}. If $(V^+
- L)$ is maintained to be less than the center value of the optimal
hyperinterval, then the algorithm holds the finite-time loss bound of
Theorem~\ref{the2} with
\[
n\ge\bbfloor{\frac{m}{2H'}}.
\]
\end{corollary}

\begin{proof}
As the policy search is just a special case of the optimization problem,
it is trivial that the loss bound of Theorem~\ref{the2} holds for
the LOGO algorithm when it is applied to policy search. Because every
function evaluation takes $H$ steps in the planning horizon, we have
$n\ge\floor{m/2H}$ in this case. For the LOGO-OP algorithm, only the
effect that new parameter $L$ has in the loss analysis takes place in
the proof of Lemma~\ref{lem2}. If $(V^+ - L)$ is maintained to be less
than the center value of the optimal interval, then all the statements
in the proof hold true for the LOGO-OP algorithm as well. Here, due the
effect of $L$, function evaluation may take less than $H$ steps in the
planning horizon. Therefore, we have the statement of this corollary.
\end{proof}

We can tighten the regret bound of the LOGO-OP algorithm by decreasing
$L$, since the algorithm can then terminate evaluations of unpromising
policies earlier, which means that the value of $H'$ in the bound is
reduced. However, using a too small value of $L$ that violates the
condition in Corollary~\ref{cor3} leads us to discard the theoretical
guarantee. Even in that case, because the too small value of $L$
only results in a more global search, the consistency property,
$\lim_{n\to\infty}r_n=0$, is still trivially maintained. On the other
hand, if we set $L =\infty$, the LOGO-OP algorithm becomes equivalent to
the direct application of the LOGO algorithm to policy search, and thus,
we have the regret bound of Corollary~\ref{cor3} with $H' = H$.

The pLOGO-OP algorithm also maintains the same regret bound with $n =
n_p$ where $n_p$ counts the number of the total divisions that are devoted
to the set of $\delta$-optimal hyperinterval $\psi_{kw+l}(l+1)^*$, where
$(w - 1)\ge l\ge 0$. While non-parallel versions ensure the devotion to
$\psi_{kw+l}(l+1)^*$, the parallelization makes it possible to conduct
division on other hyperintervals. Thus, considering the worst case,
the pLOGO-OP may not improve the bound in our proof procedure, although
the parallelization is likely beneficial in practice.

\subsection{Application Study on Nuclear Accident
Management}\label{sec3.4}

The management of the risk of potentially hazardous complex systems,
such as nuclear power plants, is a major challenge in modern society. In
this section, we apply the proposed method to accident management of
nuclear power plants and demonstrate the potential utility and usage of
our method in a real-world application. Our focus is on assessing the
efficiency of containment venting as an accident management measure and
on obtaining knowledge about its effective operational procedure (i.e.,
policy $\pi$). This problem requires planning with continuous state space
and with a very long planning horizon ($H\ge 86400$), for which dynamic
programming (e.g., value iteration), tree-based planning (e.g., $\rm A^*$
search and its variants) would not work well (dynamic programming suffers
from the curse of dimensionality for the state space, and the search
space of tree-based methods grows exponentially in the planning horizon).

Containment venting is an operation that is used to maintain the integrity
of the containment vessel and to mitigate accident consequences by
releasing gases from the containment vessel to the atmosphere.  In the
accident at the Fukushima Daiichi nuclear power plant in 2011, the
containment venting was activated as an essential accident management
measure. As a result, in 2012, the United States Nuclear Regulatory
Commission (USNRC) issued an order for $31$ nuclear power plants to
install the containment vent system {\DeclareRobustCommand{\us}[2]{#2}\cite{usn13}}. Currently, many
countries are considering the improvement of the containment venting
system and its operational procedures \cite{oec14}. The difficulty of
determining its actual benefit and effective operation comes from the fact
that the containment venting also releases fission products (radioactive
materials) into the atmosphere. In other words, the effective containment
venting must trade off the future risk of containment failure against
the immediate release of fission products (radioactive materials). In
our experiments, we use the release amount of the main fission product
compound, cesium iodide (CsI), as a measure of the effectiveness of the
containment venting.

In the nuclear accident management literature, an integrated physics
simulator is used as the model of world dynamics or the transition
function $T$ and the state space $S$. The simulator that we adopt in
this paper is THALES2 (\underline{T}hermal \underline{H}ydraulics and
radionuclide behavior \underline{A}nalysis of \underline{L}ight water
reactor to \underline{E}stimate \underline{S}ource terms under severe
accident conditions) \cite{ish02}. Thus, the transition function $T$
and the state space $S$ are fully specified by THALES2. The initial
condition $s_0 \in S$ is designed to approximately simulate the accident
at the Fukushima Daiichi nuclear power plant. In this experiment, we
focus on a single initial condition with the deterministic simulator,
the relaxation of which is discussed in the next section. The reward
function $R$ is the negative of the amount of CsI being released in the
atmosphere as a result of a state-action pair. We use the finite-time
horizon $H = 86400$ seconds ($24$ hours), which is a traditional first
phase time-window considered in risk analysis with nuclear power plant
simulations (owing to the assumption that after $24$ hours, many highly
uncertain human operations are expected). We use the following policy
structure based on our engineering judgment.
\[
\pi_x =
\begin{cases}
1 & \text{if }((\mathit{FP}\le x_1) \cap (\mathit{Press}\ge x_2)) \cup
(\mathit{Press} > 100490),\\
0 & \text{otherwise},
\end{cases}
\]
where $\pi_x = 1$ indicates the implementation of the containment venting,
$\mathit{FP}$ (g) represents the amount of CsI in the gas phase of the suppression
chamber, and $\mathit{Press}$ (kgf/m$^2$) is the pressure of the
suppression chamber. Here, the suppression chamber is the volume in the
containment vessel that is connected to the atmosphere via the containment
venting system. This policy structure reflects our engineering knowledge
that the venting should be done while the fission products exist under
a certain amount in the suppression chamber, but should not be operated
before the pressure gets larger than a specific value. We consider $x_1 =
[0, 3000]$ and $x_2 = [10810, 100490]$. We let $\pi_x = 1$ whenever the
pressure exceeds $100490$\,kgf/m$^2$, since the containment failure is
considered to probably occur after the pressure exceeds this point. The
detail of the experimental setting is outlined in Appendix~A.

We first compare the performance of various algorithms in this
problem. For all the algorithms, we used the same parameter settings as in
the benchmark tests in Section~\ref{sec2.4}. That is, we used $\hmax(n)
= w\sqrt{n} -w$ and a simple adaptive procedure for the parameter $w$
with $W = \{3, 4, 5, 6, 8, 30\}$. For the LOGO-OP algorithm and the
pLOGO-OP algorithm, we blindly set $L = 1000$ (i.e., there is likely a
better parameter setting for $L$). We used only eight parallel workers
for the pLOGO-OP algorithm.

\begin{figure}
\centering
\hair2pt
\labellist
\pinlabel \rotatebox{90}{CsI release by the Computed Policy (g)} [r] at 0 100
\pinlabel {Wall time (s)} [t] at 160 0
\endlabellist
\includegraphics[width=0.65\textwidth]{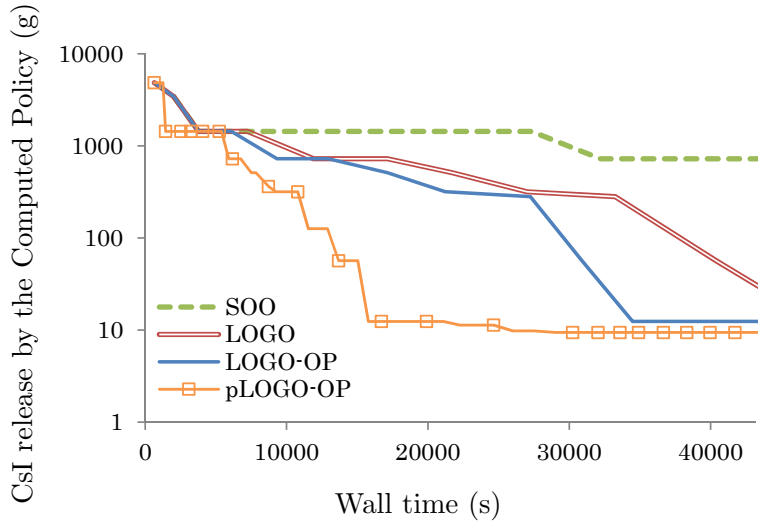}
\vskip5mm
\caption{Performance of the computed policy (CsI release) vs.\ Wall
Time.}\label{fig9}
\end{figure}

Figure~\ref{fig9} shows the result of the comparison with wall
time $\le 12$ hours. The vertical axis is the total amount of CsI
released into the atmosphere (g), which we want to minimize. Since
we conducted containment venting whenever the pressure exceeded
$100490$\,kgf/m$^2$, containment failure was prevented in all the
simulation experiments. Thus, the lower the value along the vertical
axis gets, the better the algorithm's performances is. As can be seen,
the new algorithms performed well compared to the SOO algorithm. It
is also clear that the two modified versions of the LOGO algorithm
improved the performance of the original. For the LOGO-OP algorithm,
the effect of $L$ on the computational efficiency becomes greater as
the found best policy improves. Indeed, the LOGO algorithm required
$10798$ seconds for ten policy evaluations and $52329$ seconds for
$48$ evaluations. The LOGO-OP algorithm required $9297$ seconds for
ten policy evaluations, and $44678$ seconds for $48$ evaluations. This
data in conjunction with Figure~\ref{fig9} illustrates the property of
the LOGO-OP algorithm that the policy evaluation becomes faster as the
found best policy improves. For the pLOGO-OP algorithm, the number of
function evaluations performed by the algorithm increased by a factor
of approximately eight (the number of parallel workers) compared to
the non-parallel versions. Notice that the parallel version tends to
allocate the extra resources to the global search (as opposed to the local
search). We can focus more on the local search by utilizing the previous
results of the policy evaluations; however, the parallel version must
initiate several policy evaluations without waiting for the previous
evaluations, resulting in a tendency for global search. This tendency
forced the improvement, in terms of reducing the amount of CsI, to be
moderate relative to the number of policy evaluations in this particular
experiment. However, such a tendency may have a more positive effect in
different problems where increased global search is beneficial. The CPU
time per policy evaluation varied significantly for different policies
owing to the different phenomenon computed in the simulator. On the
average, for the LOGO-OP algorithm, it took approximately $930$ seconds
per policy evaluation.

Now that we partially confirmed the validity of the pLOGO-OP algorithm,
we attempt to use it to provide meaningful information to this application
field. Based on the examination of the results in the above comparison,
we narrowed the range of the parameter values as $x_1 = [0, 1.2]$ and
$x_2 = [10330, 10910]$. After the computation with CPU time of $86400$
(s) and with eight workers for the parallelization, the pLOGO-OP
algorithm found the policy with $x_1\approx 0.195$ (g) and $x_2\approx
10880$ (kgf/m$^2$). With the policy determined, containment failure was
prevented and the total amount of CsI released into the atmosphere was
limited to approximately $0.5$ (g) (approximately $0.002\%$ of the total
CsI) in the $24$ hours after the initiation of the accident. This is a
major improvement because this scenario with our experimental setting
is considered to result in a containment failure or at best, in a large
amount of CsI release, more than $2000$ (g) (about $10\%$ of total CsI)
in our setting. The computational cost of CPU time of $86400$ (s) is
likely acceptable in the application field. In terms of computational
cost, we must consider two factors: the offline computation and the
variation of scenarios. The computational cost with CPU time of $86400$
(s) for a phenomenon that requires $86400$ (s) is not acceptable for
online computation (i.e., determining a satisfactory policy while the
accident is progressing). However, such computational cost is likely
acceptable if we consider preparing acceptable policies for various
scenarios in an offline manner (i.e., determining satisfactory polices
before the accident). Such information regarding these polices can be
utilized during an accident by first identifying the accident scenario
with heuristics or machine learning methods \cite{par10}. For
offline preparation, we must determine policies for various major
scenarios and thus if each computation takes, for example, one month,
it may not be acceptable.

\begin{figure}
\centering
\small\hair2pt
\labellist
\pinlabel \rotatebox{90}{Venting (-) / CsI (g)} [r] at -10 80
\pinlabel {Time along Accident Progression (s)} [t] at 230 -5
\endlabellist
\includegraphics[width=0.8\textwidth]{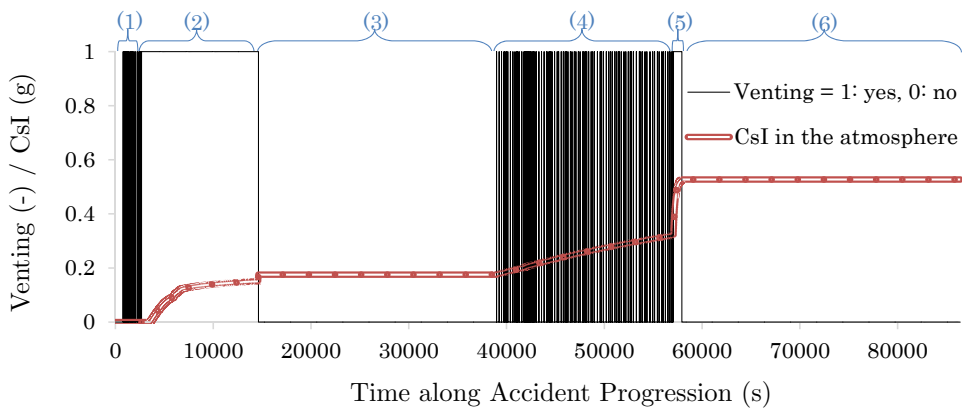}
\vskip5mm
\caption{Action sequence generated by found policy and CsI
release}\label{fig10}
\end{figure}

Note that the policy found by our method is both novel and nontrivial in
the literature, and yet worked very well. Accordingly, we explain why the
policy performed as well as it did. Figure~\ref{fig10} shows the action
sequence generated by the policy found and the amount of CsI (g) released
versus accident progression time (s). We analyze the action sequence
by dividing it into six phases, as indicated in Figure~\ref{fig10},
with the six numbers inside the parentheses. In the first phase (1),
the venting is conducted intermittently in order to keep the pressure
around $x_2\approx 10880$ (kgf/m$^2$). In this phase, no fission product
has yet been released from the nuclear fuels. Reducing the pressure and
the heat should be done preferably without releasing fission products,
and the actions in this phase accomplish this. One may wonder why the
venting should be done intermittently, instead of continuing to conduct
venting to reduce the pressure as much as possible, which can be done
without the release of fission products only in this phase. This is
because reducing the pressure too much leads to a large difference
between the pressures in the suppression chamber and the reactor pressure
vessel, which in turn results in a large mass flow and fission product
transportation from the reactor pressure vessel to the suppression chamber
(see Figure~\ref{fig11} in Appendix~A for information about
the mass flow paths). The increase in the amount of fission products in
the suppression chamber will likely result in a large release of fission
products into the atmosphere when venting is conducted. Therefore, this
specific value of $x_2$ that generates the intermittent venting works
well in the first phase. In the second phase (2), containment venting
is executed all the time since the pressure in the suppression chamber
increases rapidly in this phase (due to the operation of depressurizing
the reactor pressure vessel via the SRV line), and thus, the criterion
($\mathit{Press}\ge x_2$) in the policy is satisfied all the time from
this point. In the beginning of the third phase (3), the amount of CsI
in the suppression chamber exceeds $x_1\approx 0.195$ (g) and thereby
no venting is conducted. In the fourth phase (4), the pressure reaches
$100490$ (kgf/m$^2$) and containment venting is intermittently done
in order to keep the pressure under the point to avoid catastrophic
containment failure. In the fifth phase (5), the containment vent is
kept open because the amount of CsI in the gas phase of the suppression
chamber decreases to below $x_1$ (due to the phenomenon illustrated in
Figure~\ref{fig12} in Appendix~A). This continuous containment
venting decreases the pressure such that no venting is required in terms
of the pressure in the final phase (6), where venting is not conducted
also because the amount of CsI becomes larger than $x_1$.

Thus, it is clear that the policy found by this AI-related method also
has a basis in terms of physical phenomenon. In addition, the generated
action sequence is likely not simple enough for an engineer to discover
with several sensitivity analyses. In particular, not only did our method
solve the known tradeoff between the immediate CsI release and the risk of
future containment failure, the method also discovered the existence of
a new tradeoff between the immediate reduction of the pressure without
CsI release and future increase in the mass flow. Although there is
no consensus as to how to operate the containment venting system at the
moment, the tendency is to use it only when the pressure exceeds a certain
point in order to prevent immediate sever damage of containment vessel,
which corresponds only to the fourth phase (4) in Figure~\ref{fig10}. In
our experiment, such a myopic operation resulted in containment failure,
or a significantly large amount of CsI being released into the atmosphere
(at least more than $4800$ (g)). 

In summary, we successfully applied
the proposed method to investigate the containment venting policy in
nuclear power plant accidents. 
 As a preliminary
application study, several topics are left to future work. From a
theoretical viewpoint, future work should consider a way to mitigate the
simulation bias due to model error and model uncertainty. For the model
error, the robotics community is already cognizant that a small error in
a simulator can result in poor performance of the derived policy (i.e.,
simulation bias) \cite{kob13}. We can mitigate this problem by adding a
small noise to the model, since the noise works as \emph{regularization}
to prevent over-fitting as demonstrated by \citeA{atk98}. For the model
uncertainty, recent studies in the field of nuclear accident analysis
provide possible directions for the treatment of
uncertainty in accident phenomena \cite{zhe15} and accident scenarios \cite{kaw12}. As a result of
either or both of these countermeasures, the objective function becomes
stochastic, and thereby we may first expand the pLOGO-OP algorithm to
stochastic case. On the other hand, from the phenomenological point
of view, future work should consider other fission products as well as
CsI. Such fission products include, but are not limited to, Xe, Cs, I,
Te, Sr, and Ru. In particular, a noble gas element, such as Xe, can be a
major concern in an accident (it tends to be released a lot and is easily
diffused into the atmosphere), but its property is different from CsI (its
half-life is much smaller). Thus, if Xe is identified as a major concern,
one may consider a significantly different policy from ours (considering
its half-life, one may delay conducting the containment venting).

\section{Conclusions}\label{sec4}

In this paper, we proposed the LOGO algorithm, the global
optimization algorithm that is designed to operate well in practice
while maintaining a finite-loss bound with no strong additional
assumption. The analysis of the LOGO algorithm generalized previous
finite-loss bound analysis. Importantly, the analysis also provided several
insights regarding practical usage of this type of algorithm by showing
the relationship among the loss bound, the division strategy, and the
algorithm's parameters.  

We applied the LOGO algorithm to an AI planning problem with the policy
search framework, and showed that the performance of the algorithm can be
improved by leveraging not only the unknown smoothness in policy space,
but also the known smoothness in the planning horizon. As our study is
motivated to solve real-world engineering applications, we also discussed
a parallelization design that utilizes the property of AI planning in
order to minimize the overhead. The resulting algorithm, the pLOGO-OP
algorithm, was successfully applied to a complex engineering problem,
namely, policy derivation for nuclear accident management.

Aside from the planning problem that we considered, the LOGO algorithm can
be also used, for example, to optimize parameters of other algorithms (i.e., algorithm configuration). In
the AI community, the algorithm configuration problem has been addressed
by several methods, including a genetic algorithm  \cite{ans09}, discrete optimization with convergence guarantee
in the limit \cite{hut09}, the racing approach originated from the
machine learning community (Hoeffding Races) \cite{bir10}, model-based
optimization with convergence guarantee in the limit \cite{hut11},
a simultaneous use of several randomized local optimization methods
\cite{gyo11}, and Bayesian optimization \cite{sno12}. Compared to
the previous parameter tuning methods, the LOGO algorithm itself is
limited to optimizing continuous deterministic functions. To apply it
to stochastic functions, a future work would modify the LOGO algorithm
as was done for the SOO algorithm in a previous study \cite{val13}. To
consider categorical and/or discrete parameters in addition to continuous
parameters, a possibility could be to use the LOGO algorithm as a subroutine to deal with the continuous variables in one of the previous methods.

The promising results presented in this paper suggest several interesting
directions for future research.  An important direction is to leverage additional assumptions. Since LOGO is based on a weak set of assumptions, it would be natural to use LOGO as a main subroutine but add other  mechanisms to account for additional assumptions. As a example, we illustrated that LOGO would be able to  scale up for a higher dimension with additional assumptions in Section \ref{sec2.4}. Another possibility is to add a GP assumption  based on the idea presented in a recent  paper \cite{kawaguchiNIPS2015}.
Future work also would design an autonomous agent by integrating our planning algorithm with a learning/exploration algorithm \cite{kawaguchiAAAI2016}. One remaining challenge of LOGO is to
derive a series of methods that adaptively determine the algorithm's
free parameter $w$. As illustrated in our experiment, the achievement
in this topic not only mitigates the problem of parameter sensitivity,
but also would improve the algorithm's performance. 

\acks{This work was carried out while 
the first author was at the Japan Atomic Energy Agency. The authors
would like to thank Dr.~Hiroto Itoh and Mr.~Jun Ishikawa at JAEA for several discussions on the related topics. The authors
would like to thank Mr.~Lawson Wong at MIT for his insightful comments. The authors would like to thank anonymous reviewers for
their insightful and constructive comments. }

\appendix
\section*{Appendix A. Experimental Design of Application Study on Nuclear Accident
Management}
\refstepcounter{foo}
\label{appA}

\begin{figure}[t!]
\centering
\includegraphics[width=0.8\textwidth]{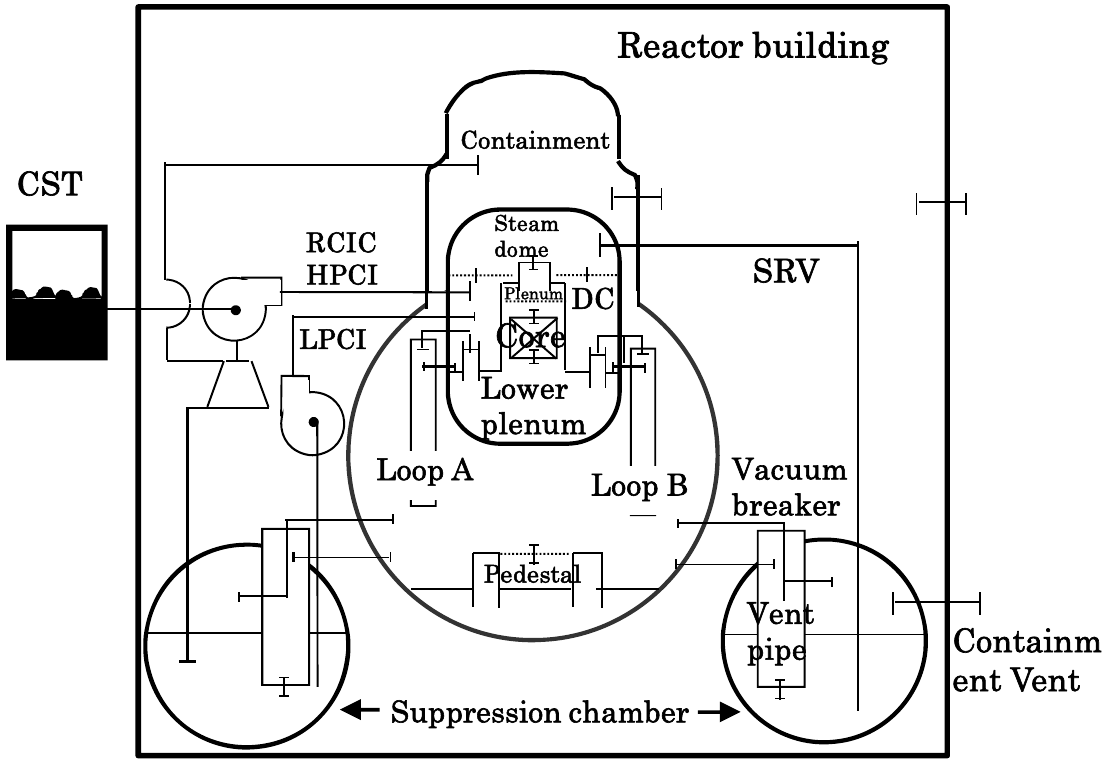}
\small
\caption{Nodalization of physical space}\label{fig11}
\end{figure}

\begin{figure}[b!]
\centering
\labellist
\footnotesize\hair2pt
\pinlabel {\textbf{\begin{scriptsize}FP: Fission Products\end{scriptsize}}} [tl] at 50 -4
\endlabellist
\includegraphics[width=0.85\textwidth]{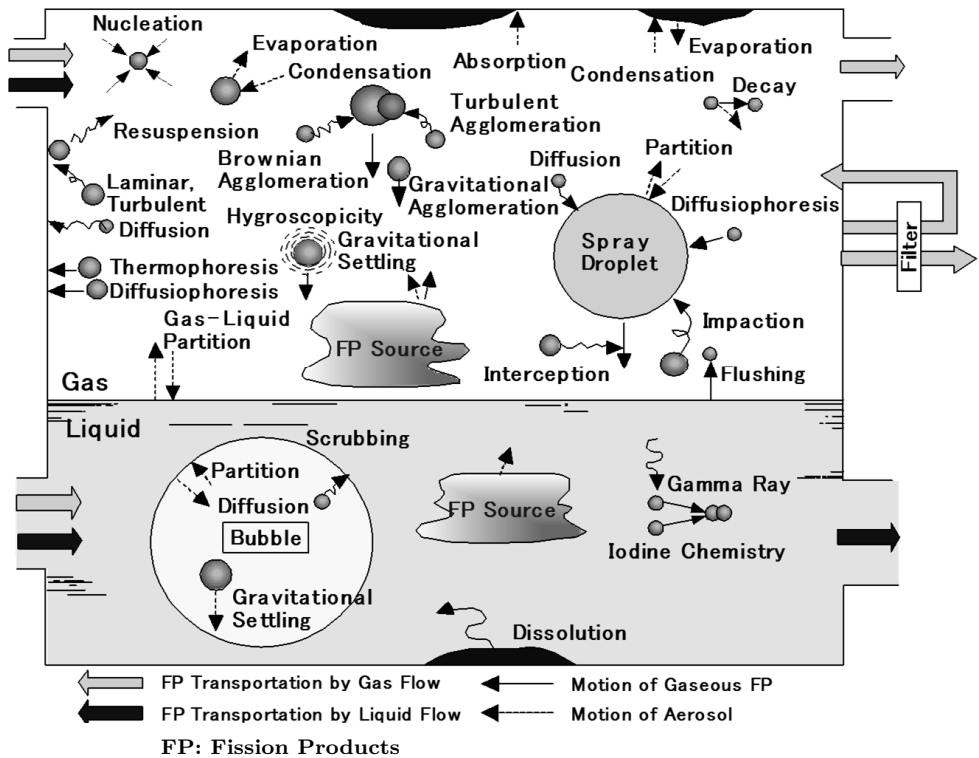}
\vspace{10pt}
\caption{Phenomenon considered for fission product
transportation}\label{fig12}
\end{figure}

In this appendix, we present the experimental setting for
the \emph{Application Study on Nuclear Accident Management} in
Section~\ref{sec3.4}. With THALES2, we consider the volume nodalization
as shown in Figure~\ref{fig11}. The reactor pressure vessel was divided
into seven volumes, consisting of core, upper plenum, lower plenum, steam
dome, downcomer, and recirculation loops A and B. The containment vessel
consists of drywell, suppression chamber, pedestal and vent pipes. The
atmosphere and suppression chamber are connected via the containment
venting system (S/C venting). The plant data and initial conditions were
determined based on the data of the Unit 1 of the Browns Ferry nuclear
power plant (BWR4/Mark-I) and \emph{the construction permit application
forms} of BWR plants in Japan. The failure of the containment vessel is
assumed to occur when the pressure of the vessel becomes $2.5$ times
greater than the design pressure. Here, the design pressure is $3.92$
(kgf/cm$^2$g) and the criterion for the containment failure is 
$108330$\,kgf/m$^2$. The degree of opening for the containment venting was fixed
at $25\%$ and no filtering was considered.

We consider the TQUV sequence as the accident scenario. In the TQUV
sequence, no Emergency Core Cooling Systems (ECCSs) functions, similar
to the case of the accident at the Fukushima Daiichi nuclear power
plant. The TQUV sequence is one of the major scenarios considered in
Probabilistic Risk Assessment (PRA) of nuclear plants. Therefore, our
results show a promising benefit of containment venting, as long as we
use it with a good policy.

The simulator we developed at the Japan Atomic Energy Agency and
adopted in this experiment (THALES2) computes the transportation
of fission products as well as thermal hydraulics in each volume
of Figure~\ref{fig11} and core melt progression in the \emph{Core}
volume. The transportation of fission products considered in the
experiment is shown in Figure~\ref{fig12}. The details of the computation
of the THALES2 code are found in the paper by \citeauthor{ish02} \citeyear{ish02}. Figure~\ref{fig11} and
Figure~\ref{fig12} are the modified versions of the graphs used in a
previous presentation about the ongoing development of the THALES2 code,
which was given at the USNRC's 25th Regulatory Information Conference
\cite{mar13}.

\vskip 0.2in
\bibliography{refs}
\bibliographystyle{theapa}

\end{document}